\documentclass[12pt,english]{extarticle}
\usepackage[osf]{mathpazo}

\usepackage[T1]{fontenc}
\usepackage[latin9]{inputenc}
\usepackage[paperwidth=30cm,paperheight=35cm]{geometry}
\usepackage{amsmath}
\usepackage{amssymb}

\makeatletter

\providecommand{\tabularnewline}{\\}

\@ifundefined{date}{}{\date{}}
\usepackage{tikz}
\usetikzlibrary{matrix,arrows,decorations.pathmorphing}
\usetikzlibrary{shapes.geometric}
\usepackage{tikz-cd}
\usepackage{amsthm}
\usepackage{xparse,etoolbox}

\theoremstyle{plain}
\newtheorem{theorem}{Theorem}[section]
\newtheorem{lemma}[theorem]{Lemma}
\newtheorem{prop}{Proposition}[section]
\newtheorem{cor}{Corollary}
\newtheorem{conj}{Conjecture}
\theoremstyle{definition}
\newtheorem{defn}{Definition}[section]

\newtheorem{example}{Example}[section]
\theoremstyle{remark}
\newtheorem{rem}{Remark}

\usepackage{graphicx}
\usepackage{amssymb}
\usepackage{tikz-cd}
\usetikzlibrary{calc,arrows,decorations.pathreplacing}
\tikzset{mydot/.style={circle,fill,inner sep=1.5pt},
commutative diagrams/.cd,
  arrow style=tikz,
  diagrams={>={Straight Barb[scale=0.8]}},
}

\usepackage{babel}
\usepackage{hyperref}
\hypersetup{
    colorlinks,
    citecolor=blue,
    filecolor=blue,
    linkcolor=blue,
    urlcolor=blue
}
\usepackage{pgfplots}
\usetikzlibrary{decorations.markings}
\pgfplotsset{compat=1.9}

\newcommand{\blocktheorem}[1]{%
  \csletcs{old#1}{#1}% Store \begin
  \csletcs{endold#1}{end#1}% Store \end
  \RenewDocumentEnvironment{#1}{o}
    {\par\addvspace{1.5ex}
     \noindent\begin{minipage}{\textwidth}
     \IfNoValueTF{##1}
       {\csuse{old#1}}
       {\csuse{old#1}[##1]}}
    {\csuse{endold#1}
     \end{minipage}
     \par\addvspace{1.5ex}}
}

\blocktheorem{theorem}% Make theo into a block
\blocktheorem{defn}% Make defi into a block
\blocktheorem{lemma}% Make lem into a block
\blocktheorem{rem}% Make rem into a block
\blocktheorem{cor}% Make col into a block
\blocktheorem{prop}% Make prop into a block

\makeatletter
\newcommand*{\@old@slash}{}\let\@old@slash\slash
\def\slash{\relax\ifmmode\delimiter"502F30E\mathopen{}\else\@old@slash\fi}
\makeatother

\def\backslash{\delimiter"526E30F\mathopen{}}

\usepackage[bottom]{footmisc}

\makeatother

\usepackage{babel}
\usepackage{listings}

\begin{document}
\title{Homological Associativity of Differential Graded Algebras and Gröbner
Bases}
\author{Michael Nelson}
\maketitle
\begin{abstract}
We investigate associativity of multiplications on chain complexes
over commutative noetherian rings from two perspectives. First, we
introduce a natural associator subcomplex and show how its homology
can detect associativity. Second, we use Gröbner bases to compute
associators.
\end{abstract}

\section{Introduction}

In this paper, we study algebraic structures that we can attach to
free resolutions. Our motivation is the following: let $(R,\mathfrak{m},\Bbbk)$
be a local (or standard graded) commutative noetherian ring, let $I\subseteq\mathfrak{m}$
be an ideal of $R$, and let $F=(F,\mathrm{d})$ be the minimal free
resolution of $R\slash I$ over $R$. The usual multiplication map
$R\slash I\otimes_{R}R\slash I\to R\slash I$ can be lifted to a chain
map $\mu\colon F\otimes_{R}F\to F$ defined by $a_{1}\otimes a_{2}\mapsto a_{1}\star_{\mu}a_{2}$
where $a_{1},a_{2}\in F$ (we simplify notation to $a_{1}\star_{\mu}a_{2}=a_{1}a_{2}$
whenever $\mu$ is clear from context). Furthermore, we can choose
$\mu$ to be unital (with $1\in F_{0}=R$ being the identify element)
and strictly graded-commutative; see Definition~(\ref{defnpropertiestobesatisfied}).
In this case we call $\mu$ a multiplication on $F$, and when we
equip $F$ with this multiplication, we say $F$ is an MDG algebra
(the ``M'' in ``MDG'' stands for multiplication which we always
require to be strictly graded-commutative and unital though not necessarily
associative). It was first shown that $F$ always possesses an MDG
algebra structure by Buchsbaum and Eisenbud in \cite{BE77}, and in
that paper they posed the following question:

\hfill

\textbf{Question 1.1: }Does $F$ possess the structure of a DG algebra?
In other words, can $\mu$ be chosen such that it is also associative?

\hfill

~~~~~~One reason this question is interesting is that when we
know the answer is ``yes'', then we gain a lot of information about
the shape of $F$. For instance, Buchsbaum and Eisenbud proved that
if we further assume $R$ is a domain and we know that an associative
multiplication on $F$ exists, then one obtains important lower bounds
of the Betti numbers $\beta_{i}=\beta_{i}^{R}(R\slash I)$. In particular,
let $\boldsymbol{t}=t_{1},\dots,t_{g}$ be a maximal $R$-sequence
contained in $I$ and let $E$ be the Koszul algebra which resolves
$R\slash\boldsymbol{t}$ over $R$. Any expression of the $t_{i}$
in terms of the generators for $I$ yields a canonical comparison
map $E\to F$. Buchsbaum and Eisenbud showed that under these assumptions,
this comparison map $E\to F$ is injective, hence we get the lower
bound $\beta_{i}\geq{g \choose i}$ for each $i\leq g$. Unfortunately,
it turns out that the answer to Question 1.1 is that $F$ need not
have a DG algebra structure on it (see \cite{Avr81,Kat19,Sri92} for
counterexamples), so Buchsbaum and Eisenbud's proof of these lower
bounds would fail in these cases. Nonetheless, these lower bounds
are still conjectured to hold. It is known as the (local) Buchsbaum-Eisenbud-Horrocks
(BEH) conjecture (see \cite{Erm10,VW23,Wal17} for more on this topic): 

\hfill

\begin{conj}\label{behconjecture} (BEH Conjecture). Let $M$ be a
nonzero $R$-module of finite projective dimension. Then we have
\[
\beta_{i}(M)\geq{\mathrm{codim}\,M \choose i}
\]
for all $i,$ where $\beta_{i}(M)$ is the $i$th Betti number of
$M$ and where $\mathrm{codim}\,M=\mathrm{height}(\mathrm{Ann}\,M)$.
\end{conj}

\hfill

~~~~~~One of the starting points for this paper is based on
the observation that by slightly modifying Buchsbaum and Eisenbud's
proof one can still obtain these lower bounds even in cases where
it is known that we cannot choose $\mu$ to be associative. Indeed,
we just need to find a multiplication $\mu$ on $F$ together with
a comparison map $\varphi\colon E\to F$ such that $\varphi\colon E\to F$
is multiplicative, meaning
\[
\varphi(a_{1}a_{2})=\varphi(a_{1})\varphi(a_{2})
\]
for all $a_{1},a_{2}\in E$. The proof given by Buchsbaum and Eisenbud
which shows $\varphi\colon E\to F$ is injective would still apply
in this case. Furthermore, in their proof, Buchsbaum and Eisenbud
used a property that the Koszul algebra $E$ satisfies, namely that
every nonzero DG ideal of $E$ intersects the top degree $E_{g}$
non-trivially. However there are many other MDG algebras which satisfy
this property as well (the property being that every nonzero MDG ideal
intersect the top degree non-trivially). In particular, Taylor algebras
satisfy this property. Thus one can generalize this further by replacing
$\boldsymbol{t}$ with an ideal $J$ such that $\boldsymbol{t}\subseteq J\subseteq I$
and such that there exists a multiplication on the minimal free resolution
$G$ of $R\slash J$ over $R$ which satisfies this property. To see
that we really do gain a new perspective here, we consider Example~(\ref{examplemultiplicative})
where it is known that we cannot choose an associative multiplication
$\mu$ on $F$ yet we can find a multiplicative map $T\to F$ where
$T$ is a Taylor algebra resolution. In general, we would like to
choose a multiplication which is as associative as possible. To this
end, we pose the following question:

\hfill

\textbf{Question 1.2: }Equip $F$ with a multiplication $\mu$ giving
it the structure of an MDG algebra. How can we measure the failure
of $F$ to being associative? 

\hfill

~~~~~~We answer this question 1.2 in by studying the maximal
associative quotient of $F$. In short, in Subsection 3.1, we define
the\textbf{ }associator\textbf{ }submodule of an MDG module $X$ over
an MDG algebra $A$ to be the smallest MDG $A$-submodule containing
all ``associators'' of $X$:
\[
\langle X\rangle=\langle\{(a_{1}a_{2})x-a_{1}(a_{2}x)\mid a_{1},a_{2}\in A\text{ and }x\in X\}\rangle\subseteq X.
\]
It is clear that if $X$ is associative, then $\mathrm{H}(\langle X\rangle)=0$.
The first main result of this paper Theorem~(\ref{theoremhomologyassociator})
shows that the converse holds under certain conditions. The second
main result of this paper is Theorem~(\ref{theoremsecond}) where
we show that every multiplication on a resolution considered in \cite{BE77}
will be non-associative at a particular triple. Note that Avramov
had already shown that no associative multiplication on this resolution
exists, however it seems somewhat surprising that one cannot choose
a multiplication which can be made associative at one triple at the
possible expense of being non-associative at some other triple. The
technique we used in proving this made use of a particularly nice
MDG algebra which is described in Example~(\ref{example2}). In Subsection
4.1, we exploit a criterion for exactness. We apply this criterion
in our third main result, Theorem~(\ref{theorem3.1}) to demonstrate
associativity of exterior extensions. In the final section of this
paper, we construct the symmetric DG algebra of an $R$-complex $A$
which is centered at $R$ (meaning $A_{0}=R$ and $A_{i}=0$ for all
$i<0$), denoted by $\mathrm{S}_{R}(A)=S$. This section contains
our fourth result of the paper, namely Theorem~(\ref{theorempresentationsym}),
which says that if we fix a multiplication $\mu$ on $A$, then the
quotient $A^{\mathrm{as}}:=A\slash\langle A\rangle$ can be presented
as a quotient of $S$ by a DG $S$-ideal $\mathfrak{s}=\mathfrak{s}(\mu)$
which is constructed from $\mu$ in a natural way. In particular,
we can study MDG algebra structures on $A$ by studying certain DG
ideals of $S$. This presentation allows us to use Gröbner bases to
help calculate $A^{\mathrm{as}}$ when working over an integral domain
where we can see how associators naturally arise when performing Buchberger's
algorithm to certain set of polynomials with respect to this monomial
ordering.

\hfill

~~~~~~This paper is organized into five sections, the first
section being this introduction. In the second section, we work over
an arbitrary commutative ring $R$ and we define the category of MDG
$R$-algebras as well as the category of modules over them. Briefly,
an MDG $R$-algebra $A$ is essentially just a DG $R$-algebra except
we don't require the associative rule to hold. Similarly, an MDG $A$-module
$X$ is essentially just a DG $A$-module except we do not require
the associative rule to hold. In the third section, we introduce tools
which help us measure how far away MDG objects are from being DG objects.
In particular, we define the associator\textbf{ }of $X$ to be the
chain map $[\cdot]\colon A\otimes A\otimes X\to X$ defined on elementary
tensors by
\[
[a_{1}\otimes a_{2}\otimes x]=(a_{1}a_{2})x-a_{1}(a_{2}x)=[a_{1},a_{2},x]
\]
for all $a_{1},a_{2}\in A$ and $x\in X$, where we denote by $[\cdot,\cdot,\cdot]\colon A\times A\times X\to X$
to be the unique map corresponding to $[\cdot]$ via the universal
mapping property of tensor products. We set $\langle X\rangle$ to
be the smallest MDG $A$-submodule of $X$ which contains the image
of the associator of $X$. The quotient $X^{\mathrm{as}}:=X\slash\langle X\rangle$
is called the maximal associative quotient\textbf{ }of $X$; it plays
a role analogous to the role of the maximal abelian quotient of a
group. We study the homology of $\langle X\rangle$ as well as the
homology of $X^{\mathrm{as}}$. In this section we also define and
study the multiplicator\textbf{ }of a chain map $\varphi\colon X\to Y$,
where $X$ and $Y$ are MDG $A$-modules. This is the chain map $[\cdot]_{\varphi}\colon A\otimes X\to Y$
defined on elementary tensors by
\[
[a\otimes x]_{\varphi}=\varphi(ax)-a\varphi(x)=[a,x]
\]
for all $a\in A$ and $x\in X$, where we denote by $[\cdot,\cdot]\colon A\times X\to Y$
to be the unique map corresponding to $[\cdot]_{\varphi}$ via the
universal mapping property of tensor products. In the fourth section,
we turn our attention towards the associator functor which takes an
MDG $A$-module $X$ to the MDG $A$-module $\langle X\rangle$ and
takes an MDG $A$-module homomorphism $\varphi\colon X\to Y$ to the
restriction map $\varphi\colon\langle X\rangle\to\langle Y\rangle$.
Under certain conditions, a short exact sequence \begin{equation}\label{diagram2342d3}\begin{tikzcd} 0 \arrow[r] & X \arrow[r,"\varphi "]  & Y \arrow[r," \psi "] & Z  \arrow[r] & 0 \end{tikzcd}\end{equation}
of MDG $A$-modules induces a long exact sequence in homology: \begin{equation}\label{diagram3}\begin{tikzcd}[row sep=40]  && \cdots \arrow[r] \arrow[d, phantom, ""{coordinate, name=Z'}] & \mathrm{H}_{i+1} \langle Z   \rangle \arrow[dll,  swap, rounded corners, to path={ -- ([xshift=2ex]\tikztostart.east) |- (Z') [near end]\tikztonodes -| ([xshift=-2ex]\tikztotarget.west) -- (\tikztotarget)}] \\  & \mathrm{H}_{i} \langle X \rangle \arrow[r] & \mathrm{H}_{i} \langle Y \rangle \arrow[r] \arrow[d, phantom, ""{coordinate, name=Z}] & \mathrm{H}_{i} \langle Z \rangle \arrow[dll,  swap, rounded corners, to path={ -- ([xshift=2ex]\tikztostart.east) |- (Z) [near end]\tikztonodes -| ([xshift=-2ex]\tikztotarget.west) -- (\tikztotarget)}] \\ & \mathrm{H}_{i-1} \langle X \rangle \arrow[r] & \cdots \end{tikzcd}\end{equation}
We end this section with an application of this long exact sequence
to certain exterior extensions. In a future paper, we would like to
assign a finite number to a multiplication $\mu$ on a minimal free
resolution $F$ of a cyclic $R$-module over $R$ where $R$ is a
local noetherian ring. This quantity should measure the failure for
$\mu$ to being associative. We believe studying such exterior extensions
will help us to move closer towards that goal.

\hfill

~~~~~~In the final section of this paper, we construct the symmetric
DG algebra of an $R$-complex $A$ which is centered at $R$ (meaning
$A_{0}=R$ and $A_{i}=0$ for all $i<0$), denoted by $\mathrm{S}_{R}(A)=S$.We
will show that if we fix a multiplication $\mu$ on $A$, then the
maximal associative quotient of $A$ can be presented as a quotient
of $S$ by a DG $S$-ideal $\mathfrak{s}=\mathfrak{s}(\mu)$ which
is constructed from $\mu$ in a completely natural way. This presentation
also has interesting Gröbner basis applications in the case where
$R$ is a domain with fraction field $K$ and $F$ is an MDG $R$-algebra
centered at $R$ such that the underlying graded $R$-module of $F$
is finite and free. Indeed, suppose that
\[
F_{+}=Re_{1}+\cdots+Re_{n}
\]
where $e_{1},\dots,e_{n}$ is an ordered homogeneous basis of $F_{+}$
which is ordered in such a way that if $|e_{i'}|>|e_{i}|$, then $i'>i$,
and let $R[\boldsymbol{e}]=R[e_{1},\dots,e_{n}]$ be the free non-strict
graded-commutative $R$-algebra generated by $e_{1},\dots,e_{n}$.
We will equip $K[\boldsymbol{e}]:=K\otimes_{R}R[\boldsymbol{e}]$
with a specific monomial ordering and show how associators naturally
arise when performing Buchberger's algorithm to a certain set of polynomials
with respect to this monomial ordering. We further demonstrate in
Example~(\ref{example455}) how, with the help of a computer algebra
system like Singular, this monomial ordering can help us find associative
multiplications on minimal free resolutions. For instance, we used
Singular to find an associative multiplication on the minimal free
resolution in Example~(\ref{example4}).

\section{MDG Algebras and MDG Modules}

In this section, we define MDG algebras and MDG modules over them.
We also discuss several examples of them in a multigraded setting.

\subsection{MDG Algebras}

Let $R$ be a commutative ring and let $A=(A,\mathrm{d})$ be an $R$-complex.
We further equip $A$ with a chain map $\mu\colon A\otimes_{R}A\to A$.
We denote by $\star_{\mu}\colon A\times A\to A$ (or more simply by
$\cdot$ if context is clear) to be the unique graded $R$-bilinear
map which corresponds to $\mu$ via the universal mapping property
of tensors products. Thus we have
\[
\mu(a_{1}\otimes a_{2})=a_{1}\star_{\mu}a_{2}=a_{1}a_{2}
\]
for all $a_{1},a_{2}\in A$, where we further simplify the notation
by writing $a_{1}\star_{\mu}a_{2}=a_{1}a_{2}$ when context is clear.
In order to simplify our notation in what follows, we often refer
to the triple $(A,\mathrm{d},\mu)$ via its underlying graded $R$-module
$A$, where we think of $A$ as a graded $R$-module which is equipped
with a differential $\mathrm{d}\colon A\to A$, giving it the structure
of an $R$-complex, and which is further equipped with a chain map
$\mu\colon A\otimes_{R}A\to A$. For instance, if $\mu$ satisfies
a property (such as being associative), then we also say $A$ satisfies
that property. 

\begin{defn}\label{defnpropertiestobesatisfied} With the notation
as above, we make the following definitions:
\begin{enumerate}
\item We say $A$ is \textbf{unital} if there exists $1\in A$ such that
$1a=a=a1$ for all $a\in A$. 
\item We say $A$ is \textbf{graded-commutative }if $a_{1}a_{2}=(-1)^{|a_{1}||a_{2}|}a_{2}a_{1}$
for all homogeneous $a_{1},a_{2}\in A$. 
\item We say $A$ is \textbf{strictly graded-commutative }if it is graded-commutative
and satisfies the additional property that $a^{2}=0$ for all elements
$a\in A$ with $|a|$ odd.
\item We say $A$ is \textbf{associative} if $(a_{1}a_{2})a_{3}=a_{1}(a_{2}a_{3})$
for all for all $a_{1},a_{2},a_{3}\in A$. 
\end{enumerate}
We say $A$ is an \textbf{MDG $R$-algebra} if $A$ is unital and
strictly graded-commutative (thought not necessarily associative)
and in this case we call $\mu$ the \textbf{multiplication }of $A$
(just as we call $\mathrm{d}$ the differential of $A$). Suppose
$B$ is another MDG $R$-algebra and let $\varphi\colon A\to B$ be
a chain map.
\begin{enumerate}
\item We say $\varphi$ is \textbf{unital }if $\varphi(1)=1$.
\item We say $\varphi$ is \textbf{multiplicative }if $\varphi(a_{1}a_{2})=\varphi(a_{1})\varphi(a_{2})$
for all $a_{1},a_{2}\in A$.
\end{enumerate}
We say $\varphi\colon A\to B$ is an \textbf{MDG $R$-algebra homomorphism
}(or more simply just homomorphism if context is clear) if it is both
unital and multiplicative. \end{defn}

\begin{rem}\label{rem} Note in the literature, an \emph{associative
}MDG $R$-algebra is often called a DG $R$-algebra, thus an MDG $R$-algebra
is essentially just a ``not necessarily associative'' DG $R$-algebra.
\end{rem}

\subsubsection{Examples of Multigraded MDG Algebras}

In this subsubsection, we consider six examples of multigraded MDG
algebras. The first two examples were considered in \cite{Kat19}
and \cite{Avr81} respectively and were both shown to be examples
of minimal free resolutions which do not admit DG algebra structures
on them.

\begin{example}\label{example1} Let $R=\Bbbk[x,y,z,w]$, let $\boldsymbol{m}_{\mathrm{K}}=\boldsymbol{m}=x^{2},w^{2},xy,zw,y^{2}z^{2}$
and let $F_{\mathrm{K}}=F$ be the minimal free resolution of $R\slash\boldsymbol{m}$
over $R$. One can visualize $F$ as being supported on the $\boldsymbol{m}$-labeled
simplicial complex below:\begin{center}
\begin{tabular}{cccc}
\begin{tikzpicture}[scale=1]
\draw[fill=gray!20] (0,0) -- (3,-0.5) -- (3.2,1.2)-- (0,0); 
\draw[fill=gray!20] (0,0) -- (1.5,1.5) -- (3.2,1.2)-- (0,0); \draw[] (1.5,1.5) -- (3,-0.5); 
\draw[fill=gray!20] (3,-0.5) -- (5,0.2) -- (3.2,1.2); 
\draw[color=black!100] (1.5,1.5) -- (3,-0.5) -- (5,0.2);
\draw[color=black!100] (0,0) -- (3.2,1.2) -- (5,0.2);

\node[circle, fill=black, inner sep=1pt, label=left:$e_1 $] (a) at (0,0) {}; 
\node[circle, fill=black, inner sep=1pt, label=above:$e_2 $] (b) at (1.5,1.5) {}; 
\node[circle, fill=black, inner sep=1pt, label=below:$e_3 $] (c) at (3,-0.5) {};
\node[circle, fill=black, inner sep=1pt, label=above:$e_4 $] (d) at (3.2,1.2) {}; 
\node[circle, fill=black, inner sep=1pt, label=right:$e_5 $] (e) at (5,0.2) {};

\node[] (x) at (-1,1.3) {$ e_{12} $}; \node[] (x') at (0.8,0.6) {};
\draw[-{Straight Barb[length=3pt,width=3pt]}] (x) edge[out=0, in=130] node[below] {} (x');

\node[] (y) at (6.5,1.3) {$ e_{45} $}; \node[] (y') at (4.1,0.5) {}; 
\draw[-{Straight Barb[length=3pt,width=3pt]}] (y) edge[out=180, in=45] node[below] {} (y');

\node[] (z) at (6.5,-0.7) {$e_{35} $}; \node[] (z') at (4.1,0.05) {};
\draw[-{Straight Barb[length=3pt,width=3pt]}] (z) edge[out=180, in=-45] node[below] {} (z');

\node[] (w) at (4.5,2) {$e_{345} $}; \node[] (w') at (3.8,0) {};
\draw[-{Straight Barb[length=3pt,width=3pt]}] (w) edge[out=-120, in=90] node[below] {} (w');
\end{tikzpicture} &  &  & \begin{tikzpicture}[scale=1]

\draw[fill=gray!20] (0,0) -- (3,-0.5) -- (3.2,1.2)-- (0,0); 
\draw[fill=gray!20] (0,0) -- (1.5,1.5) -- (3.2,1.2)-- (0,0); 
\draw[] (1.5,1.5) -- (3,-0.5); \draw[fill=gray!20] (3,-0.5) -- (5,0.2) -- (3.2,1.2);
\draw[color=black!100] (1.5,1.5) -- (3,-0.5) -- (5,0.2);
\draw[color=black!100] (0,0) -- (3.2,1.2) -- (5,0.2);

\node[circle, fill=black, inner sep=1pt, label=left:$x^2 $] (a) at (0,0) {};
\node[circle, fill=black, inner sep=1pt, label=above:$w^2 $] (b) at (1.5,1.5) {};
\node[circle, fill=black, inner sep=1pt, label=below:$zw $] (c) at (3,-0.5) {};
\node[circle, fill=black, inner sep=1pt, label=above:$xy $] (d) at (3.2,1.2) {}; 
\node[circle, fill=black, inner sep=1pt, label=right:$y^2 z^2 $] (e) at (5,0.2) {};

\node[] (x) at (-1,1.3) {$ x^2 w^2 $}; \node[] (x') at (0.8,0.6) {};
\draw[-{Straight Barb[length=3pt,width=3pt]}] (x) edge[out=0, in=130] node[below] {} (x');

\node[] (y) at (6.5,1.3) {$ x y^2 z^2 $}; \node[] (y') at (4.1,0.5) {}; 
\draw[-{Straight Barb[length=3pt,width=3pt]}] (y) edge[out=180, in=45] node[below] {} (y');

\node[] (z) at (6.5,-0.7) {$ y^2 z^2 w $}; \node[] (z') at (4.1,0.05) {};
\draw[-{Straight Barb[length=3pt,width=3pt]}] (z) edge[out=180, in=-45] node[below] {} (z');

\node[] (w) at (4.5,2) {$ x y^2 z^2 w $}; \node[] (w') at (3.8,0) {}; 
\draw[-{Straight Barb[length=3pt,width=3pt]}] (w) edge[out=-120, in=90] node[below] {} (w');

\end{tikzpicture}\tabularnewline
\end{tabular}
\par\end{center}In particular, the homogeneous components of $F$ as a graded $R$-module
are given by
\begin{align*}
F_{0} & =R\\
F_{1} & =Re_{1}+Re_{2}+Re_{3}+Re_{4}+Re_{5}\\
F_{2} & =Re_{12}+Re_{13}+Re_{14}+Re_{23}+Re_{24}+Re_{34}+Re_{35}+Re_{45}\\
F_{3} & =Re_{123}+Re_{124}+Re_{134}+Re_{234}+Re_{345}\\
F_{4} & =Re_{1234},
\end{align*}
and the differential $\mathrm{d}$ of $F$ behaves just like the usual
boundary map of the simplicial complex above except some monomials
can show up as coefficients (so that the differential respects the
multidegree). For instance, we have
\[
\mathrm{d}(e_{1234})=-ye_{123}+ze_{124}-we_{134}+xe_{234}.
\]
For more details on this construction, see \cite{BPS98}. We now wish
to equip $F$ with a multigraded multiplication $\mu_{\mathrm{K}}=\mu$
giving it the structure of a multigraded MDG algebra. Since $\mu$
respects the multigrading and satisfies Leibniz rule, we are forced
to have:
\begin{alignat*}{3}
e_{1}\star e_{5} & =yz^{2}e_{14}+xe_{45} &  &  & e_{2}\star e_{45} & =-yze_{234}+we_{345}\\
e_{1}\star e_{2} & =e_{12} &  &  & e_{1}\star e_{35} & =yze_{134}-xe_{345}\\
e_{2}\star e_{5} & =y^{2}ze_{23}+we_{35} &  &  & e_{1}\star e_{23} & =e_{123}\\
 &  & \qquad\qquad &  & e_{2}\star e_{14} & =-e_{124}
\end{alignat*}
At this point however, one can conclude that $F$ is not associative
since
\begin{equation}
[e_{1},e_{5},e_{2}]:=(e_{1}\star e_{5})\star e_{2}-e_{1}\star(e_{5}\star e_{2})=-yz\mathrm{d}(e_{1234})\neq0.\label{eq:associatorsingular-1-1-1}
\end{equation}

The multiplication is not uniquely determined on all pairs $(e_{\sigma},e_{\tau})$;
for instance there are two possible ways in which $\mu$ is defined
at the pair $(e_{5},e_{12})$. We assume that $\mu$ is defined at
$(e_{5},e_{12})$ by
\[
e_{5}\star e_{12}=yz^{2}e_{124}+xyze_{234}+xwe_{345}.
\]
Finally, we would still like for $\mu$ to be as associative as possible
even though we already know it is not associative at the triple $(e_{1},e_{5},e_{2})$.
In particular, we want $\mu$ to be associative on all triples of
the form $(e_{\sigma},e_{\sigma},e_{\tau})$. It turns out this can
be done and we will assume that $\mu$ is associative on all such
triples. \end{example}

\begin{example}\label{example2} Let $R=\Bbbk[x,y,z,w]$, let $\boldsymbol{m}_{\mathrm{A}}=\boldsymbol{m}=x^{2},w^{2},zw,xy,yz$
and let $F_{\mathrm{A}}=F$ be the minimal free resolution of $R\slash\boldsymbol{m}$
over $R$. One can visualize $F$ as being supported on the $\boldsymbol{m}$-labeled
cellular complex below: \begin{center}
\begin{tabular}{cccc}
\begin{tikzpicture}[scale=1]

\draw[fill=gray!20] (0,0) -- (3,-0.5) -- (5,0.2) -- (3.2,1.2)-- (0,0); 
\draw[fill=gray!20] (0,0) -- (1.5,1.5) -- (3.2,1.2)-- (0,0); 
\draw[] (1.5,1.5) -- (3,-0.5); \draw[fill=gray!20] (3,-0.5) -- (5,0.2) -- (3.2,1.2);
\draw[color=black!100] (1.5,1.5) -- (3,-0.5) -- (5,0.2);
\draw[color=black!100] (0,0) -- (3.2,1.2) -- (5,0.2);

\node[circle, fill=black, inner sep=1pt, label=left:$\varepsilon _1 $] (a) at (0,0) {};
\node[circle, fill=black, inner sep=1pt, label=above:$\varepsilon _2 $] (b) at (1.5,1.5) {};
\node[circle, fill=black, inner sep=1pt, label=below:$\varepsilon _3 $] (c) at (3,-0.5) {};
\node[circle, fill=black, inner sep=1pt, label=above:$\varepsilon _4 $] (d) at (3.2,1.2) {}; 
\node[circle, fill=black, inner sep=1pt, label=right:$\varepsilon _5 $] (e) at (5,0.2) {};

\node[] (x) at (-1,1.3) {$ \varepsilon _{12} $}; \node[] (x') at (0.8,0.6) {};
\draw[-{Straight Barb[length=3pt,width=3pt]}] (x) edge[out=0, in=130] node[below] {} (x');

\node[] (y) at (6.5,1.3) {$ \varepsilon _{45} $}; \node[] (y') at (4.1,0.5) {}; 
\draw[-{Straight Barb[length=3pt,width=3pt]}] (y) edge[out=180, in=45] node[below] {} (y');

\node[] (z) at (6.5,-0.7) {$ \varepsilon _{35} $}; \node[] (z') at (4.1,0.05) {};
\draw[-{Straight Barb[length=3pt,width=3pt]}] (z) edge[out=180, in=-45] node[below] {} (z');

\node[] (w) at (4.5,2) {$ \varepsilon _{2345} $}; \node[] (w') at (3.2,0) {}; 
\draw[-{Straight Barb[length=3pt,width=3pt]}] (w) edge[out=-120, in=90] node[below] {} (w');

\end{tikzpicture} &  &  & \begin{tikzpicture}[scale=1]

\draw[fill=gray!20] (0,0) -- (3,-0.5) -- (5,0.2) -- (3.2,1.2)-- (0,0); 
\draw[fill=gray!20] (0,0) -- (1.5,1.5) -- (3.2,1.2)-- (0,0); 
\draw[] (1.5,1.5) -- (3,-0.5); \draw[fill=gray!20] (3,-0.5) -- (5,0.2) -- (3.2,1.2);
\draw[color=black!100] (1.5,1.5) -- (3,-0.5) -- (5,0.2);
\draw[color=black!100] (0,0) -- (3.2,1.2) -- (5,0.2);

\node[circle, fill=black, inner sep=1pt, label=left:$x^2 $] (a) at (0,0) {};
\node[circle, fill=black, inner sep=1pt, label=above:$w^2 $] (b) at (1.5,1.5) {};
\node[circle, fill=black, inner sep=1pt, label=below:$zw $] (c) at (3,-0.5) {};
\node[circle, fill=black, inner sep=1pt, label=above:$xy $] (d) at (3.2,1.2) {}; 
\node[circle, fill=black, inner sep=1pt, label=right:$yz $] (e) at (5,0.2) {};

\node[] (x) at (-1,1.3) {$ x^2 w^2 $}; \node[] (x') at (0.8,0.6) {};
\draw[-{Straight Barb[length=3pt,width=3pt]}] (x) edge[out=0, in=130] node[below] {} (x');

\node[] (y) at (6.5,1.3) {$ x y z $}; \node[] (y') at (4.1,0.5) {}; 
\draw[-{Straight Barb[length=3pt,width=3pt]}] (y) edge[out=180, in=45] node[below] {} (y');

\node[] (z) at (6.5,-0.7) {$ y z w $}; \node[] (z') at (4.1,0.05) {};
\draw[-{Straight Barb[length=3pt,width=3pt]}] (z) edge[out=180, in=-45] node[below] {} (z');

\node[] (w) at (4.5,2) {$ x y z w^2 $}; \node[] (w') at (3.2,0) {}; 
\draw[-{Straight Barb[length=3pt,width=3pt]}] (w) edge[out=-120, in=90] node[below] {} (w');

\end{tikzpicture}\tabularnewline
\end{tabular}
\par\end{center}We write down the homogeneous components of $F$ as a graded $R$-module
below:
\begin{align*}
F_{0} & =R\\
F_{1} & =R\varepsilon_{1}+R\varepsilon_{2}+R\varepsilon_{3}+R\varepsilon_{4}+R\varepsilon_{5}\\
F_{2} & =R\varepsilon_{12}+R\varepsilon_{13}+R\varepsilon_{14}+R\varepsilon_{23}+R\varepsilon_{24}+R\varepsilon_{35}+R\varepsilon_{45}\\
F_{3} & =R\varepsilon_{123}+R\varepsilon_{124}+R\varepsilon_{1345}+R\varepsilon_{2345}\\
F_{4} & =R\varepsilon_{12345}
\end{align*}
The differential $\mathrm{d}_{\mathrm{A}}=\mathrm{d}$ is defined
on the non-simplicial faces as below
\begin{align*}
\mathrm{d}(\varepsilon_{12345}) & =x\varepsilon_{2345}-z\varepsilon_{124}+w\varepsilon_{1345}-y\varepsilon_{123}\\
\mathrm{d}(\varepsilon_{1345}) & =x^{2}\varepsilon_{35}-xw\varepsilon_{45}-zw\varepsilon_{14}+y\varepsilon_{13}\\
\mathrm{d}(\varepsilon_{2345}) & =xw\varepsilon_{35}-w^{2}\varepsilon_{45}-z\varepsilon_{24}+xy\varepsilon_{23}.
\end{align*}
We obtain a multiplication $\mu_{\mathrm{A}}$ on $F_{\mathrm{A}}$
from the one we constructed on $F_{\mathrm{K}}$ as follows: first
note that the canonical map $R\slash\boldsymbol{m}_{\mathrm{K}}\to R\slash\boldsymbol{m}_{\mathrm{A}}$
induces a multigraded comparison map $\pi\colon F_{\mathrm{K}}\to F_{\mathrm{A}}$
defined by
\begin{align*}
\pi(e_{5}) & =yz\varepsilon_{5} & \pi(e_{345}) & =0\\
\pi(e_{35}) & =yz\varepsilon_{35} & \pi(e_{234}) & =\varepsilon_{2345}\\
\pi(e_{45}) & =yz\varepsilon_{45} & \pi(e_{134}) & =\varepsilon_{1345}\\
\pi(e_{34}) & =x\varepsilon_{35}-w\varepsilon_{45} & \pi(e_{1234}) & =\varepsilon_{12345}
\end{align*}
and $\pi(e_{\sigma})=\varepsilon_{\sigma}$ for the remaining homogeneous
basis elements. Base changing to $R_{yz}$, we obtain quasi-isomorphisms
$F_{\mathrm{A},yz}\to0\leftarrow F_{\mathrm{K},yz}$. In particular,
there exists a comparison map $\iota\colon F_{\mathrm{A},yz}\to F_{\mathrm{K},yz}$
which splits comparison map $\pi\colon F_{\mathrm{K},yz}\to F_{\mathrm{A},yz}$.
By considering the multigrading as well as the Leibniz rule, we see
that
\begin{align*}
\iota(\varepsilon_{5}) & =e_{5}/yz & \iota(\varepsilon_{2345}) & =-e_{234}+e_{345}/yz\\
\iota(\varepsilon_{35}) & =e_{35}/yz & \iota(\varepsilon_{1345}) & =e_{134}-e_{345}/yz\\
\iota(\varepsilon_{45}) & =e_{45}/yz & \iota(\varepsilon_{12345}) & =e_{1234}
\end{align*}
and $\iota(\varepsilon_{\sigma})=e_{\sigma}$ for the remaining homogeneous
basis elements. With this in mind, we define a multiplication $\mu_{\mathrm{A}}$
on $F_{\mathrm{A}}$ by transporting the multiplication $\mu_{\mathrm{K}}$
on $F_{\mathrm{K},yz}$ by setting $\mu_{\mathrm{A}}:=\pi\mu_{\mathrm{K}}\iota^{\otimes2}$.
In other words, we have
\begin{equation}
\varepsilon_{\sigma}\star_{\mu_{\mathrm{A}}}\varepsilon_{\tau}=\pi(\iota(\varepsilon_{\sigma})\star_{\mu_{\mathrm{K}}}\iota(\varepsilon_{\tau}))\label{eq:avmult}
\end{equation}
for all homogeneous basis elements $\varepsilon_{\sigma},\varepsilon_{\tau}$
of $F_{\mathrm{A},yz}$. It is straightforward to check that $\mu_{\mathrm{A}}$
restricts to a multiplication on $F_{\mathrm{A}}$ (the coefficients
in (\ref{eq:avmult}) are in $R$). Note that $\mu_{\mathrm{A}}$
is not associative since
\[
[\varepsilon_{1},\varepsilon_{5},\varepsilon_{2}]=-\mathrm{d}(\varepsilon_{12345})\neq0.
\]
\end{example}

\begin{example}\label{example3} Let $R=\Bbbk[x,y,z,w]$, let $\boldsymbol{m}_{\mathrm{M}}=\boldsymbol{m}=x^{2},w^{2},zw,xy,y^{2}z,yz^{2}$
and let $F_{\mathrm{M}}=F$ be the minimal free resolution of $R\slash\boldsymbol{m}$
of $R$. One can visualize $F$ as being supported on the $\boldsymbol{m}$-labeled
simplicial complex below: \begin{center}
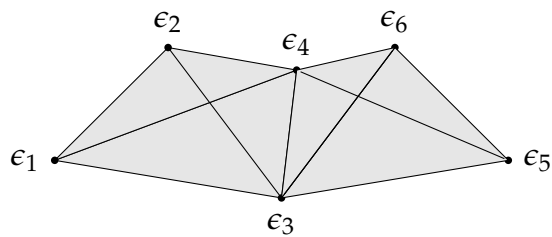
\begin{tabular}{ccccc}
\begin{tikzpicture}[scale=1]
\draw[fill=gray!20] (0,0) -- (3,-0.5) -- (3.2,1.2)-- (0,0); 
\draw[fill=gray!20] (0,0) -- (1.5,1.5) -- (3.2,1.2)-- (0,0); 
\draw[] (1.5,1.5) -- (3,-0.5);

\node[circle, fill=black, inner sep=1pt, label=left:$\epsilon _1 $] (e1) at (0,0) {}; 
\node[circle, fill=black, inner sep=1pt, label=above:$\epsilon _2 $] (e2) at (1.5,1.5) {}; 
\node[circle, fill=black, inner sep=1pt, label=below:$\epsilon _3 $] (e3) at (3,-0.5) {};
\node[circle, fill=black, inner sep=1pt, label=above:$\epsilon _4 $] (e4) at (3.2,1.2) {}; 
\node[circle, fill=black, inner sep=1pt, label=right:$\epsilon _5 $] (e5) at (6,0) {};
\node[circle, fill=black, inner sep=1pt, label=above:$\epsilon _6 $] (e6) at (4.5,1.5) {};

\draw[fill=gray!20] (3,-0.5) -- (6,0) -- (4.5,1.5) -- (3,-0.5);
\draw[fill=gray!20]  (3.2,1.2) -- (4.5,1.5) -- (3,-0.5) -- (3.2,1.2);

\draw[] (e4) -- (e5);
\draw[] (e3) -- (e6);

\end{tikzpicture} &  &  &  & 
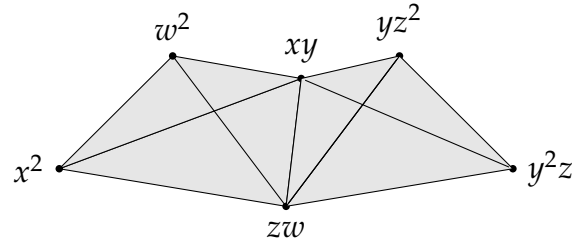
\begin{tikzpicture}[scale=1]
\draw[fill=gray!20] (0,0) -- (3,-0.5) -- (3.2,1.2)-- (0,0); 
\draw[fill=gray!20] (0,0) -- (1.5,1.5) -- (3.2,1.2)-- (0,0); 
\draw[] (1.5,1.5) -- (3,-0.5);

\node[circle, fill=black, inner sep=1pt, label=left:$x^2 $] (e1) at (0,0) {}; 
\node[circle, fill=black, inner sep=1pt, label=above:$w^2 $] (e2) at (1.5,1.5) {}; 
\node[circle, fill=black, inner sep=1pt, label=below:$zw $] (e3) at (3,-0.5) {};
\node[circle, fill=black, inner sep=1pt, label=above:$xy $] (e4) at (3.2,1.2) {}; 
\node[circle, fill=black, inner sep=1pt, label=right:$y^2 z $] (e5) at (6,0) {};
\node[circle, fill=black, inner sep=1pt, label=above:$y z^2 $] (e6) at (4.5,1.5) {};

\draw[fill=gray!20] (3,-0.5) -- (6,0) -- (4.5,1.5) -- (3,-0.5);
\draw[fill=gray!20]  (3.2,1.2) -- (4.5,1.5) -- (3,-0.5) -- (3.2,1.2);

\draw[] (e4) -- (e5);
\draw[] (e3) -- (e6);

\end{tikzpicture}\tabularnewline
\end{tabular}
\par\end{center}We write down the homogeneous components of $F$ as a graded $R$-module
below:
\begin{align*}
F_{0} & =R\\
F_{1} & =R\epsilon_{1}+R\epsilon_{2}+R\epsilon_{3}+R\epsilon_{4}+R\epsilon_{5}+R\epsilon_{6}\\
F_{2} & =R\epsilon_{12}+R\epsilon_{13}+R\epsilon_{14}+R\epsilon_{23}+R\epsilon_{24}+R\epsilon_{34}+R\epsilon_{35}+R\epsilon_{36}+R\epsilon_{45}+R\epsilon_{46}+R\epsilon_{56}\\
F_{3} & =R\epsilon_{123}+R\epsilon_{124}+R\epsilon_{134}+R\epsilon_{234}+R\epsilon_{345}+R\epsilon_{346}+R\epsilon_{356}+R\epsilon_{456}\\
F_{4} & =R\epsilon_{1234}+R\epsilon_{3456}.
\end{align*}
The canonical map $R\slash\boldsymbol{m}_{\mathrm{K}}\to R\slash\boldsymbol{m}_{\mathrm{M}}$
induces multigraded comparison maps $\pi_{\lambda}\colon F_{\mathrm{K}}\to F_{\mathrm{M}}$
where $\lambda\in\Bbbk$ and where $\pi_{\lambda}$ is defined by
\begin{align*}
\pi_{\lambda}(e_{5}) & =\lambda z\epsilon_{5}+(1-\lambda)y\epsilon_{6}\\
\pi_{\lambda}(e_{35}) & =\lambda z\epsilon_{35}+(1-\lambda)y\epsilon_{36}\\
\pi_{\lambda}(e_{45}) & =\lambda z\epsilon_{45}+(1-\lambda)y\epsilon_{46}\\
\pi_{\lambda}(e_{345}) & =\lambda z\epsilon_{345}+(1-\lambda)y\epsilon_{346}
\end{align*}
and $\pi_{\lambda}(e_{\sigma})=\epsilon_{\sigma}$ for the remaining
homogeneous basis elements. We will choose $\lambda=1$ and view $F_{\mathrm{K}}$
as a subcomplex of $F_{\mathrm{M}}$ via $\pi=\pi_{1}$. We define
a multigraded multiplication $\mu_{\mathrm{M}}$ on $F_{\mathrm{M}}$
so that it extends the multiplication $\mu_{\mathrm{K}}$ on $F_{\mathrm{K}}$.
Considerations of the Leibniz rule and the multigrading tells us that
we are already forced to have:
\begin{align*}
\epsilon_{1}\star\epsilon_{5} & =yz\epsilon_{14}+x\epsilon_{45} &  &  & \epsilon_{1}\star\epsilon_{6} & =z^{2}e_{14}+xe_{46}\\
\epsilon_{2}\star\epsilon_{5} & =y^{2}\epsilon_{23}+w\epsilon_{35} &  &  & \epsilon_{2}\star\epsilon_{6} & =yz\epsilon_{23}+w\epsilon_{36}\\
\epsilon_{2}\star\epsilon_{45} & =-y\epsilon_{234}+w\epsilon_{345} &  &  & \epsilon_{2}\star\epsilon_{46} & =-ze_{234}+w\epsilon_{346}\\
\epsilon_{1}\star\epsilon_{35} & =y\epsilon_{134}-x\epsilon_{345} &  &  & \epsilon_{1}\star\epsilon_{36} & =z\epsilon_{134}-x\epsilon_{346}.
\end{align*}
In particular, $\mu_{\mathrm{K}}$ is not associative (and in fact
any multigraded multiplication on $F_{\mathrm{M}}$ is not associative)
since:
\[
[\epsilon_{1},\epsilon_{5},\epsilon_{2}]=-y\mathrm{d}(\epsilon_{1234})\neq0\quad\text{and}\quad[\epsilon_{1},\epsilon_{6},\epsilon_{2}]=-z\mathrm{d}(\epsilon_{1234})\neq0.
\]
On the other hand, since the multiplication of $F_{\mathrm{M}}$ extends
the multiplication of $F_{\mathrm{K}}$, we see that the comparison
map $F_{\mathrm{K}}\to F_{\mathrm{M}}$ is multiplicative, and hence
$F_{\mathrm{K}}$ is an MDG subalgebra of $F_{\mathrm{M}}$. 

\end{example}

\begin{example}\label{example4} Let $R=\Bbbk[x,y,z,w]$, let $\boldsymbol{m}_{\mathrm{O}}=\boldsymbol{m}=x^{2},w^{2},zw,xy,y^{2},z^{2}$
and let $F_{\mathrm{O}}=F$ be the minimal free resolution of $R\slash\boldsymbol{m}$
over $R$. One can visualize $F$ as being supported on the $\boldsymbol{m}$-labeled
simplicial complex below:\begin{center}
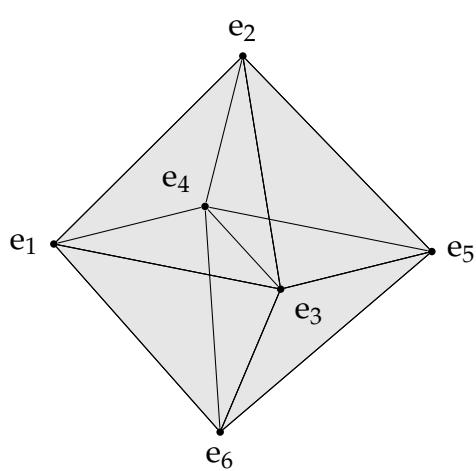
\begin{tabular}{ccccc}
\begin{tikzpicture}

\draw[fill=gray!20] (-1,0) -- (2,-0.6) -- (1.5,2.5)-- (-1,0); 
\draw[fill=gray!20] (2,-0.6)  -- (1.5,2.5) -- (4,-0.1)-- (2,-0.6) ; 
\draw[fill=gray!20] (-1,0) -- (2,-0.6) -- (1.2,-2.5)-- (-1,0); 
\draw[fill=gray!20] (2,-0.6) -- (4,-0.1) -- (1.2,-2.5)-- (2,-0.6); 

\node[circle, fill=black, inner sep=1pt, label=left:$\mathrm{e}_1 $] (e1) at (-1,0) {}; 
\node[circle, fill=black, inner sep=1pt, label=above left :$\mathrm{e}_4 $] (e2) at (1,0.5) {}; 
\node[circle, fill=black, inner sep=1pt, label=below right:$\mathrm{e}_3 $] (e3) at (2,-0.6) {};
\node[circle, fill=black, inner sep=1pt, label=above:$\mathrm{e}_2 $] (e4) at (1.5,2.5) {}; 
\node[circle, fill=black, inner sep=1pt, label=right:$\mathrm{e}_5 $] (e5) at (4,-0.1) {};
\node[circle, fill=black, inner sep=1pt, label=below:$\mathrm{e}_6 $] (e6) at (1.2,-2.5) {};

\draw[] (e1) -- (e2);
\draw[] (e1) -- (e3);
\draw[] (e1) -- (e4);
\draw[] (e3) -- (e6);
\draw[] (e4) -- (e5);
\draw[] (e3) -- (e4);
\draw[] (e2) -- (e5);
\draw[] (e3) -- (e5);
\draw[] (e1) -- (e6);
\draw[] (e5) -- (e6);
\draw[] (e2) -- (e6);
\draw[] (e2) -- (e4);
\draw[] (e3) -- (e2);

\end{tikzpicture} &  &  &  & 
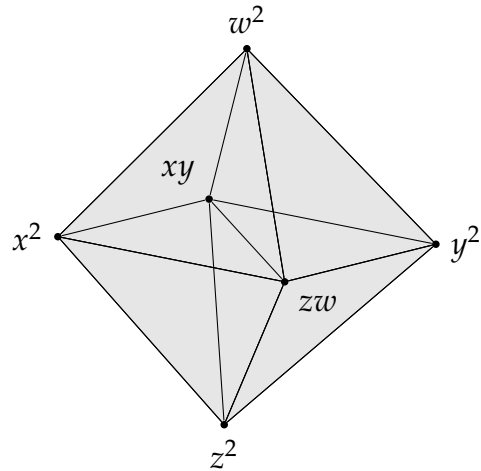
\begin{tikzpicture}

\draw[fill=gray!20] (-1,0) -- (2,-0.6) -- (1.5,2.5)-- (-1,0); 
\draw[fill=gray!20] (2,-0.6)  -- (1.5,2.5) -- (4,-0.1)-- (2,-0.6) ; 
\draw[fill=gray!20] (-1,0) -- (2,-0.6) -- (1.2,-2.5)-- (-1,0); 
\draw[fill=gray!20] (2,-0.6) -- (4,-0.1) -- (1.2,-2.5)-- (2,-0.6); 

\node[circle, fill=black, inner sep=1pt, label=left:$x^2 $] (e1) at (-1,0) {}; 
\node[circle, fill=black, inner sep=1pt, label=above left :$xy $] (e2) at (1,0.5) {}; 
\node[circle, fill=black, inner sep=1pt, label=below right:$zw $] (e3) at (2,-0.6) {};
\node[circle, fill=black, inner sep=1pt, label=above:$w^2 $] (e4) at (1.5,2.5) {}; 
\node[circle, fill=black, inner sep=1pt, label=right:$y^2 $] (e5) at (4,-0.1) {};
\node[circle, fill=black, inner sep=1pt, label=below:$z^2 $] (e6) at (1.2,-2.5) {};

\draw[] (e1) -- (e2);
\draw[] (e1) -- (e3);
\draw[] (e1) -- (e4);
\draw[] (e3) -- (e6);
\draw[] (e4) -- (e5);
\draw[] (e3) -- (e4);
\draw[] (e2) -- (e5);
\draw[] (e3) -- (e5);
\draw[] (e1) -- (e6);
\draw[] (e5) -- (e6);
\draw[] (e2) -- (e6);
\draw[] (e2) -- (e4);
\draw[] (e3) -- (e2);

\end{tikzpicture}\tabularnewline
\end{tabular}
\par\end{center}We write down the homogeneous components of $F$ as a graded $R$-module
below:
\begin{align*}
F_{0} & =R\\
F_{1} & =R\mathrm{e}_{1}+R\mathrm{e}_{2}+R\mathrm{e}_{3}+R\mathrm{e}_{4}+R\mathrm{e}_{5}+R\mathrm{e}_{6}\\
F_{2} & =R\mathrm{e}_{12}+R\mathrm{e}_{13}+R\mathrm{e}_{14}+R\mathrm{e}_{16}+R\mathrm{e}_{23}+R\mathrm{e}_{24}+R\mathrm{e}_{25}+R\mathrm{e}_{34}+R\mathrm{e}_{35}+R\mathrm{e}_{36}+R\mathrm{e}_{45}+R\mathrm{e}_{46}+R\mathrm{e}_{56}\\
F_{3} & =R\mathrm{e}_{123}+R\mathrm{e}_{124}+R\mathrm{e}_{134}+R\mathrm{e}_{136}+R\mathrm{e}_{146}+R\mathrm{e}_{234}+R\mathrm{e}_{235}+R\mathrm{e}_{245}+R\mathrm{e}_{345}+R\mathrm{e}_{346}+R\mathrm{e}_{356}+R\mathrm{e}_{456}\\
F_{4} & =R\mathrm{e}_{1234}+R\mathrm{e}_{1346}+R\mathrm{e}_{2345}+R\epsilon_{3456}.
\end{align*}
The canonical map $R\slash\boldsymbol{m}_{\mathrm{M}}\to R\slash\boldsymbol{m}_{\mathrm{O}}$
induces an injective multigraded comparison map $F_{\mathrm{M}}\to F_{\mathrm{O}}$
and we identify $F_{\mathrm{M}}$ with this subcomplex of $F_{\mathrm{O}}$.
This time it is not possible extend the multiplication of $F_{\mathrm{M}}$
to a multiplication on $F_{\mathrm{O}}$. Indeed, assuming we could
extend the multiplication, then we'd have
\begin{align*}
z(\mathrm{e}_{2}\star\mathrm{e}_{5}) & =\mathrm{e}_{2}\star(z\mathrm{e}_{5})\\
 & =\epsilon_{2}\star\epsilon_{5}\\
 & =y^{2}\epsilon_{23}+w\epsilon_{35}\\
 & =y^{2}\mathrm{e}_{23}+w\mathrm{e}_{35},
\end{align*}
which would imply $\mathrm{e}_{2}\star\mathrm{e}_{5}=(y^{2}/z)\mathrm{e}_{23}+(w/z)\mathrm{e}_{35}$.
However this is obviously not in $F_{\mathrm{O}}$ since the coefficients
are not in $R$. On the other hand, it turns out that there is a better
choice of a multigraded multiplication on $F_{\mathrm{O}}$ that we
can use anyways: namely $\mathrm{e}_{2}\star\mathrm{e}_{5}=\mathrm{e}_{25}$.
In fact, this is the only possible choice we can make if we want the
multiplication to be multigraded. Similarly, we are forced to have
$\mathrm{e}_{1}\star\mathrm{e}_{6}=\mathrm{e}_{16}$. Using the computer
algebra system Singular, we found that this extends to an \emph{associative}
multigraded multiplication on $F_{\mathrm{O}}$ which has the following
minimal presentation:
\begin{alignat*}{5}
\mathrm{e}_{1}^{2} & =0 &  &  & \mathrm{e}_{2}\star\mathrm{e}_{5} & =\mathrm{e}_{25} &  &  & \mathrm{e}_{2}\star\mathrm{e}_{16} & =-z\mathrm{e}_{123}-w\mathrm{e}_{136}\\
\mathrm{e}_{2}^{2} & =0 &  &  & \mathrm{e}_{2}\star\mathrm{e}_{6} & =z\mathrm{e}_{23}+w\mathrm{e}_{36} &  &  & \mathrm{e}_{2}\star\mathrm{e}_{46} & =\mathrm{e}_{234}+\mathrm{e}_{346}\\
\mathrm{e}_{3}^{2} & =0 &  &  & \mathrm{e}_{3}\star\mathrm{e}_{4} & =\mathrm{e}_{34} &  &  & \mathrm{e}_{2}\star\mathrm{e}_{56} & =-z\mathrm{e}_{235}+w\mathrm{e}_{356}\\
\mathrm{e}_{4}^{2} & =0 &  &  & \mathrm{e}_{3}\star\mathrm{e}_{5} & =\mathrm{e}_{35} &  &  & \mathrm{e}_{3}\star\mathrm{e}_{45} & =\mathrm{e}_{345}\\
\mathrm{e}_{5}^{2} & =0 &  &  & \mathrm{e}_{3}\star\mathrm{e}_{6} & =z\mathrm{e}_{36} &  &  & \mathrm{e}_{5}\star\mathrm{e}_{24} & =y\mathrm{e}_{245}\\
\mathrm{e}_{6}^{2} & =0 &  &  & \mathrm{e}_{4}\star\mathrm{e}_{5} & =y\mathrm{e}_{45} &  &  & \mathrm{e}_{6}\star\mathrm{e}_{13} & =z\mathrm{e}_{136}\\
\mathrm{e}_{1}\star\mathrm{e}_{2} & =\mathrm{e}_{12} & \qquad\qquad &  & \mathrm{e}_{4}\star\mathrm{e}_{6} & =\mathrm{e}_{46} & \qquad\qquad &  & \mathrm{e}_{6}\star\mathrm{e}_{34} & =z\mathrm{e}_{346}\\
\mathrm{e}_{1}\star\mathrm{e}_{3} & =\mathrm{e}_{13} &  &  & \mathrm{e}_{5}\star\mathrm{e}_{6} & =\mathrm{e}_{56} &  &  & \mathrm{e}_{6}\star\mathrm{e}_{35} & =z\mathrm{e}_{356}\\
\mathrm{e}_{1}\star\mathrm{e}_{4} & =x\mathrm{e}_{14} &  &  & \mathrm{e}_{1}\star\mathrm{e}_{25} & =y\mathrm{e}_{124}-x\mathrm{e}_{245} &  &  & \mathrm{e}_{6}\star\mathrm{e}_{45} & =\mathrm{e}_{456}\\
\mathrm{e}_{1}\star\mathrm{e}_{5} & =y\mathrm{e}_{14}+x\mathrm{e}_{45} &  &  & \mathrm{e}_{1}\star\mathrm{e}_{35} & =y\mathrm{e}_{134}-x\mathrm{e}_{345} &  &  & \mathrm{e}_{1}\star\mathrm{e}_{235} & =y\mathrm{e}_{1234}+x\mathrm{e}_{2345}\\
\mathrm{e}_{1}\star\mathrm{e}_{6} & =\mathrm{e}_{16} &  &  & \mathrm{e}_{1}\star\mathrm{e}_{56} & =y\mathrm{e}_{146}+x\mathrm{e}_{456} &  &  & \mathrm{e}_{1}\star\mathrm{e}_{346} & =x\mathrm{e}_{1346}\\
\mathrm{e}_{2}\star\mathrm{e}_{3} & =w\mathrm{e}_{23} &  &  &  &  &  &  & \mathrm{e}_{1}\star\mathrm{e}_{356} & =y\mathrm{e}_{1346}-x\mathrm{e}_{3456}\\
\mathrm{e}_{2}\star\mathrm{e}_{4} & =\mathrm{e}_{24} &  &  &  &  &  &  & \mathrm{e}_{2}\star\mathrm{e}_{456} & =z\mathrm{e}_{2345}+w\mathrm{e}_{3456}
\end{alignat*}
In Example~(\ref{example455}), we demonstrate how one can find associative
multiplications like this using a computer algebra system like Singular. 

\end{example}

\begin{example}\label{example5} Let $R=\Bbbk[x,y,z,w]$, let $\boldsymbol{m}_{\mathrm{N}}=\boldsymbol{m}=x^{2},w^{2},zw,xy,yz,y^{2},z^{2}$,
and let $F_{\mathrm{N}}=F$ be the minimal free resolution of $R\slash\boldsymbol{m}$
over $R$. One can visualize $F$ as being supported on the $\boldsymbol{m}$-labeled
cellular complex below: \begin{center}
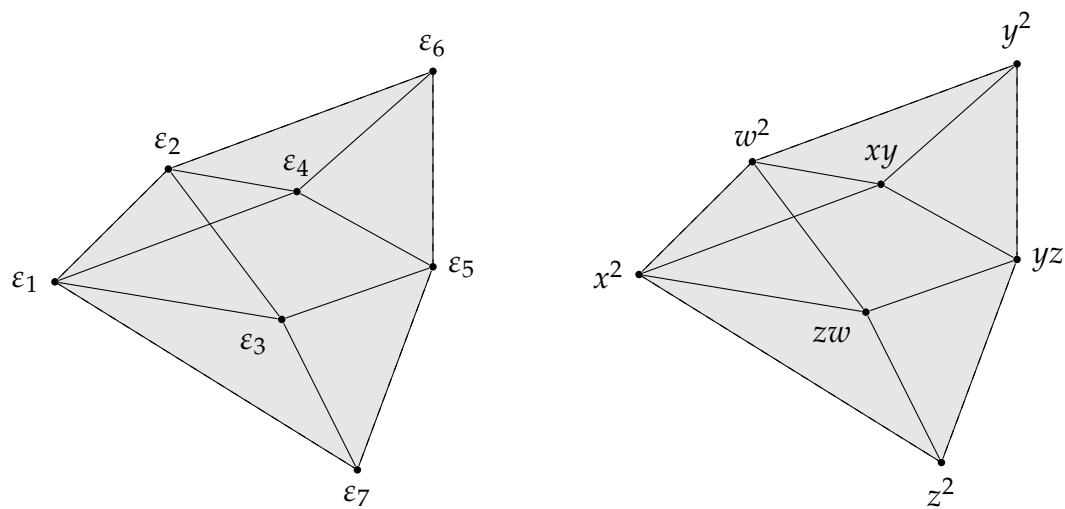
\begin{tabular}{cccc}
\begin{tikzpicture}[scale=1]

\draw[fill=gray!20, dashed] (0,0) -- (4,-2.5) -- (5,0.2) -- (5,2.8) -- (1.5,1.5) -- (0,0); 

\node[circle, fill=black, inner sep=1pt, label=left:$\varepsilon _1 $] (a) at (0,0) {};
\node[circle, fill=black, inner sep=1pt, label=above:$\varepsilon _2 $] (b) at (1.5,1.5) {};
\node[circle, fill=black, inner sep=1pt, label=below left :$\varepsilon _3 $] (c) at (3,-0.5) {};
\node[circle, fill=black, inner sep=1pt, label=above:$\varepsilon _4 $] (d) at (3.2,1.2) {}; 
\node[circle, fill=black, inner sep=1pt, label=right:$\varepsilon _5 $] (e) at (5,0.2) {};
\node[circle, fill=black, inner sep=1pt, label=above:$\varepsilon _6 $] (f) at (5,2.8) {};
\node[circle, fill=black, inner sep=1pt, label=below:$\varepsilon _7 $] (g) at (4,-2.5) {};

\draw[] (b) -- (d);
\draw[] (a) -- (d);
\draw[] (g) -- (e); 
\draw[] (f) -- (d);
\draw[] (f) -- (b);
\draw[] (f) -- (e);
\draw[] (a) -- (b);
\draw[] (a) -- (c);
\draw[] (e) -- (d);
\draw[] (e) -- (c);
\draw[] (b) -- (c);
\draw[] (a) -- (g);
\draw[] (g) -- (c);

\end{tikzpicture} &  &  & \begin{tikzpicture}[scale=1]

\draw[fill=gray!20, dashed] (0,0) -- (4,-2.5) -- (5,0.2) -- (5,2.8) -- (1.5,1.5) -- (0,0); 

\node[circle, fill=black, inner sep=1pt, label=left:$x^2  $] (a) at (0,0) {};
\node[circle, fill=black, inner sep=1pt, label=above:$w^2 $] (b) at (1.5,1.5) {};
\node[circle, fill=black, inner sep=1pt, label=below left :$zw $] (c) at (3,-0.5) {};
\node[circle, fill=black, inner sep=1pt, label=above:$xy $] (d) at (3.2,1.2) {}; 
\node[circle, fill=black, inner sep=1pt, label=right:$yz$] (e) at (5,0.2) {};
\node[circle, fill=black, inner sep=1pt, label=above:$y^2 $] (f) at (5,2.8) {};
\node[circle, fill=black, inner sep=1pt, label=below:$z^2 $] (g) at (4,-2.5) {};

\draw[] (b) -- (d);
\draw[] (a) -- (d);
\draw[] (g) -- (e); 
\draw[] (f) -- (d);
\draw[] (f) -- (b);
\draw[] (f) -- (e);
\draw[] (a) -- (b);
\draw[] (a) -- (c);
\draw[] (e) -- (d);
\draw[] (e) -- (c);
\draw[] (b) -- (c);
\draw[] (a) -- (g);
\draw[] (g) -- (c);

\end{tikzpicture}\tabularnewline
\end{tabular}
\par\end{center}It is visibly clear that the map $R\slash\boldsymbol{m}_{\mathrm{A}}\to R\slash\boldsymbol{m}_{\mathrm{N}}$
induces a comparison map $\iota\colon F_{\mathrm{A}}\to F_{\mathrm{N}}$
defined by $\iota(\varepsilon_{\sigma})=\varepsilon_{\sigma}$ for
all homogeneous basis element $\varepsilon_{\sigma}$ of $F_{\mathrm{A}}$
(in particular, there are no monomials showing up in the coefficients
in this comparison map). Thus we run into the same problem as in Example~(\ref{example2}),
and so there is no way to choose a multigraded multiplication on $F_{\mathrm{N}}$
which is associative. \end{example}

\begin{example}\label{example6} Let $R=\Bbbk[x,y,z,w]$, let $m=xyzw,$
let $\boldsymbol{m}=mx,my,mz,mw$, and let $F$ be the minimal free
resolution of $R\slash\boldsymbol{m}$ over $R$. Then $F$ is just
the Taylor resolution with respect to $\boldsymbol{m}$ and is supported
on the $3$-simplex. Usually $F$ comes equipped with an associative
multiplication giving it the structure of a DG algebra, however we
wish to consider a different multiplication $\mu$ which gives it
the structure of a non-associative MDG algebra. In particular, this
multiplication will start out as:
\begin{align*}
e_{1}\star e_{2} & =xyzwe_{12}\\
e_{1}\star e_{3} & =xyz^{2}e_{14}-x^{2}yze_{34}\\
e_{2}\star e_{3} & =xyzwe_{23}\\
e_{3}\star e_{12} & =xyzwe_{123}-xy^{2}ze_{134}\\
e_{2}\star e_{14} & =-xyzwe_{124}\\
e_{2}\star e_{34} & =xyzwe_{234}
\end{align*}
At this point, no matter how we extend this multiplication, it will
not be associative since
\[
[e_{2},e_{1},e_{3}]=x^{2}y^{2}z^{2}w\mathrm{d}(e_{1234})\neq0.
\]
The point we wish to emphasize here is that there is a ``better''
mutliplication that we can use on $F$ anyways, namely the Taylor
multiplication. In general we would like to find the best possible
multiplication in the sense that it is as associative as possible.\end{example}

\subsubsection{Multigraded Multiplications coming from the Taylor Algebra}

In this subsubsection, we want to explain how all of the multigraded
multiplications that we have considered thus far can be viewed as
coming from a Taylor multiplication. Let $R=\Bbbk[x_{1},\dots,x_{d}]$,
let $I$ be a monomial ideal in $R$, let $F$ be the minimal free
resolution of $R\slash I$ over $R$, and let $T$ be the Taylor algebra
resolution of $R\slash I$ over $R$. We denote the Taylor multiplication
on $T$ by $\nu_{T}$. Let $\nu$ be a possibly different multiplication
on $T$. We write $T_{\nu}$ to be the MDG $R$-algebra whose underlying
$R$-complex is the same as the underlying complex of $T$ but whose
multiplication is $\nu$. Since $F$ is the minimal free resolution
of $R\slash I$ over $R$ and since $T$ is a free resolution of $R\slash I$
over $R$, there exists multigraded chain maps $\iota\colon F\to T$
and $\pi\colon T\to F$ which lift the identity map $R\slash I\to R\slash I$
such that $\iota\colon F\to T$ is injective and is split by $\pi\colon T\to F$,
meaning $\pi\iota=1$. By identifying $F$ with $\iota(F)$ if necessary,
we may assume that $\iota\colon F\subseteq T$ is inclusion and that
$\pi\colon T\to F$ is a projection, meaning $\pi\colon T\to F$ is
a surjective chain map which satisfies $\pi^{2}=\pi$, or equivalently,
$\pi\colon T\to T$ is a chain map with $\mathrm{im}\,\pi=F$. Using
the comparison maps $\iota\colon F\to T$ and $\pi\colon T\to F$,
we can transport multiplications on $F$ to multiplications on $T$
and vice versa. Namely, given a multiplication $\mu$ on $F$, we
set $\widetilde{\mu}:=\iota\mu\pi^{\otimes2}$. Similarly, given a
multiplication $\nu$ on $T$, we set $\widetilde{\nu}:=\pi\nu\iota^{\otimes2}$.
All of the multigraded multiplications that we've considered thus
far are of the form $\widetilde{\nu}_{T}$. For instance: 

\begin{example}\label{example} The multiplication $\mu$ in Example~(\ref{example1})
is given by $\mu=\pi\nu_{T}\iota^{\otimes2}$ where $T$ is the Taylor
algebra resolution of $R\slash\boldsymbol{m}_{\mathrm{K}}$ and where
$\pi\colon T\to F$ is defined by
\begin{align*}
\pi(e_{15}) & =yz^{2}e_{14}+xe_{45}\\
\pi(e_{25}) & =y^{2}ze_{23}+we_{35}\\
\pi(e_{245}) & =-yze_{234}+we_{35}\\
\pi(e_{235}) & =0\\
\pi(e_{2345}) & =0\\
 & \vdots
\end{align*}
and so on. \end{example}

\newpage{}

\subsection{MDG Modules}

We now want to define MDG $A$-modules where $A$ is an MDG $R$-algebra. 

\begin{defn}\label{defn} Let $X$ be an $R$-complex equipped with
chain maps $\mu_{A,X}\colon A\otimes_{R}X\to X$ and $\mu_{X,A}\colon X\otimes_{R}A\to X$,
denoted $a\otimes x\mapsto ax$ and $x\otimes a\mapsto xa$ respectively.
\begin{enumerate}
\item We say $X$ is \textbf{unital }if $1x=x=x1$ for all $x\in X$.
\item We say $X$ is \textbf{graded-commutative} if $ax=(-1)^{|a||x|}xa$
for all $a\in A$ homogeneous and $x\in X$ homogeneous. In this case,
$\mu_{X,A}$ is completely determined by $\mu_{A,X}$, and thus we
completely forget about it and write $\mu_{X}=\mu_{A,X}$.
\item We say $X$ is \textbf{associative }if $a_{1}(a_{2}x)=(a_{1}a_{2})x$
for all $a_{1},a_{2}\in A$ and $x\in X$.
\end{enumerate}
We say $X$ is an \textbf{MDG $A$-module }if it is unital and graded-commutative.
In this case we call $\mu_{X}$ the $A$\textbf{-scalar multiplication
}of $X$. Note that if both $A$ and $X$ are associative, then often
in the literature one calls $X$ a\textbf{ }DG\textbf{ $A$}-module.
Suppose $Y$ is another MDG $A$-module. An \textbf{MDG $A$-module
homomorphism }is a chain map $\varphi\colon X\to Y$ such that $\varphi$
is also multiplicative, meaning
\[
\varphi(ax)=a\varphi(x)
\]
for all $a\in A$ and $x\in X$. \end{defn}

\begin{rem}\label{rem} Let $A$ and $B$ be MDG $R$-algebras and
let $\varphi\colon A\to B$ be a chain map such that $\varphi(1)=1$.
Then we give $B$ the structure of an MDG $A$-module by defining
an $A$-scalar multiplication on $B$ via
\[
a\cdot b=\varphi(a)b
\]
for all $a\in A$ and $b\in B$. Note that we need $\varphi(1)=1$
in order for $B$ to be unital as an $A$-module. Also note that $\varphi$
is an MDG $A$-module homomorphism if and only if it is multiplicative.
Indeed, it is an MDG $A$-module homomorphism if and only if for all
$a_{1},a_{2}\in A$ we have
\[
\varphi(a_{1}a_{2})=a_{1}\cdot\varphi(a_{2})=\varphi(a_{1})\varphi(a_{2}),
\]
which is equivalent to saying $\varphi$ is multiplicative (since
we already have $\varphi(1)=1$). \end{rem}

\section{Associators and Multiplicators}

In order to get a better understanding as to how far away MDG objects
are from being DG objects, we need to discuss associators and multiplicators.
Associators will help us measure the failure for an MDG $A$-module
$X$ to be associative, whereas multiplicators will help up measure
the failure for a chain map $\varphi\colon X\to Y$ between MDG $A$-modules
$X$ and $Y$ to be multiplicative. In the case where $A$ and $B$
are MDG algebras and $\varphi\colon A\to B$ is a chain map such that
$\varphi(1)=1$, it will turn out that the multiplicator of $\varphi$
is just a special type of associator. Thus our main focus in this
section will be on associators.

\subsection{Associators}

We begin by defining associators. Throughout this subsection, let
$R$ be a commutative ring, let $A$ be an MDG $R$-algebra, and let
$X$ be an MDG $A$-module. 

\begin{defn}\label{defnmulthompermcomp} The \textbf{associator }of
$X$ is the chain map, denoted $[\cdot]_{X}$ (or more simply by $[\cdot]$
if $X$ is understood from context), from $A\otimes_{R}A\otimes_{R}X$
to $X$ defined by
\[
[\cdot]:=\mu(\mu\otimes1-1\otimes\mu).
\]
Note that we use $\mu$ to denote both the multiplication $\mu_{A}$
on $A$ and the $A$-scalar multiplication $\mu_{X}$ on $X$ where
context makes clear which multiplication $\mu$ refers to. We denote
by $[\cdot,\cdot,\cdot]\colon A\times A\times X\to X$ to be the unique
$R$-trilinear map which corresponds to $[\cdot]$ via the universal
mapping property of tensor products. Thus we have
\[
[a_{1}\otimes a_{2}\otimes x]=(a_{1}a_{2})x-a_{1}(a_{2}x)=[a_{1},a_{2},x]
\]
for all $a_{1},a_{2}\in A$ and $x\in X$. \end{defn}

\subsubsection{Associator Identities}

In order to familiarize ourselves with the associator we collect together
some useful identities that the associator satisfies in this subsubsection:
\begin{itemize}
\item For all $a_{1},a_{2}\in A$ homogeneous and $x\in X$ we have the
Leibniz rule
\begin{equation}
\mathrm{d}[a_{1},a_{2},x]=[\mathrm{d}a_{1},a_{2},x]+(-1)^{|a_{1}|}[a_{1},\mathrm{d}a_{2},x]+(-1)^{|a_{1}|+|a_{2}|}[a_{1},a_{2},\mathrm{d}x].\label{eq:leibnizlaw7}
\end{equation}
\item For all $a_{1},a_{2}\in A$ homogeneous and $x\in X$ homogeneous
we have
\begin{equation}
[a_{1},a_{2},x]=-(-1)^{|a_{1}||a_{2}|+|a_{1}||x|+|a_{2}||x|}[x,a_{2},a_{1}].\label{eq:identity3}
\end{equation}
\item For all $a_{1},a_{2}\in A$ homogeneous and $x\in X$ homogeneous
we have
\begin{equation}
[a_{1},a_{2},x]=-(-1)^{|a_{1}||x|+|a_{2}||x|}[x,a_{1},a_{2}]-(-1)^{|a_{1}||a_{2}|+|a_{1}||x|}[a_{2},x,a_{1}]\label{eq:identity2}
\end{equation}
\item For all $a_{1},a_{2}\in A$ homogeneous and $x\in X$ homogeneous
we have
\begin{equation}
[a_{1},a_{2},x]=(-1)^{|a_{1}||a_{2}|}[a_{2},a_{1},x]+(-1)^{|a_{2}||x|}[a_{1},x,a_{2}]\label{eq:identity4-1}
\end{equation}
\item For all $a_{1},a_{2},a_{3}\in A$ and $x\in X$ we have
\begin{equation}
a_{1}[a_{2},a_{3},x]=[a_{1}a_{2},a_{3},x]-[a_{1},a_{2}a_{3},x]+[a_{1},a_{2},a_{3}x]-[a_{1},a_{2},a_{3}]x\label{eq:identity1}
\end{equation}
\end{itemize}
~~~~~~The way the signs in (\ref{eq:identity3}) show up can
be interpreted as follows: in order to go from $[a_{1},a_{2},x]$
to $[x,a_{2},a_{1}]$, we have to first swap $a_{1}$ with $a_{2}$
(this is where the $(-1)^{|a_{1}|a_{2}|}$ comes from), then swap
$a_{1}$ with $x$ (this is where the $(-1)^{|a_{1}||x|}$ comes from),
and then finally swap $a_{2}$ with $x$ (this is where the $(-1)^{|a_{2}||x|}$
comes from). We then obtain one extra minus sign by swapping terms
in the associator at the final step:
\begin{align*}
[a_{1},a_{2},x] & =(a_{1}a_{2})x-a_{1}(a_{2}x)\\
 & =(-1)^{|a_{1}|a_{2}|}(a_{2}a_{1})x-(-1)^{|a_{2}|||x|}a_{1}(xa_{2})\\
 & =(-1)^{|a_{1}||a_{2}|+|a_{2}||x|+|a_{1}||x|}x(a_{2}a_{1})-(-1)^{|a_{2}||x|+|a_{1}||x|+|a_{1}||a_{2}|}(xa_{2})a_{1}\\
 & =(-1)^{|a_{1}||a_{2}|+|a_{1}||x|+|a_{2}||x|}(x(a_{2}a_{1})-(xa_{2})a_{1})\\
 & =-(-1)^{|a_{1}||a_{2}|+|a_{1}||x|+|a_{2}||x|}[x,a_{2},a_{1}].
\end{align*}
A similar interpretation is also given to (\ref{eq:identity2}) and
(\ref{eq:identity4-1}). For instance, in order to get from $[a_{1},a_{2},x]$
to $[x,a_{1},a_{2}]$, we have to swap $x$ with $a_{2}$ and then
swap $x$ with $a_{1}$ (this is where the $(-1)^{|a_{1}||x|+|a_{2}||x|}$
comes from). We do add an extra minus sign in (\ref{eq:identity4-1})
however since we never swap terms in the associator:
\begin{align*}
(-1)^{|a_{1}||a_{2}|}[a_{2},a_{1},x]+(-1)^{|a_{2}||x|}[a_{1},x,a_{2}] & =(a_{1}a_{2})x-(-1)^{|a_{1}||a_{2}|}a_{2}(a_{1}x)+(-1)^{|a_{2}||x|}(a_{1}x)a_{2}-a_{1}(a_{2}x)\\
 & =(a_{1}a_{2})x-(-1)^{|a_{1}||a_{2}|}a_{2}(a_{1}x)+(-1)^{|a_{1}|a_{2}|}a_{2}(a_{1}x)-a_{1}(a_{2}x)\\
 & =(a_{1}a_{2})x-a_{1}(a_{2}x)\\
 & =[a_{1},a_{2},x].
\end{align*}

\subsubsection{Alternative MDG Modules}

If $X$ is not associative, then we are often interested in knowing
whether or not $X$ satisfies the following weaker property:

\begin{defn}\label{defn} We say $X$ is \textbf{alternative }if $[a,a,x]=0$
for all $a\in A$ and $x\in X$. \end{defn}

~~~~~~In other words, $X$ is alternative if for each $a\in A$
and $x\in X$, we have $a^{2}x=a(ax).$ The reason behind the name
``alternative'' comes from the fact that in the case where $X=A$,
then $A$ is alternative if and only if the associator $[\cdot,\cdot,\cdot]$
is alternating.

\begin{prop}\label{propalternative} Let $a\in A$ and $x\in X$ be
homogeneous.
\begin{enumerate}
\item We have $[a,a,x]=0$ if and only if $[x,a,a]=0$.
\item If $[a,a,x]=0$, then $[a,x,a]=0$. The converse holds if $|a|$ is
odd and $\mathrm{char}\,R\ne2$.
\item If $|a|$ is even, we have $[a,x,a]=0$, and if $|a|$ is odd, we
have $[a,x,a]=(-1)^{|x|}2[a,a,x]$. In particular, if $\mathrm{char}\,R=2$,
we always have $[a,x,a]=0$. 
\end{enumerate}
\end{prop}

\begin{proof} From identities (\ref{eq:identity3}) and (\ref{eq:identity4-1})
we obtain
\begin{align*}
[a,a,x] & =-(-1)^{|a|}[x,a,a]\\{}
[a,x,a] & =(-1)^{|x||a|}(1-(-1)^{|a|})[a,a,x].
\end{align*}
In particular, we see that
\begin{equation}
[a,x,a]=\begin{cases}
=(-1)^{|x|}2[a,a,x]=-(-1)^{|x|}2a(ax) & \text{if }|a|\text{ is odd}\\
0 & \text{if }|a|\text{ is even}
\end{cases}\label{eq:relationshipcommut-1}
\end{equation}
Similarly we have
\begin{equation}
[a,a,x]=\begin{cases}
(-1)^{|x|}\frac{1}{2}[a,x,a] & \text{if }a\text{ is odd}\text{ and }\mathrm{char}\,R\neq2\\
(-1)^{|a|}[x,a,a] & \text{if }a\text{ is even}
\end{cases}\label{eq:relationshipcommut-1-1}
\end{equation}
\end{proof}

\begin{prop}\label{prop} Suppose $A$ is an alternative MDG $R$-algebra.
Then $[a_{1},a_{2},a_{3}]=0$ whenever $|a_{1}|$ and $|a_{3}|$ are
odd. \end{prop}

\begin{proof}\label{proof} Observe that
\begin{align*}
0 & =[a_{1}+a_{3},a_{1}+a_{3},a_{2}]\\
 & =[a_{1},a_{1},a_{2}]+[a_{1},a_{3},a_{2}]+[a_{3},a_{1},a_{2}]+[a_{3},a_{3},a_{2}]\\
 & =[a_{1},a_{3},a_{2}]+[a_{3},a_{1},a_{2}]\\
 & =[a_{1},a_{3},a_{2}]-[a_{1},a_{3},a_{2}]+(-1)^{|a_{2}|}[a_{3},a_{2},a_{1}]\\
 & =(-1)^{|a_{2}|}[a_{3},a_{2},a_{1}]\\
 & =(-1)^{|a_{2}|}[a_{1},a_{2},a_{3}].
\end{align*}

\end{proof}

\begin{example}\label{example} Consider the MDG $R$-algebra $F_{\mathrm{K}}$
given in Example~(\ref{example1}). Then we have $[e_{\sigma},e_{\sigma},e_{\tau}]=0$
for all $\sigma,\tau\in\Delta$, however $F$ is not alternative since
$[e_{1},e_{5},e_{2}]\neq0$. \end{example}

\subsubsection{The Maximal Associative Quotient}

\begin{defn}\label{defn} The \textbf{associator $R$-subcomplex }of
$X$, denoted $[X]$, is the $R$-subcomplex of $X$ given by the
image of the associator of $X$. Thus the underlying graded $R$-module
of $[X]$ is
\[
[X]=\mathrm{span}_{R}\{[a_{1},a_{2},x]\mid a_{1},a_{2}\in A\text{ and }x\in X\},
\]
and the differential of $[X]$ is simply the restriction of the differential
of $X$ to $[X]$. The \textbf{associator $A$-submodule }of $X$,
denoted $\langle X\rangle$, is defined to be the smallest $A$-submodule
of $X$ which contains $[X]$. Observe that
\begin{equation}
a_{1}(a_{2}[a_{3},a_{4},x])=(a_{1}a_{2})[a_{3},a_{4},x]-[a_{1},a_{2},[a_{3},a_{4},x]]\label{eq:jklfdsm-1}
\end{equation}
for all $a_{1},a_{2},a_{3},a_{4}\in A$ and $x\in X$. Using identities
like (\ref{eq:jklfdsm-1}) together with graded-commutativity, one
can show that the underlying graded $R$-module of $\langle X\rangle$
is given by 
\[
\langle X\rangle=\mathrm{span}_{R}\{a_{1}[a_{2},a_{3},x]\mid a_{1},a_{2},a_{3}\in A\text{ and }x\in X\}
\]

The quotient $X^{\mathrm{as}}:=X\slash\langle X\rangle$ is a DG $A$-module
(i.e. an associative MDG $A$-module). We call $X^{\mathrm{as}}$
(together with its canonical quotient map $X\twoheadrightarrow X^{\mathrm{as}}$)
the \textbf{maximal associative quotient }of\textbf{ }$X$. \end{defn}

~~~~~~The maximal associative quotient of $X$ satisfies the
following universal mapping property: 

\begin{prop}\label{prop} Every $\mathrm{MDG}$ $A$-module homomorphism
$\varphi\colon X\to Y$ in which $Y$ is associative factors through
a unique $\mathrm{MDG}$ $A$-module homomorphism $\overline{\varphi}\colon X^{\mathrm{as}}\to Y$,
meaning $\overline{\varphi}\rho=\varphi$ where $\rho\colon X\twoheadrightarrow X^{\mathrm{as}}$
is the canonical quotient map. We express this in terms of a commutative
diagram as below: \begin{equation}\label{equationump}\begin{tikzcd}[row sep =40, column sep = 40]  X \arrow[r, "\rho  "] \arrow[dr, "\varphi ", swap] & X^{\mathrm{as} } \arrow[d, dashed," \overline{\varphi}"] \\ & Y  \end{tikzcd}\end{equation}
\end{prop}

\begin{proof}\label{proof} Indeed, suppose $\varphi\colon X\to Y$
is any MDG $A$-module homomorphism where $Y$ is associative. In
particular, we must have $[X]\subseteq\ker\varphi$, and since $\langle X\rangle$
is the smallest MDG $A$-submodule of $X$ which contains $[X]$,
it follows that $\langle X\rangle\subseteq\ker\varphi$. Thus the
map $\overline{\varphi}\colon X^{\mathrm{as}}\to Y$ given by $\overline{\varphi}(\overline{x}):=\varphi(x)$
where $\overline{x}\in X^{\mathrm{as}}$ is well-defined. Furthermore,
it is easy to see that $\overline{\varphi}$ is an MDG $A$-module
homomorphism and the unique such one which makes the diagram (\ref{equationump})
commute. \end{proof}

\begin{defn}\label{defn} The \textbf{associator homology }of $X$
is the homology of the associator $A$-submodule of $X$. We often
simplify notation and denote the associator homology of $X$ by $\mathrm{H}\langle X\rangle$
instead of $\mathrm{H}(\langle X\rangle)$. We say $X$ is \textbf{homologically
associative} if $\mathrm{H}\langle X\rangle=0$ and we say $X$ is
\textbf{homologically associative in degree} $i$ if $\mathrm{H}_{i}\langle X\rangle=0$.
Similarly we say $X$ is associative in degree if $\langle X\rangle_{i}=0$.
\end{defn}

~~~~~~Clearly, if $X$ is associative, then $X$ is homologically
associative. The converse holds under certain conditions. This is
the first main theorem given in the introduction.

\begin{theorem}\label{theoremhomologyassociator} Assume $R$ is a
local ring with maximal ideal $\mathfrak{m}$ and assume that $\langle X\rangle$
is minimal (meaning $\mathrm{d}\langle X\rangle\subseteq\mathfrak{m}\langle X\rangle$)
and such that each $\langle X\rangle_{i}$ is a finitely generated
$R$-module. If $X$ is associative in degree $i$, then $X$ is associative
in degree $i+1$ if and only if $X$ is homologically associative
in degree $i+1$. In particular, if $\langle X\rangle$ is also bounded
below (meaning $\langle X\rangle_{i}=0$ for $i\ll0$), then $X$
is associative if and only if $X$ is homologically associative. \end{theorem}

\begin{proof} Assume that $X$ is associative in degree $i$. Clearly
if $X$ is associative in degree $i+1$, then it is homologically
associative in degree $i+1$. To show the converse, assume for a contradiction
that $X$ is homologically associative in degree $i+1$ but that it
is not associative in degree $i+1$. In other words, assume
\[
\mathrm{H}_{i+1}\langle X\rangle=0\quad\text{and}\quad\langle X\rangle_{i+1}\neq0.
\]
Then by Nakayama's lemma, we can find homogeneous $a_{1},a_{2},a_{3}\in A$
and homogeneous $x\in X$ such that such that $a_{1}[a_{2},a_{3},x]\notin\mathfrak{m}\langle X\rangle_{i+1}$.
Since $\langle X\rangle_{i}=0$ by assumption, we have $\mathrm{d}(a_{1}[a_{1},a_{2},x])=0$.
Also, since $\langle X\rangle$ is minimal, we have $\mathrm{d}\langle X\rangle\subseteq\mathfrak{m}\langle X\rangle$.
Thus $a_{1}[a_{2},a_{3},x]$ represents a nontrivial element in homology
in degree $i+1$. This is a contradiction. \end{proof}

~~~~~~We are often also interested in the homology of the maximal
associative quotient of $X$ as well. To this end, observe that the
short exact sequence of MDG $A$-modules \begin{center}\begin{tikzcd} 0 \arrow[r] & \langle X \rangle \arrow[r,  ]  & X \arrow[r] & X^{\mathrm{as}} \arrow[r] & 0 \end{tikzcd}\end{center}induces
a sequence of graded $\mathrm{H}(A)$-modules\begin{center}\begin{tikzcd} \mathrm{H}\langle X \rangle \arrow[r] & \mathrm{H}(X) \arrow[r] & \mathrm{H}(X^{\mathrm{as}} ) \arrow[r, " \overline{\mathrm{d}} " ]  & \Sigma \mathrm{H}\langle X \rangle   \arrow[r] & \Sigma \mathrm{H}(X) \end{tikzcd}\end{center}which
is exact at $\mathrm{H}\langle X\rangle$, $\mathrm{H}(X)$, and $\mathrm{H}(X^{\mathrm{as}})$
and where the connecting map $\overline{\mathrm{d}}\colon\mathrm{H}(X^{\mathrm{as}})\to\Sigma\mathrm{H}\langle X\rangle$
is essentially defined in terms of the differential $\mathrm{d}$
of $X$, namely given $\overline{x}\in\mathrm{H}(X^{\mathrm{as}})$,
we set $\overline{\mathrm{d}}\overline{x}=\overline{\mathrm{d}x}$.
In particular, if $\mathrm{H}_{i}(X)=0=\mathrm{H}_{i-1}(X)$, then
$\mathrm{H}_{i}(X^{\mathrm{as}})\cong\mathrm{H}_{i-1}\langle X\rangle$.

\begin{example}\label{example} Assume that $R$ is a local noetherian
ring with maximal ideal $\mathfrak{m}$. Let $I\subseteq\mathfrak{m}$
be an ideal of $R$ and let $F$ be the minimal free resolution of
$R\slash I$ over $R$ and equip $F$ with a multiplication giving
it the structure of an MDG $R$-algebra. Then
\[
\mathrm{H}_{i}(F^{\mathrm{as}})\cong\begin{cases}
R\slash I & \text{if }i=0\\
\mathrm{H}_{i-1}\langle F\rangle & \text{else}
\end{cases}
\]
\end{example}

\subsubsection{Computing Annihilators of the Associator Homology}

In this subsubsection, we assume that $R$ is an integral domain with
quotient field $K$. We further assume that the underlying graded
$R$-module of $A$ is free. Recall that the $A$-scalar multiplication
map $\mu_{\langle X\rangle}\colon A\otimes_{R}\langle X\rangle\to\langle X\rangle$
induces an $\mathrm{H}(A)$-scalar multiplication map $\overline{\mu}_{\langle X\rangle}\colon\mathrm{H}(A)\otimes_{R}\mathrm{H}\langle X\rangle\to\mathrm{H}\langle X\rangle$
which gives $\mathrm{H}\langle X\rangle$ an $\mathrm{H}(A)$-module
structure. In particular, $\mathrm{d}A$ annihilates $\mathrm{H}\langle X\rangle$.
However we can often find more annihilators of $\mathrm{H}\langle X\rangle$
than just the ones contained in $\mathrm{d}A$. Indeed, set
\[
A_{K}=\{a/r\mid a\in A\text{ and }r\in R\backslash\{0\}\}\quad\text{and}\quad B=\{b\in A_{K}\mid b\langle X\rangle\subseteq\langle X\rangle\}.
\]
Then $A_{K}$ is an MDG $K$-algebra and $B$ is an MDG subalgebra
of $A_{K}$ which contains $A$. Furthermore $\langle X\rangle$ is
an MDG $B$-module (in fact $B$ is the largest MDG subalgebra of
$A_{K}$ for which $\langle X\rangle$ is an MDG module over). In
particular, $A\cap\mathrm{d}B$ annihilates $\mathrm{H}\langle X\rangle$.
In general we have
\[
\mathrm{d}A\subseteq A\cap\mathrm{d}B\subseteq A,
\]
where the inclusions may be strict. 

\begin{example}\label{examplefdklds} Consider Example~(\ref{example1})
where $R=\Bbbk[x,y,z,w]$, $\boldsymbol{m}=x^{2},w^{2},zw,xy,y^{2}z^{2}$,
and $F$ is the minimal free resolution of $R\slash\boldsymbol{m}$
over\- $R$. Observe that
\begin{align*}
\frac{e_{1}}{x}[e_{1},e_{5},e_{2}] & =\frac{1}{x}\left([e_{1}^{2},e_{5},e_{2}]-[e_{1},e_{1}e_{5},e_{2}]+[e_{1},e_{1},e_{5}e_{2}]-[e_{1},e_{1},e_{5}]e_{2}\right)\\
 & =-\frac{1}{x}[e_{1},e_{1}e_{5},e_{2}]\\
 & =-\frac{1}{x}[e_{1},yz^{2}e_{14}+xe_{45},e_{2}]\\
 & =-\frac{yz^{2}}{x}[e_{1},e_{14},e_{2}]-[e_{1},e_{45},e_{2}]\\
 & =-[e_{1},e_{45},e_{2}].
\end{align*}
It follows that $\mathrm{d}(e_{1}/x)=x$ annihilates $\mathrm{H}\langle F\rangle$.
Similar calculations likes this shows that $\langle x,y,z,w\rangle$
annihilates $\mathrm{H}\langle F\rangle$. It follows that
\[
\mathrm{H}_{i}\langle F\rangle\cong\begin{cases}
\Bbbk & \text{if }i=3\\
0 & \text{else}
\end{cases}
\]
One can interpret this as saying that the multiplication $\mu$ is
very close to being associative (the failure for $\mu$ to being associative
is reflected in the fact that $\dim_{\Bbbk}(\mathrm{H}\langle F\rangle)=1$).
Note that $\mu$ is not associative in homological degree $4$ since
\[
[e_{1},e_{45},e_{2}]=xyze_{1234}\neq0.
\]
In some sense however, the fact that the associator $[e_{1},e_{45},e_{2}]$
is nonzero isn't really a \emph{new }obstruction to $\mu$ being associative.
Indeed, one could argue that $[e_{1},e_{45},e_{2}]$ being nonzero
is simply a consequence of $[e_{1},e_{5},e_{2}]$ being nonzero. More
generally, in order for a nonzero element $\gamma\in\langle F\rangle$
to be considered an obstruction for $\mu$ to be associative, we should
have $\mathrm{d}\gamma=0$ (otherwise one could argue that $\gamma$
being nonzero is simply a consequence of the associators in $\mathrm{d}\gamma$
being nonzero). Similarly, we shouldn't have $\gamma=\mathrm{d}\gamma'$
(otherwise one could argue that $\gamma$ being nonzero is simply
a consequence of $\gamma'$ being nonzero). Thus the associators which
really do contribute new obstructions for $\mu$ to be associative
should be the ones which represent nonzero elements in homology. This
is how we interpret the associator homology of $F$. In this case,
we have precisely one nontrivial associator $[e_{1},e_{5},e_{2}]$
which represents a nonzero element in homology (all of the other nonzero
associators are derived from the fact that $[e_{1},e_{5},e_{2}]\neq0$).\end{example}

\begin{example}\label{examplefdklds} Consider Example~(\ref{example3})
where $R=\Bbbk[x,y,z,w]$, $\boldsymbol{m}=x^{2},w^{2},zw,xy,y^{2}z,yz^{2}$,
and $F$ is the minimal free resolution of $R\slash\boldsymbol{m}$
of $R$. By performing similar calculations as in Example~(\ref{examplefdklds}),
one can show that
\[
\mathrm{H}_{i}\langle F\rangle\cong\begin{cases}
\Bbbk\oplus\Bbbk & \text{if }i=3\\
0 & \text{else}
\end{cases}
\]
 \end{example}

\subsubsection{Associators up to Homotopy}

Let $I$ be an ideal of $R$ and let $F$ be a free resolution of
$R\slash I$ over $R$. We write $F^{\otimes2}=F\otimes_{R}F$ in
what follows. A chain map $\mu\in F^{\otimes2}\to F$ which lifts
the multiplication map on $R\slash I$ is unique up to homotopy. What
this means is that if $\mu'\in F^{\otimes2}\to F$ is another chain
map which lifts the multiplication map on $R\slash I$, then there
exists a graded $R$-linear map $h\colon F^{\otimes2}\to F$ of degree
one such that $\mu'=\mu_{h}$ where
\begin{equation}
\mu_{h}:=\mu+\mathrm{d}h+h\mathrm{d}.\label{eq:contextclear}
\end{equation}
Notice how we are simplifying notation in (\ref{eq:contextclear})
by lettting $\mathrm{d}$ denote the differentials for both $F^{\otimes2}$
and $F$ where context makes clear which differential $\mathrm{d}$
stands for (for instance, the $\mathrm{d}$ in $\mathrm{d}h$ is the
differential of $F$ and the $\mathrm{d}$ in $h\mathrm{d}$ is the
differential of $F^{\otimes2}$). This notational simplification will
be beneficial when we perform calculations in what follows. 

~~~~~~If both $\mu$ and $\mu_{h}$ are graded-commutative,
then $h\sigma\colon F^{\otimes2}\to F$ must be a chain map of degree
$1$, where $\sigma\colon F^{\otimes2}\to F^{\otimes2}$ is defined
by
\[
\sigma(a_{1}\otimes a_{2})=a_{1}\otimes a_{2}-(-1)^{|a_{1}||a_{2}|}a_{2}\otimes a_{1}
\]
for all homogeneous $a_{1},a_{2}\in F$. Indeed, if both $\mu$ and
$\mu_{h}$ are graded-commutative, then we have
\begin{align*}
\mathrm{d}h\sigma+h\sigma\mathrm{d} & =\mathrm{d}h\sigma+h\mathrm{d}\sigma\\
 & =(\mathrm{d}h+h\mathrm{d})\sigma\\
 & =(\mu_{h}-\mu)\sigma\\
 & =\mu_{h}\sigma-\mu\sigma\\
 & =0-0\\
 & =0.
\end{align*}
Similarly, if both $\mu$ and $\mu_{h}$ are unital, then $h|_{F\otimes1}$
and $h|_{1\otimes F}$ must be chain maps of degree $1$. Next observe
that the associator for $\mu_{h}$ is given by
\begin{equation}
[\cdot]_{\mu_{h}}=[\cdot]_{\mu}+\mathrm{d}H+H\mathrm{d}\label{eq:associatorhomotopy}
\end{equation}
where $H=[\cdot]_{\mu,h}+[\cdot]_{h,\mu_{h}}$. Here, we set
\[
[\cdot]_{\mu,h}=\mu(h\otimes1-1\otimes h)\quad\text{and}\quad[\cdot]_{h,\mu_{h}}=h(\mu_{h}\otimes1-1\otimes\mu_{h}).
\]
Note that additional signs will appear in $[\cdot]_{\mu,h}$ when
applied to elements due to the Koszul sign rule. In particular, if
$a_{1},a_{2}\in F$ are homogeneous, then
\[
(1\otimes h)(a_{1}\otimes a_{2})=(-1)^{|a_{1}|}a_{1}\otimes ha_{2}
\]
since $h$ is graded of degree $1$. We can decompose $[\cdot]_{h,\mu_{h}}$
further as
\[
[\cdot]_{h,\mu_{h}}=[\cdot]_{h,\mu}+[\cdot]_{h,\mathrm{d}h}+[\cdot]_{h,h\mathrm{d}}
\]
where
\[
[\cdot]_{h,\mu}=h(\mu\otimes1-1\otimes\mu),\quad[\cdot]_{h,\mathrm{d}h}=h(\mathrm{d}h\otimes1-1\otimes\mathrm{d}h),\quad\text{and}\quad[\cdot]_{h,h\mathrm{d}}=h(h\mathrm{d}\otimes1-1\otimes h\mathrm{d}).
\]
~~~~~~We now want to use the multiplication constructed in Example~(\ref{example2})
to show that there does not exist any DG algebra structure on that
resolution. In fact, it was already shown that this resolution has
no DG algebra structure on it in \cite{Avr81}, however we prove something
slightly stronger: every MDG algebra structure on that resolution
will be non-associative at a particular triple. This is our second
main theorem from the introduction:

\begin{theorem}\label{theoremsecond} Let $R=\Bbbk[x,y,z,w]$, let
$\boldsymbol{m}=x^{2},w^{2},zw,xy,yz$, and let $F$ be the minimal
free resolution of $R\slash\boldsymbol{m}$ over $R$. Every multiplication
on $F$ is non-associative at the triple $(\varepsilon_{1},\varepsilon_{45},\varepsilon_{2})$.
\end{theorem}

\begin{proof}\label{proof} Let $\mu$ be the multiplication constructed
in Example~(\ref{example2}) and let $\mu_{h}=\mu+\mathrm{d}h+h\mathrm{d}$
be another multiplication on $F$. We claim that $[\varepsilon_{1},\varepsilon_{45},\varepsilon_{5}]_{\mu_{h}}\neq0$.
Indeed, the idea is that on the one hand we have $[\varepsilon_{1},\varepsilon_{45},\varepsilon_{2}]_{\mu}=-x\varepsilon_{12345}$
but on the other hand we have
\[
(\mathrm{d}H+H\mathrm{d})(\varepsilon_{1}\otimes\varepsilon_{45}\otimes\varepsilon_{2})\in IF
\]
where $H$ is the map described in (\ref{eq:associatorhomotopy})
and where $I=\langle x^{2},y,z,w\rangle$. In particular, $[\varepsilon_{1},\varepsilon_{45},\varepsilon_{2}]_{\mu_{h}}\not\equiv0$
modulo $IF$ which implies $[\varepsilon_{1},\varepsilon_{45},\varepsilon_{2}]_{\mu_{h}}\neq0$.
To see this, first note that $\mathrm{d}H(\varepsilon_{1}\otimes\varepsilon_{45}\otimes\varepsilon_{2})=0$,
so we only need to show that
\[
H\mathrm{d}(\varepsilon_{1}\otimes\varepsilon_{45}\otimes\varepsilon_{2})=([\cdot]_{\mu,h}+[\cdot]_{h,\mu}+[\cdot]_{h,\mathrm{d}h}+[\cdot]_{h,h\mathrm{d}})\mathrm{d}(\varepsilon_{1}\otimes\varepsilon_{45}\otimes\varepsilon_{2})\in IF.
\]
Now clearly both $[\cdot]_{h,\mathrm{d}h}\mathrm{d}$ and $[\cdot]_{h,h\mathrm{d}}\mathrm{d}$
land in $\mathfrak{m}^{2}F\subseteq IF$ where $\mathfrak{m}=\langle x,y,z,w\rangle$
since $F$ is minimal and the differential shows up twice in each
case. Next note in $F\slash IF$ we have
\begin{align*}
[\cdot]_{h,\mu}\mathrm{d}(\varepsilon_{1}\otimes\varepsilon_{45}\otimes\varepsilon_{2}) & \equiv x^{2}[1\otimes\varepsilon_{45}\otimes\varepsilon_{2}]_{h,\mu}-x[\varepsilon_{1}\otimes\varepsilon_{5}\otimes\varepsilon_{2}]_{h,\mu}+z[\varepsilon_{1}\otimes\varepsilon_{4}\otimes\varepsilon_{2}]_{h,\mu}+w^{2}[\varepsilon_{1}\otimes\varepsilon_{45}\otimes1]_{h,\mu}\\
 & \equiv-x[\varepsilon_{1}\otimes\varepsilon_{5}\otimes\varepsilon_{2}]_{h,\mu}\\
 & \equiv-xh((z\varepsilon_{14}+x\varepsilon_{45})\otimes\varepsilon_{2}-\varepsilon_{1}\otimes(z\varepsilon_{23}+y\varepsilon_{35}))\\
 & \equiv0.
\end{align*}
Similarly in $F\slash IF$ we have
\begin{align*}
[\cdot]_{\mu,h}\mathrm{d}(\varepsilon_{1}\otimes\varepsilon_{45}\otimes\varepsilon_{2}) & \equiv x^{2}[1\otimes\varepsilon_{45}\otimes\varepsilon_{2}]_{\mu,h}-x[\varepsilon_{1}\otimes\varepsilon_{5}\otimes\varepsilon_{2}]_{\mu,h}+z[\varepsilon_{1}\otimes\varepsilon_{4}\otimes\varepsilon_{2}]_{\mu,h}+w^{2}[\varepsilon_{1}\otimes\varepsilon_{45}\otimes1]_{\mu,h}\\
 & \equiv-x[\varepsilon_{1}\otimes\varepsilon_{5}\otimes\varepsilon_{2}]_{\mu,h}\\
 & \equiv0
\end{align*}
where we used the fact that $\varepsilon_{1}F_{3}\in\mathfrak{m}F_{4}$
and $\varepsilon_{2}F_{3}\in\mathfrak{m}F_{4}$. \end{proof}

\subsection{Multiplicators}

Having discussed associators, we now wish to discuss multiplicators.
Throughout this subsection, let $A$ be an MDG $R$-algebra, let $X$
be and $Y$ be MDG $A$-modules, and let $\varphi\colon X\to Y$ be
a chain map.

\begin{defn}\label{defnmulthompermcomp} The are two types of multiplicators
were are interested in:
\begin{enumerate}
\item The \textbf{multiplicator }of $\varphi$ is the chain map, denoted
$[\cdot]_{\varphi}$, from $A\otimes_{R}X$ to $Y$ defined by
\[
[\cdot]_{\varphi}:=\varphi\mu-\mu(1\otimes\varphi).
\]
Note that we use $\mu$ to denote both $A$-scalar multiplications
$\mu_{X}$ and $\mu_{Y}$ where context makes clear which multiplication
$\mu$ refers to. We denote by $[\cdot,\cdot]_{\varphi}\colon A\times X\to Y$
(or more simply by $[\cdot,\cdot]$ if context is clear) to be the
unique graded $R$-bilinear map which corresponds to $[\cdot]_{\varphi}$
(in order to avoid confusion with the associator, we will \emph{always
}keep $\varphi$ in the subscript of $[\cdot]_{\varphi}$). Thus we
have
\[
[a\otimes x]_{\varphi}=\varphi(ax)-a\varphi(x)=[a,x]
\]
for all $a\in A$ and $x\in X$. We say $\varphi$ is \textbf{multiplicative
}if $[\cdot]_{\varphi}=0$. 
\item The \textbf{$2$-multiplicator }of $\varphi$ is the chain map, denoted
$[\cdot]_{\varphi}^{(2)}$, from $A\otimes_{R}A\otimes_{R}X$ to $Y$
defined by
\[
[\cdot]_{\varphi}^{(2)}:=\varphi[\cdot]_{\mu}-[\cdot]_{\mu}(1\otimes1\otimes\varphi)
\]
where we write $[\cdot]_{\mu}$ to denote both the associator of $X$
and the associator $Y$ where context makes clear which multiplication
$\mu$ refers to. We denote by $[\cdot,\cdot,\cdot]_{\varphi}\colon A\times X\to Y$
to be the unique graded $R$-bilinear map which corresponds to $[\cdot]_{\varphi}^{(2)}$
(in order to avoid confusion with the associator, we will \emph{always
}keep $\varphi$ in the subscript of $[\cdot,\cdot,\cdot]_{\varphi}$).
Thus we have
\[
[a_{1}\otimes a_{2}\otimes x]_{\varphi}^{(2)}=\varphi([a_{1},a_{2},x])-[a_{1},a_{2},\varphi(x)]=[a_{1},a_{2},x]_{\varphi}
\]
for all $a_{1},a_{2}\in A$ and $x\in X$. We say $\varphi$ is \textbf{$2$-multiplicative
}if $[\cdot]_{\varphi}^{(2)}=0$.
\end{enumerate}
\end{defn}

~~~~~~Let $A$ and $B$ be MDG $R$-algebras and let $\varphi\colon A\to B$
be a chain map such that $\varphi(1)=1$. Recall that we view $B$
as an $A$-module via the $A$-scalar multiplication map defined by
$a\cdot b=\varphi(a)b$. In this case, the multiplicator of $\varphi$
is just a special case of the usual associator of $B$ viewed as an
$A$-module. Indeed, we have
\begin{align*}
[a_{1},a_{2},1] & =(a_{1}a_{2})\cdot1-a_{1}\cdot(a_{2}\cdot1)\\
 & =\varphi(a_{1}a_{2})-\varphi(a_{1})\varphi(a_{2})\\
 & =\varphi(a_{1}a_{2})-a_{1}\cdot\varphi(a_{2})\\
 & =[a_{1},a_{2}]
\end{align*}
for all $a_{1},a_{2}\in A$. In particular, if $B$ is associative
as an $A$-module, then $\varphi\colon A\to B$ is multiplicative.
The converse on the other hand need not hold as can be seen in the
following example:

\begin{example}\label{examplemultiplicative} We continue with Example~(\ref{example1})
where $R=\Bbbk[x,y,z,w]$, $\boldsymbol{m}=x^{2},w^{2},zw,xy,y^{2}z^{2}$,
and $F$ is the minimal free resolution of $R\slash\boldsymbol{m}$
over $R$. Let $\boldsymbol{m}'=x^{2},w^{2},y^{2}z^{2}$ and let $E'$
be the Koszul algebra which resolves $R\slash\boldsymbol{m}'$ over
$R$. We denote the standard homogeneous basis of $E'$ by $e_{\sigma}'$
and we denote the standard homogeneous basis of $F$ by $e_{\sigma}$.
Choose a chain map $\iota'\colon E'\to F$ which lifts the projection
$R\slash\boldsymbol{m}'\to R\slash\boldsymbol{m}$ such that $\iota'$
is unital and respects the multigrading. Then $\iota'$ being a chain
map together with the fact that it is unital and respects the multigrading
forces us to have
\begin{align*}
\iota'(e_{1}') & =e_{1} & \iota'(e_{12}') & =e_{12}\\
\iota'(e_{2}') & =e_{2} & \iota'(e_{13}') & =yz^{2}e_{14}+xe_{45}\\
\iota'(e_{3}') & =e_{5} & \iota'(e_{23}') & =y^{2}ze_{23}+we_{35}.
\end{align*}
On the other hand, $\iota'$ can be defined at $e_{123}'$ in two
possible ways. Assume that it is defined by
\[
\iota'(e_{123}')=yz^{2}e_{124}+xyze_{234}-xwe_{345}.
\]
We can picture $\iota'(E')$ inside of $F$ as being supported on
the red-shaded subcomplex below: 
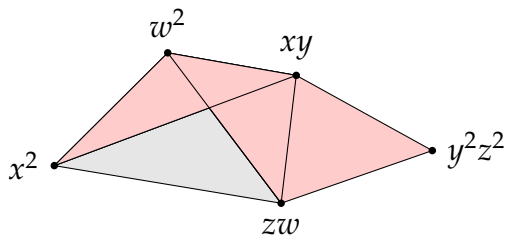
\begin{center}\begin{tikzpicture}[scale=1]

\draw[fill=gray!20] (0,0) -- (3,-0.5) -- (3.2,1.2)-- (0,0); 
\draw[fill=gray!20] (0,0) -- (1.5,1.5) -- (3.2,1.2)-- (0,0); 
\draw[] (1.5,1.5) -- (3,-0.5);

\draw[fill=red!20] (3,-0.5) -- (5,0.2) -- (3.2,1.2);
\draw[fill=red!20] (3,-0.5) -- (1.5,1.5) -- (3.2,1.2);
\draw[] (3,-0.5) -- (3.2,1.2);
\draw[fill=red!20] (0,0) -- (1.5,1.5) -- (3.2,1.2);

\node[circle, fill=black, inner sep=1pt, label=left:$x^2 $] (a) at (0,0) {};
\node[circle, fill=black, inner sep=1pt, label=above:$w^2 $] (b) at (1.5,1.5) {};
\node[circle, fill=black, inner sep=1pt, label=below:$zw $] (c) at (3,-0.5) {};
\node[circle, fill=black, inner sep=1pt, label=above:$xy $] (d) at (3.2,1.2) {}; 
\node[circle, fill=black, inner sep=1pt, label=right:$y^2 z^2 $] (e) at (5,0.2) {};

\draw[color=black!100] (1.5,1.5) -- (3,-0.5);
\draw[color=black!100] (0,0) -- (3.2,1.2);

\end{tikzpicture} \end{center}We claim that $\iota'$ is \emph{not }multiplicative. To see this,
assume for a contradiction that it was multiplicative. Then we would
have
\begin{align*}
0 & =\iota'(0)\\
 & =\iota'([e_{1}',e_{2}',e_{3}'])\\
 & =[\iota'(e_{1}'),\iota'(e_{2}'),\iota'(e_{3}')]\\
 & =[e_{1},e_{2},e_{5}]\\
 & \neq0,
\end{align*}
which is a contradiction. Next let $\boldsymbol{m}''=x^{2},w^{2},zw,xy$
and let $T''$ be the Taylor algebra which resolves $R\slash\boldsymbol{m}''$
over $R$. We denote the standard homogeneous basis of $T''$ by $e_{\sigma}''$.
Choose a comparison map $\iota''\colon T''\to F$ which lifts the
projection $R\slash\boldsymbol{m}''\to R\slash\boldsymbol{m}$ such
that $\iota''$ is unital and respects the multigrading. Then $\iota''$
being a chain map together with the fact that it is unital and multigraded
forces us to have $\iota''(e_{\sigma}'')=e_{\sigma}$ for all $\sigma$.
We can picture $\iota''(T'')$ inside of $F$ as being supported on
the blue-shaded subcomplex  below: 
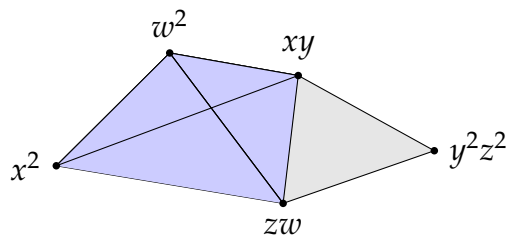
\begin{center}\begin{tikzpicture}[scale=1]

\draw[fill=gray!20] (0,0) -- (3,-0.5) -- (3.2,1.2)-- (0,0); 
\draw[fill=gray!20] (0,0) -- (1.5,1.5) -- (3.2,1.2)-- (0,0); 
\draw[fill=gray!20] (3,-0.5) -- (5,0.2) -- (3.2,1.2);
\draw[] (1.5,1.5) -- (3,-0.5);

\draw[fill=blue!20] (0,0) -- (1.5,1.5) -- (3.2,1.2);
\draw[fill=blue!20] (3,-0.5) -- (1.5,1.5) -- (3.2,1.2);
\draw[] (3,-0.5) -- (3.2,1.2);
\draw[fill=blue!20] (0,0) -- (1.5,1.5) -- (3,-0.5);

\node[circle, fill=black, inner sep=1pt, label=left:$x^2 $] (a) at (0,0) {};
\node[circle, fill=black, inner sep=1pt, label=above:$w^2 $] (b) at (1.5,1.5) {};
\node[circle, fill=black, inner sep=1pt, label=below:$zw $] (c) at (3,-0.5) {};
\node[circle, fill=black, inner sep=1pt, label=above:$xy $] (d) at (3.2,1.2) {}; 
\node[circle, fill=black, inner sep=1pt, label=right:$y^2 z^2 $] (e) at (5,0.2) {};

\draw[color=black!100] (1.5,1.5) -- (3,-0.5);
\draw[color=black!100] (0,0) -- (3.2,1.2);

\end{tikzpicture} \end{center}This time it is easy to check that $\iota''$ \emph{is} multiplicative.
However notice that $F$ is \emph{not }associative as a $T''$-module
since $[e_{1},e_{2},e_{5}]\neq0$. \end{example}

\begin{example}\label{example} Continuing with the notation as in
Example~(\ref{example1}), let $T$ be the Taylor algebra resolution
of $R\slash\boldsymbol{m}$ over $R$. We denote the Taylor multiplication
on $T$ by $\nu$. Recall that the multiplication $\mu$ on $F$ described
in Example~(\ref{example1}) arises from the Taylor multiplication
in the sense that there is a projection $\pi\colon T\to F$ such that
$\mu=\pi\nu\iota^{\otimes2}$ where $\iota\colon F\to T$ is the inclusion
map. Observe that
\begin{align*}
[e_{1},e_{25}]_{\pi} & =\pi(e_{1}\star_{\nu}e_{25})-\pi(e_{1})\star_{\mu}\pi(e_{25})\\
 & =\pi(e_{125})-e_{1}\star_{\mu}(y^{2}ze_{23}+we_{35})\\
 & =yz^{2}e_{124}+xyze_{234}+xwe_{345}-y^{2}ze_{123}-yzwe_{134}-xwe_{345}\\
 & =-yz\mathrm{d}(e_{1234})\\
 & =[e_{1},e_{5},e_{2}]_{\mu}\\
 & \neq0.
\end{align*}
Thus $\pi\colon T\to F$ is not multiplicative. \end{example}

\subsubsection{Multiplicator Identities}

We want to familiarize ourselves with the multiplicator of $\varphi\colon X\to Y$,
so in this subsubsection we collect together some identities which
the multiplicator satisfies:
\begin{itemize}
\item For all $a\in A$ homogeneous and $x\in X$, we have the Leibniz rule:
\[
\mathrm{d}[a,x]=[\mathrm{d}a,x]+(-1)^{|a|}[a,\mathrm{d}x].
\]
\item For all $a\in A$ homogeneous and $x\in X$ homogeneous, we have
\begin{equation}
[a,x]=(-1)^{|a||x|}[x,a].\label{eq:identity3-1-1-1}
\end{equation}
\item For all $a_{1},a_{2}\in A$ and $x\in X$, we have
\begin{equation}
a_{1}[a_{2},x]-[a_{1}a_{2},x]+[a_{1},a_{2}x]=[a_{1},a_{2},x]_{\varphi}\label{eq:identity1-1-1-1}
\end{equation}
\end{itemize}
Furthermore, if $Z$ is another MDG $A$-module and $\psi\colon Y\to Z$
is another chain map, then for all $a\in A$ and $x\in X$, we have
\begin{equation}
[a,x]_{\psi\varphi}=\psi([a,x]_{\varphi})+[a,\varphi x]_{\psi}\label{eq:identityjklfds}
\end{equation}
Next let $A$ and $B$ be MDG $R$-algebras and let $\varphi\colon A\to B$
be a chain map such that $\varphi(1)=1$. Then we can rewrite (\ref{eq:identity1-1-1-1})
as follows: for all $a_{1},a_{2},a_{3}\in A$, we have
\begin{equation}
\varphi(a_{1})[a_{2},a_{3}]-[a_{1}a_{2},a_{3}]+[a_{1},a_{2}a_{3}]-[a_{1},a_{2}]\varphi(a_{3})=[\varphi a_{1},\varphi a_{2},\varphi a_{3}]-\varphi([a_{1},a_{2},a_{3}])\label{eq:identity1-1-1}
\end{equation}
Indeed, this follows from the fact that
\[
[\varphi a_{1},\varphi a_{2},\varphi a_{3}]=[a_{1},a_{2},\varphi a_{3}]-[a_{1},a_{2}]\varphi(a_{3}).
\]
Furthermore, in this case we also have
\begin{equation}
[a,a]_{\varphi}=0\label{eq:identity12345}
\end{equation}
for all $a\in A$ where $|a|$ is odd. 

\subsubsection{The Maximal Multiplicative Quotient}

The \textbf{multiplicator complex }of $\varphi$, denoted $[Y]_{\varphi}$,
is the $R$-subcomplex of $Y$ given by $[Y]_{\varphi}:=\mathrm{im}\,[\cdot]_{\varphi}$,
so the underlying graded module of $[Y]_{\varphi}$
\[
[Y]_{\varphi}:=\mathrm{span}_{R}\{[a,x]_{\varphi}\mid a\in A\text{ and }x\in X\},
\]
and the differential of $[Y]_{\varphi}$ is simply the restriction
of the differential of $Y$ to $[Y]_{\varphi}$. In order to avoid
confusion with the associator complex, we will always write $\varphi$
in the subscript of $[Y]_{\varphi}$. Even though the multiplicator
complex of $\varphi$ is closed under the differential, it need not
be closed under $A$-scalar multiplication. In other words, if $a_{1},a_{2}\in A$
and $x\in X$, then it need not be the case that $a_{1}[a_{2},x]_{\varphi}\in[Y]_{\varphi}$.
We denote by $\langle Y\rangle_{\varphi}$ to be the MDG $A$-submodule
of $Y$ generated by $[Y]_{\varphi}$. In other words, $\langle Y\rangle_{\varphi}$
is the smallest MDG $A$-submodule of $Y$ which contains $[Y]_{\varphi}$.
Unlike the associator submodule, the multiplicator submodule is difficult
to describe in terms of an $R$-span of elements. Indeed, as a first
guess, one might think that $\langle Y\rangle_{\varphi}$ is given
by
\begin{equation}
\mathrm{span}_{R}\{[a,x]_{\varphi}\mid a\in A\text{ and }x\in X\}.\label{eq:multiplci-1}
\end{equation}
However this is clearly incorrect in general as we may need to adjoin
elements of the form $a_{1}[a_{2},x]$ to (\ref{eq:multiplci-1}).
As a second guess, one might think that $\langle Y\rangle_{\varphi}$
is given by
\begin{equation}
\mathrm{span}_{R}\{a_{1}[a_{2},x]_{\varphi}\mid a_{1},a_{2}\in A\text{ and }x\in X\}.\label{eq:multiplci}
\end{equation}
However this is not correct in general either since the identity
\[
a_{1}(a_{2}[a_{3},x]_{\varphi})=(a_{1}a_{2})[a_{3},x]_{\varphi}-[a_{1},a_{2},[a_{3},x]_{\varphi}]
\]
tells us that should really adjoin elements of the form $a_{1}[a_{2},a_{3},[a_{4},x]]$
to (\ref{eq:multiplci}) as well. As a third guess, one might think
that $\langle Y\rangle_{\varphi}$ is given by
\begin{equation}
\mathrm{span}_{R}\{a_{1}[a_{2},x]_{\varphi},\,a_{1}[a_{2},a_{3},[a_{4},x]_{\varphi}]\mid a_{1},a_{2},a_{3},a_{4}\in A\text{ and }x\in X\}.\label{eq:multiplicatorcomplexrspan-1}
\end{equation}
Again this is not correct in general since the identity
\[
a_{1}(a_{2}[a_{3},a_{4},[a_{5},x]_{\varphi}])=(a_{1}a_{2})[a_{3},a_{4},[a_{5},x]]-[a_{1},a_{2},[a_{3},a_{4},[a_{5},x]_{\varphi}]].
\]
tells us that we should really adjoin elements of the form $a_{1}[a_{2},a_{3},[a_{4},a_{5},[a_{6},x]_{\varphi}]]$
to (\ref{eq:multiplicatorcomplexrspan-1}) as well. The problem continues
getting worse with no end in sight. It turns out however, that if
$\varphi$ is $2$-multiplicative, then $\langle Y\rangle_{\varphi}$
given by (\ref{eq:multiplci-1}).

\begin{prop}\label{prop} If $\varphi$ is $2$-multiplicative, then
for all $a_{1},a_{2},a_{3}\in A$ and $x\in X$ we have
\begin{equation}
a_{1}[a_{2},x]_{\varphi}=[a_{1}a_{2},x]_{\varphi}-[a_{1},a_{2}x]_{\varphi}\quad\text{and}\quad[a_{1},a_{2},[a_{3},x]_{\varphi}]=[[a_{1},a_{2},a_{3}],x]_{\varphi}-[a_{1},[a_{2},a_{3},x]]_{\varphi}.\label{eq:2multiident}
\end{equation}
In particular, $\langle Y\rangle_{\varphi}$ is given by (\ref{eq:multiplci-1}).
\end{prop}

\begin{proof}\label{proof} A straightforward calculation yields
\begin{align*}
a_{1}[a_{2},a_{3},x]_{\varphi} & =[a_{1}a_{2},a_{3},x]_{\varphi}-[a_{1},a_{2}a_{3},x]_{\varphi}+[a_{1},a_{2},a_{3}x]_{\varphi}-[[a_{1},a_{2},a_{3}],x]_{\varphi}+[a_{1},[a_{2},a_{3},x]]_{\varphi}-[a_{1},a_{2},[a_{3},x]_{\varphi}].
\end{align*}
Using this identity together with the identity (\ref{eq:identity1-1-1-1}),
we see that if $\varphi$ is $2$-multiplicative, then we obtain (\ref{eq:2multiident}).
This implies all elements of the form $a_{1}[a_{2},x]$ and $a_{1}[a_{2},a_{3},[a_{4},x]]$
belong to (\ref{eq:multiplci-1}). An easy induction argument shows
that $\langle Y\rangle_{\varphi}$ is given by (\ref{eq:multiplci-1}).
\end{proof}

\section{The Associator Functor}

Let $X$ and $Y$ be MDG $A$-modules and let $\varphi\colon X\to Y$
be a chain map. If $\varphi$ is multiplicative, then observe that
for all $a_{1},a_{2},a_{3}\in A$ and $x\in X$, we have 
\begin{equation}
\varphi(a_{1}[a_{2},a_{3},x])=a_{1}[a_{2},a_{3},\varphi x].\label{eq:asfunctorid}
\end{equation}
Thus $\varphi$ restricts to an MDG $A$-module homomorphism $\varphi\colon\langle X\rangle\to\langle Y\rangle$.
In particular, we obtain a functor from the category of MDG $A$-module
to itself which sends an MDG $A$-module $X$ to the MDG associator
submodule $\langle X\rangle$ and which sends an MDG $A$-module homomorphism
$\varphi\colon X\to Y$ to its restriction $\varphi|_{\langle X\rangle}\colon\langle X\rangle\to\langle Y\rangle$.
We call this the \textbf{associator functor}.

\subsection{Failure of Exactness}

The associator functor need not be exact. Indeed, let \begin{equation}\label{diagram23423}\begin{tikzcd} 0 \arrow[r] & X \arrow[r,"\varphi "]  & Y \arrow[r," \psi "] & Z  \arrow[r] & 0 \end{tikzcd}\end{equation}be
a short exact sequence of MDG $A$-modules. Then we obtain an induced
sequence of MDG $A$-modules \begin{equation}\label{diagramfjklsdfmklsd}\begin{tikzcd} 0 \arrow[r] & \langle X \rangle  \arrow[r," \varphi   "]  & \langle Y \rangle \arrow[r,"  \psi   "] & \langle Z \rangle \arrow[r] & 0 \end{tikzcd}\end{equation}which
is exact at $\langle X\rangle$ and $\langle Z\rangle$ but not necessarily
exact at $\langle Y\rangle$. In order to ensure exactness of (\ref{diagramfjklsdfmklsd}),
we need to place a condition on (\ref{diagram23423}). This leads
us to consider the following definition:

\begin{defn}\label{defn} Let $X$ be an MDG $A$-submodule of $Y$.
We say $Y$ is an \textbf{associative extension }of $X$ if
\[
\langle X\rangle=X\cap\langle Y\rangle.
\]
\end{defn}

~~~~~~It is easy to see that (\ref{diagramfjklsdfmklsd}) is
a short exact sequence of MDG $A$-modules if and only if $Y$ is
an associative extension of $\varphi(X)$. In this case, we obtain
a long exact sequence in homology: \begin{equation}\label{diagram3}\begin{tikzcd}[row sep=40]  && \cdots \arrow[r] \arrow[d, phantom, ""{coordinate, name=Z'}] & \mathrm{H}_{i+1} \langle Z   \rangle \arrow[dll,  swap, rounded corners, to path={ -- ([xshift=2ex]\tikztostart.east) |- (Z') [near end]\tikztonodes -| ([xshift=-2ex]\tikztotarget.west) -- (\tikztotarget)}] \\  & \mathrm{H}_{i} \langle X \rangle \arrow[r] & \mathrm{H}_{i} \langle Y \rangle \arrow[r] \arrow[d, phantom, ""{coordinate, name=Z}] & \mathrm{H}_{i} \langle Z \rangle \arrow[dll,  swap, rounded corners, to path={ -- ([xshift=2ex]\tikztostart.east) |- (Z) [near end]\tikztonodes -| ([xshift=-2ex]\tikztotarget.west) -- (\tikztotarget)}] \\ & \mathrm{H}_{i-1} \langle X \rangle \arrow[r] & \cdots \end{tikzcd}\end{equation}An
immediate consequence of this long exact sequence is the following
theorem:

\begin{theorem}\label{theorem3.1} Let $X$ be an MDG $A$-module
and suppose $Y$ is an associative extension of $X$. Then $Y$ is
homologically associative if and only if $X$ and $Y\slash X$ are
homologically associative. \end{theorem}

\subsection{An Application of the Long Exact Sequence}

In this subsection, we give an application of the long exact sequence
(\ref{diagram3}). Assume that $(R,\mathfrak{m})$ is a local ring.
Let $I\subseteq\mathfrak{m}$ be an ideal of $R$, let $F$ be the
minimal free resolution of $R\slash I$ over $R$, and let $r\in\mathfrak{m}$
be an $(R\slash I)$-regular element. Then the mapping cone $F+eF$
is the minimal free resolution of $R\slash\langle I,r\rangle$ over
$R$. Here, $e$ is thought of as an exterior variable of degree $1$,
and the differential of the mapping cone is given by
\[
\mathrm{d}(a+eb)=\mathrm{d}(a)+rb-e\mathrm{d}(b)
\]
for all $a,b\in F$. Now equip $F$ with a multiplication $\mu$ giving
it the structure of an MDG algebra. We give $F+eF$ the structure
of an MDG $R$-algebra by extending the multiplication on $F$ to
a multiplication on $F+eF$ by setting
\[
(a+eb)(c+ed)=ac+e(bc+(-1)^{|a|}ad)
\]
for all $a,b,c,d\in F$. In particular, note that $(eb)c=e(bc)$ for
all $b,c\in F$, so $e$ belongs to the nucleus of $F+eF$. We denote
by $\iota\colon F\to F+eF$ to be the inclusion map. We can view $F+eF$
either as an MDG $F$-module or as an MDG $R$-algebra, thus we potentially
have two different associator complexes to consider. It turns out
however that these give rise to the same $R$-complex since $e$ is
in the nucleus of $F+eF$. This is the third main theorem from the
introduction:

\begin{theorem}\label{theorem} Let $\langle F+eF\rangle_{F}$ be
the associator $F$-submodule of $F+eF$ and let $\langle F+eF\rangle$
be the associator $(F+eF)$-ideal of $F+eF$. Then
\begin{equation}
\langle F+eF\rangle_{F}=\langle F\rangle+e\langle F\rangle=\langle F+eF\rangle.\label{eq:idfjklss}
\end{equation}
In particular, $F+eF$ is an associative extension of $F$. More generally,
suppose $\boldsymbol{r}=r_{1},\dots,r_{m}$ is a maximal $(R\slash I)$-regular
sequence contained in $\mathfrak{m}$. We set
\[
F+\boldsymbol{e}F=F+\sum_{i=1}^{m}e_{i}F
\]
to be minimal $R$-free resolution of $R\slash\langle I,\boldsymbol{r}\rangle$
obtained by iterating the mapping cone construction as above, where
$e_{i}$ is an exterior variable of degree $1$ which satisfies $\mathrm{d}e_{i}=r_{i}$,
and where we extend the multiplication of $F$ to a multiplication
on $F+\boldsymbol{e}F$ by extending it from $F+\sum_{i=1}^{k}e_{i}F$
to $F+\sum_{i=1}^{k+1}e_{i}F$ for each $1\leq k<m$ as above. Then
\begin{equation}
\langle F+\boldsymbol{e}F\rangle_{F}=\langle F\rangle+\boldsymbol{e}\langle F\rangle=\langle F+\boldsymbol{e}F\rangle\label{eq:dfsd}
\end{equation}
where we set $\boldsymbol{e}\langle F\rangle:=\sum_{i=1}^{m}e_{i}\langle F\rangle$.
In particular, $F+\boldsymbol{e}F$ is an associative extension of
$F$. \end{theorem}

\begin{proof}\label{proof} Since $e$ is in the nucleus, we have
$e[a,b,c]=[ea,b,c]$ for all $a,b,c\in F$. Similarly we have
\begin{align*}
[a,b,ec] & =-(-1)^{|a||b|+|a||ec|+|ec||b|}[ec,b,a]\\
 & =-(-1)^{|a||b|+|a||c|+|b||c|}[ec,b,a]\\
 & =-(-1)^{|a||b|+|a||c|+|b||c|}e[c,b,a]\\
 & =e[a,b,c]
\end{align*}
for all $a,b,c\in F$. Similarly we have
\begin{align*}
[a,eb,c] & =-(-1)^{|a||eb|+|a||c|}[eb,c,a]-(-1)^{|eb||c|+|a||c|}[c,a,eb]\\
 & =e(-(-1)^{|a||eb|+|a||c|}[b,c,a]-(-1)^{|eb||c|+|a||c|}[c,a,b])\\
 & =e[a,b,c]
\end{align*}
for all $a,b,c\in F$. Thus we have
\begin{align*}
(a+ea')[b+eb',c+ec',d+ed'] & =(a+ea')[b,c,d]+(a+ea')(e[b',c',d'])\\
 & =a[b,c,d]+ea'[b,c,d]+(-1)^{|a|}ea[b',c',d']\\
 & =a[b,c,d]+e(a'[b,c,d]+(-1)^{|a|}a[b',c',d'])
\end{align*}
for all $a,b,c,d,a',b',c',d'\in F$. Thus we obtain (\ref{eq:idfjklss}).
To see why (\ref{eq:idfjklss}) implies $F+eF$ is an associative
extension of $F$, note that
\[
F\cap\langle F+eF\rangle=F\cap(\langle F\rangle+e\langle F\rangle)=\langle F\rangle.
\]
The last part of the theorem follows from induction.\end{proof} 

\begin{theorem}\label{theorem} Let $\varepsilon=\inf\langle F\rangle$
and let $\delta=\sup\langle F\rangle$. Then $\inf\langle F+eF\rangle=\varepsilon$
and
\begin{equation}
\sup\langle F+eF\rangle=\begin{cases}
\delta & \text{if }r\text{ is }\mathrm{H}_{\delta}\langle F\rangle\text{-regular}\\
\delta+1 & \text{otherwise}
\end{cases}\label{eq:uha(F+eF)}
\end{equation}
Moreover, we have a short exact sequence of $R\slash\langle I,r\rangle$-modules
\begin{equation}\label{applicationlong2}\begin{tikzcd} 0 \arrow[r] & \mathrm{H}_i \langle F \rangle \slash r \mathrm{H}_i \langle F \rangle \arrow[r] & \mathrm{H}_i \langle F + eF \rangle \arrow[r] & 0 : _{ \mathrm{H}_{i-1} \langle F \rangle } r \arrow[r] & 0   \end{tikzcd}\end{equation}
for each $i\in\mathbb{Z}$. In particular, we have an isomorphism
of $R\slash\langle I,r\rangle$-modules
\[
\mathrm{H}_{\varepsilon}\langle F\rangle\slash r\mathrm{H}_{\varepsilon}\langle F\rangle\cong\mathrm{H}_{\varepsilon}\langle F+eF\rangle.
\]

\end{theorem}

\begin{proof}\label{proof} Since $F+eF$ is an associative extension
of $F$, we obtain a long exact sequence in homology: \begin{equation}\label{applicationlong}\begin{tikzcd}[row sep=40,ampersand replacement=\&]  \& \& \cdots \arrow[r] \arrow[d, phantom, ""{coordinate, name=Z'}] \& \mathrm{H}_{i} \langle F \rangle  \arrow[dll, "r" , swap, rounded corners, to path={ -- ([xshift=2ex]\tikztostart.east) |- (Z') [near end]\tikztonodes -| ([xshift=-2ex]\tikztotarget.west) -- (\tikztotarget)}] \\  \& \mathrm{H}_{i} \langle F \rangle \arrow[r] \& \mathrm{H}_{i} \langle F+eF \rangle \arrow[r] \arrow[d, phantom, ""{coordinate, name=Z}] \& \mathrm{H}_{i-1} \langle F \rangle \arrow[dll,  "r" ,swap, rounded corners, to path={ -- ([xshift=2ex]\tikztostart.east) |- (Z) [near end]\tikztonodes -| ([xshift=-2ex]\tikztotarget.west) -- (\tikztotarget)}] \\ \& \mathrm{H}_{i-1} \langle F \rangle \arrow[r] \& \cdots \end{tikzcd}\end{equation}
We obtain (\ref{applicationlong2}) as well as (\ref{applicationlong})
from this long exact sequence. We obtain $\mathrm{lha}(F+eF)=\varepsilon$
from the long exact sequence together with an application of Nakayama's
lemma. \end{proof}

\begin{cor}\label{cor} Suppose $\boldsymbol{r}=r_{1},\dots,r_{m}$
is a maximal $(R\slash I)$-regular sequence contained in $\mathfrak{m}$
and let $F+\boldsymbol{e}F$ be the corresponding $R$-free resolution
of $R\slash\langle I,\boldsymbol{r}\rangle$ obtained by iterating
the mapping cone construction. Then we obtain a short exact sequence
of $R\slash\langle I,\boldsymbol{r}\rangle$-modules \begin{equation}\label{applicationlong2}\begin{tikzcd} 0 \arrow[r] & \mathrm{H}_i \langle F \rangle \slash \boldsymbol{r} \mathrm{H}_i \langle F \rangle \arrow[r] & \mathrm{H}_i \langle F + \boldsymbol{e} F \rangle \arrow[r] & 0 : _{ \mathrm{H}_{i-1} \langle F \rangle } \boldsymbol{r} \arrow[r] & 0   \end{tikzcd}\end{equation}
In particular, have an isomorphism of $R\slash\langle I,\boldsymbol{r}\rangle$-modules:
\[
\mathrm{H}_{\varepsilon}\langle F\rangle\slash\boldsymbol{r}\mathrm{H}_{\varepsilon}\langle F\rangle\cong\mathrm{H}_{\varepsilon}\langle F+\boldsymbol{e}F\rangle.
\]
We also have the length formula:
\[
\ell(\mathrm{H}_{i}\langle F+\boldsymbol{e}F\rangle)=\ell(\mathrm{H}_{i}\langle F\rangle\slash\boldsymbol{r}\mathrm{H}_{i}\langle F\rangle)+\ell(0:_{\mathrm{H}_{i-1}\langle F\rangle}\boldsymbol{r}),
\]
here $\ell(-)$ is the length function. \end{cor}

\section{The Symmetric DG Algebra}

Let $R$ be a commutative ring, let $A$ be a $\mathbb{Z}$-graded
$R$-module such that $A_{0}=R$ which is also equipped with a $\mathbb{Z}$-linear
differential $\mathrm{d}\colon A\to A$ giving it the structure of
a chain complex. Note that the differential need not be $R$-linear
and note that $A$ may be nonzero in negative homological degree.
In this section, we will construct the symmetric DG algebra of $A$,
which we denote by $\mathrm{S}(A)$. After constructing the symmetric
DG algebra in this general setting, we then specialize to the case
we are mostly interesting in, namely that $A$ is an $R$-complex
centered at $R$ meaning the differential of $A$ is $R$-linear with
$A_{0}=R$ and $A_{<0}=0$. In this case, we sometimes denote the
symmetric DG algebra of $A$ by $\mathrm{S}_{R}(A)$ with $R$ in
the subscript in order to emphasize that $A$ is centered at $R$.

~~~~~~Before we give a rigorous construction of the symmetric
DG algebra, we wish to help motivate the reader by giving an informal
description of it in this special case where $A$ is an $R$-complex
centered at $R$. In this case, the underlying graded algebra of $S=\mathrm{S}_{R}(A)$
is the usual symmetric $R$-algebra $\mathrm{Sym}(A_{+})$ where we
view $A_{+}$ as just an $R$-module. However $S$ obtains a bi-graded
structure using homological degree and total degree: we have a decomposition
of $S$ into $R$-modules:
\[
S=\bigoplus_{i\geq0}S_{i}=\bigoplus_{m\geq0}S^{m}=\bigoplus_{i,m\geq0}S_{i}^{m}.
\]
We refer to the $i$ in the subscript as homological degree\textbf{
}and we refer to the $m$ in the superscript as total degree. We have
$S_{0}=S^{0}=S_{0}^{0}=R$ and $S^{1}=A_{+}$. More generally, for
$i,m\geq1$, the $R$-module $S_{i}^{m}$ is the $R$-span of all
homogeneous elementary products of the form $\boldsymbol{a}=a_{1}\cdots a_{m}$
where $a_{1},\dots,a_{m}\in A_{+}$ are homogeneous (with respect
to homological degree of course) such that
\[
|\boldsymbol{a}|=|a_{1}|+\cdots+|a_{m}|=i.
\]
In particular, note that $A=S^{\leq1}=R+A_{+}$, thus we view $A$
as being the total degree $\leq1$ part of $S$. The differential
of $A$ extends the differential of $S$ in a natural way and is defined
on homogeneous elementary products $\boldsymbol{a}=a_{1}\cdots a_{m}$
by
\begin{align}
\mathrm{d}\boldsymbol{a} & =\sum_{j=1}^{m}(-1)^{|a_{1}|+\cdots+|a_{j-1}|}a_{1}\cdots\mathrm{d}(a_{j})\cdots a_{m}.\label{eq:caretaken}
\end{align}
If each of the $a_{j}$ in (\ref{eq:caretaken}) live in homological
degree $\ge2$, then $\mathrm{d}\boldsymbol{a}$ and $\boldsymbol{a}$
has the same total degree, namely $\deg(\mathrm{d}\boldsymbol{a})=m=\deg\boldsymbol{a}$.
However if one of the $a_{j}$ in (\ref{eq:caretaken}) lives in homological
degree $1$, then $\deg(\mathrm{d}\boldsymbol{a})=m-1$. The diagram
below illustrates how the differential acts on the bi-graded components:
\begin{center}\begin{tikzcd}[column sep = 7]

A_i \arrow[dr] && \cdots \arrow[dl] \arrow[dr] && S_i ^{m-1} \arrow[dl] \arrow[dr] && S_i ^{m}  \arrow[dl] \arrow[dr] &&  S_i ^{m+1} \arrow[dl] \arrow[dr] && \cdots \arrow[dl] \arrow[dr] && \mathbb{K} _i \arrow[dl]

\\

& A_{i-1} && S_{i-1} ^{m-2}   && S_{i-1} ^{m-1}  && S_{i-1} ^{m} && S_{i-1} ^{m+1} && \mathbb{K} _{i-1}

\end{tikzcd}\end{center}where we set $\mathbb{K}$ to be the koszul DG algebra induced by
$\mathrm{d}\colon A_{1}\to A_{0}$. Thus the differential of $S$
connects the usual differential of $A$ on the far left to a koszul
differential on the far right. In order to keep track of how the differential
operates on the bi-graded components, we express $\mathrm{d}$ as
\[
\mathrm{d}=\eth+\partial,
\]
where $\eth$ is the component of $\mathrm{d}$ which respects total
degree and where $\partial$ is the component of $\mathrm{d}$ which
drops total degree by $1$. In the next example, we consider a free
resolution of a cyclic module and work out what the symmetric DG algebra
looks like in this case.

\begin{example}\label{symmetricdgexample1} Let $R=\Bbbk[x,y]$, let
$\boldsymbol{m}=x^{2},xy$, and let $F$ be Taylor resolution of $R\slash\boldsymbol{m}$
over $R$. We write down the homogeneous components of $F$ as a graded
$R$-module as well as how the differential acts on the homogeneous
basis below:
\begin{align*}
F_{0} & =R & \mathrm{d}e_{1} & =x^{2}\\
F_{1} & =Re_{1}+Re_{2} & \mathrm{d}e_{2} & =xy\\
F_{2} & =Re_{12}, & \mathrm{d}e_{12} & =xe_{2}-ye_{1},
\end{align*}
Note that the Taylor resolution usually comes equipped with a multiplication
called the Taylor multiplication. Let us denote this by $\star$ so
as not to confuse it with the multiplication $\cdot$ of the symmetric
DG algebra $S=\mathrm{S}_{R}(F)$ of $F$. Now we write down the homogeneous
components of $S$ as a graded $R$-module (with respect to homological
degree) below:
\begin{align*}
S_{0} & =R\\
S_{1} & =Re_{1}+Re_{2}\\
S_{2} & =Re_{12}+Re_{1}e_{2}\\
S_{3} & =Re_{1}e_{12}+Re_{2}e_{12}\\
S_{4} & =Re_{12}^{2}+Re_{1}e_{2}e_{12}\\
 & \vdots\\
S_{2k-1} & =Re_{1}e_{12}^{k-1}+Re_{2}^{k-1}\\
S_{2k} & =Re_{12}^{k}+Re_{1}e_{2}e_{12}^{k-1}\\
S_{2k+1} & =Re_{1}e_{12}^{k}+Re_{2}e_{12}^{k}\\
 & \vdots
\end{align*}
Note that
\begin{align*}
\mathrm{d}(e_{1}e_{2}-xe_{12}) & =\mathrm{d}(e_{1}e_{2})-x\mathrm{d}(e_{12})\\
 & =\mathrm{d}(e_{1})e_{2}-e_{1}\mathrm{d}(e_{2})-x(xe_{2}-ye_{1})\\
 & =x^{2}e_{2}-xye_{1}-x^{2}e_{2}+xye_{1}\\
 & =0.
\end{align*}

\end{example}

\subsection{Construction of the Symmetric DG Algebra of $A$}

We now provide a rigorous construction of $\mathrm{S}(A)$ in the
general case where the differential of $A$ need not be $R$-linear
and where $A_{<0}$ is not necessarily zero. Our construction will
occur in three steps:

\hfill

\textbf{Step 1: }We define the \textbf{non-unital} \textbf{tensor
DG algebra} of $A$ to be
\[
\mathrm{U}_{\mathbb{Z}}(A):=\bigoplus_{n=1}^{\infty}A^{\otimes n},
\]
where the tensor product is taken as $\mathbb{Z}$-complexes. An elementary
tensor in $U=\mathrm{U}_{\mathbb{Z}}(A)$ is denoted $\boldsymbol{a}=a_{1}\otimes\cdots\otimes a_{n}$
where $a_{1},\dots,a_{n}\in A$ and $n\ge1$. The differential of
$U$ is denoted by $\mathrm{d}$ again to simplify notation and is
defined on $\boldsymbol{a}$ by
\[
\mathrm{d}\boldsymbol{a}=\sum_{j=1}^{n}(-1)^{|a_{1}|+\cdots+|a_{j-1}|}a_{1}\otimes\cdots\otimes\mathrm{d}a_{j}\otimes\cdots\otimes a_{n}.
\]

We say\textbf{ $\boldsymbol{a}$} is a homogeneous elementary tensors
if each $a_{i}$ is a homogeneous element in $A$. In this case, we
set
\[
|\boldsymbol{a}|=\sum_{i=1}^{n}|a_{i}|\quad\text{and}\quad\deg\boldsymbol{a}=\sum_{i=1}^{n}\deg a_{i},
\]
where $\deg$ is defined on elements $a\in A$ by
\[
\deg a=\begin{cases}
1 & \text{if }a\in A_{>0}\\
0 & \text{if }a\in R\\
-1 & \text{if }a\in A_{<0}
\end{cases}
\]
We call $|\boldsymbol{a}|$ the \textbf{homological degree }of $\boldsymbol{a}$
and we call $\deg\boldsymbol{a}$ the \textbf{total degree }of $\boldsymbol{a}$.
With $|\cdot|$ and $\deg$ defined, we observe that $U$ admits a
bi-graded decomposition:
\[
U=\bigoplus_{i\in\mathbb{Z}}U_{i}=\bigoplus_{m\in\mathbb{Z}}U^{m}=\bigoplus_{i,m\in\mathbb{Z}}U_{i}^{m},
\]
where the component $U_{i}^{m}$ consists of all finite $\mathbb{Z}$-linear
combinations of homogeneous elementary tensors $\boldsymbol{a}\in U$
such that $|\boldsymbol{a}|=i$ and $\deg\boldsymbol{a}=m$. We equip
$U$ with an associative (but not commutative nor unital) bi-graded
$\mathbb{Z}$-bilinear multiplication which is defined on homogeneous
elementary tensors by $(\boldsymbol{a},\boldsymbol{a}')\mapsto\boldsymbol{a}\otimes\boldsymbol{a}'$
and is extended $\mathbb{Z}$-bilinearly everywhere else. This multiplication
is easily seen to satisfy Leibniz rule, however note that $U$ is
not unital under this multiplication since $(1,1)\mapsto1\otimes1\neq1$
(hence why we call this the \emph{non-unital }tensor DG algebra).
Also note that $U$ already comes equipped with an $R$--scalar multiplication
(from the $R$-module structure on $A$), denoted $(r,\boldsymbol{a})\mapsto r\boldsymbol{a}$,
however the multiplication of $U$ only agrees with the $R$-scalar
multiplication wherever they are both defined and vanish. To rectify
this, let $\mathfrak{u}=\mathfrak{u}(A)$ be the $U$-ideal by all
elements of the form
\begin{align*}
[r,a]_{\mu} & =r\otimes a-ra & [a,r]_{\mu} & =a\otimes r-ar\\{}
[r,a]_{\mathrm{d}} & =\mathrm{d}r\otimes a-\mathrm{d}(ra)+r(\mathrm{d}a) & [a,r]_{\mathrm{d}} & =(-1)^{|a|}a\otimes\mathrm{d}r-\mathrm{d}(ar)+(\mathrm{d}a)r
\end{align*}
where $r\in R$ and $a\in A$. 

\begin{lemma}\label{lemma} The differential maps $\mathfrak{u}$
to itself. \end{lemma} 

\begin{proof}\label{proof} Indeed, given $r\in R$ and $a\in A$,
we have
\begin{align*}
\mathrm{d}[r,a]_{\mu} & =\mathrm{d}(r\otimes a)-\mathrm{d}(ra)\\
 & =\mathrm{d}r\otimes a+r\otimes\mathrm{d}a-\mathrm{d}r\otimes a+r(\mathrm{d}a)+[r,a]_{\mathrm{d}}\\
 & =r\otimes\mathrm{d}a+r(\mathrm{d}a)+[r,a]_{\mathrm{d}}\\
 & =[r,\mathrm{d}a]_{\mu}+[r,a]_{\mathrm{d}}\\
 & \in\mathfrak{u}.
\end{align*}
Similarly we have
\begin{align*}
\mathrm{d}[r,a]_{\mathrm{d}} & =\mathrm{d}(\mathrm{d}r\otimes a-\mathrm{d}(ra)+r(\mathrm{d}a))\\
 & =-\mathrm{d}r\otimes\mathrm{d}a+\mathrm{d}(r(\mathrm{d}a))\\
 & =-\mathrm{d}r\otimes\mathrm{d}a+\mathrm{d}(r\otimes\mathrm{d}a-[r,\mathrm{d}a]_{\mu})\\
 & =-\mathrm{d}r\otimes\mathrm{d}a+\mathrm{d}r\otimes\mathrm{d}a-\mathrm{d}[r,\mathrm{d}a]_{\mu}\\
 & =-\mathrm{d}[r,\mathrm{d}a]_{\mu}\\
 & =-[r,\mathrm{d}a]_{\mathrm{d}}\\
 & \in\mathfrak{u}.
\end{align*}
Similar calculations show $\mathrm{d}[a,r]_{\mu}\in\mathfrak{u}$
and $\mathrm{d}[a,r]_{\mathrm{d}}\in\mathfrak{u}$. \end{proof}

\hfill

\textbf{Step 2: }We define the \textbf{tensor DG algebra }of $A$
to be the quotient
\[
\mathrm{T}(A):=\mathrm{U}(A)\slash\mathfrak{u}(A).
\]

The multiplication of $U=\mathrm{U}(A)$ induces a multiplication
on $T=\mathrm{T}(A)$ which not only becomes unital but also agrees
with the $R$-scalar multiplication on $T$ where they are both defined.
Since $\mathfrak{u}=\mathfrak{u}(A)$ is generated by elements which
are homogeneous with respect to homological degree and since the differential
of $U$ maps $\mathfrak{u}$ to itself, it follows that the differential
of $U$ induces a differential on $T$, which we again denote by $\mathrm{d}$
again. This gives $T$ the structure of a non-commutative (but unital)
DG $\Bbbk$-algebra, where
\[
\Bbbk=\{r\in R\mid\mathrm{d}r\otimes a=0\text{ for all }a\in A\}.
\]
In other words, the differential of $T$ satisfies Leibniz rule and
is $\Bbbk$-linear. Note that the generator $[r,a]_{\mu}$ of $\mathfrak{u}$
is also homogeneous with respect to total degree, however the generators
$[r,a]_{\mathrm{d}}$ is homogeneous with respect to total degree
if and only if either $\mathrm{d}r\otimes a=0$, or $\mathrm{d}(ra)=r\mathrm{d}a$,
or $|a|\in\{0,1\}$. In particular, $\mathfrak{u}$ will be homogeneous
with respect to total degree if $A$ is an $R$-complex centered at
$R$ (which is a case we are interested in). In this case, $T$ inherits
from $U$ a bi-graded $R$-algebra structure:
\[
T=\bigoplus_{i\in\mathbb{Z}}T_{i}=\bigoplus_{m\in\mathbb{Z}}T^{m}=\bigoplus_{i,m\in\mathbb{Z}}T_{i}^{m}.
\]

\begin{example}\label{example} Let us describe what the total degree
$m$ component of $T=\mathrm{T}_{R}(A)$ in the case where $A$ is
an $R$-complex centered at $R$. We have
\begin{align*}
T^{0} & =R\\
T^{1} & =\bigoplus_{1\leq i}A_{i}\\
T^{2} & =\bigoplus_{\substack{1\leq i<j}
}((A_{i}\otimes A_{j})\oplus(A_{j}\otimes A_{i}))\oplus\bigoplus_{1\leq i}A_{i}^{\otimes2}
\end{align*}
The component $T^{3}$ is slightly more complicated:
\[
\bigoplus_{\substack{1\leq i<j<k\\
\pi\in S_{3}
}
}(A_{\pi(i)}\otimes A_{\pi(j)}\otimes A_{\pi(k)})\oplus\bigoplus_{\substack{1\leq i<j\\
\pi\in S_{2}
}
}((A_{\pi(i)}^{\otimes2}\otimes A_{\pi(j)})\oplus(A_{\pi(i)}\otimes A_{\pi(j)}\otimes A_{\pi(i)})\oplus(A_{\pi(i)}\otimes A_{\pi(j)}^{\otimes2}))\oplus\bigoplus_{1\leq i}A_{i}^{\otimes3}.
\]
More generally, there is an interpretation of $T^{m}$ in terms of
certain rooted trees. \end{example}

~~~~~~~Now let $\mathfrak{t}=\mathfrak{t}(A)$ be the $T$-ideal
generated by all elements of the form
\[
[a_{1},a_{2}]_{\sigma}\colon=(-1)^{|a_{1}||a_{2}|}a_{2}\otimes a_{1}-a_{1}\otimes a_{2}\quad\text{and}\quad[a]_{\tau}:=a\otimes a,
\]
where $a,a_{1},a_{2}\in A$ are homogeneous and $|a|$ is odd.

\begin{lemma}\label{lemma} The differential of $T$ maps $\mathfrak{t}$
to itself. \end{lemma}

\begin{proof}\label{proof} Indeed, if $a,a_{1},a_{2}\in A$ are homogeneous
with $|a|$ odd, then we have
\begin{align*}
\mathrm{d}[a_{1},a_{2}]_{\sigma} & =[\mathrm{d}a_{1},a_{2}]_{\sigma}+(-1)^{|a_{1}|}[a_{1},\mathrm{d}a_{2}]_{\sigma}\in\mathfrak{t}\quad\text{and}\quad\mathrm{d}[a]_{\tau}=[\mathrm{d}a,a]_{\sigma}\in\mathfrak{t}.
\end{align*}

\end{proof}

\hfill

\textbf{Step 3: }We define the \textbf{symmetric DG algebra }of $A$
to be the quotient
\[
\mathrm{S}(A):=\mathrm{T}(A)\slash\mathfrak{t}(A)
\]
The image of a homogeneous elementary tensor $a_{1}\otimes\cdots\otimes a_{m}$
in $S=\mathrm{S}(A)$ is often denoted $a_{1}\cdots a_{n}$ and is
called a homogeneous elementary product. Since $\mathfrak{t}=\mathfrak{t}(A)$
is generated by elements which are homogeneous with respect to both
homological degree and since the differential of $T=\mathrm{T}(A)$
maps $\mathfrak{t}$ to itself, we see that the differential of $T$
induces a differential on $S$, which we again denote by $\mathrm{d}$,
giving it the structure of a strictly graded-commutative DG $\Bbbk$-algebra.
Furthemore, if $T$ inherits the bi-graded structure from $U$, then
$S$ inherits the bi-graded structure from $T$ since $\mathfrak{t}$
is generated by elements which are homogeneous with respect to total
degree.

\subsection{Properties of the Symmetric DG Algebra}

We now focus our attention to the case where $A$ is an $R$-complex
centered at $R$ and we wish to study $S=\mathrm{S}_{R}(A)$ the symmetric
DG $R$-algebra of $A$ (note that we sometimes write $R$ in the
subscript of $\mathrm{S}_{R}(A)$ to emphasize that $A$ and $S=\mathrm{S}_{R}(A)$
are centered at $R$). In this case, the underlying graded $R$-algebra
of $S$ is the usual symmetric algebra of $A_{+}$:
\[
\mathrm{Sym}_{R}(A_{+})=\frac{\bigoplus_{m\ge0}A_{+}^{\otimes m}}{\langle\{[a_{1},a_{2}]_{\sigma},[a]_{\tau}\}\rangle},
\]
where the tensor product is taken over $R$. Thus the symmetric DG
algebra of $A$ inherits all of the properties that are satisfied
by the symmetric algebra of $A_{+}$ when we forget about the differential.
For instance, recall that a bounded below $R$-complex is semiprojective
if and only if its underlying graded $R$-module is projective as
a graded $R$-module. In particular, if $A$ is semiprojective, then
$S$ is semiprojective too. Thus if we assume that $A$ is semiprojective
\emph{and }that there exists a surjective chain map $\pi\colon S\twoheadrightarrow A$
which splits the inclusion map $\iota\colon A\hookrightarrow S$,
then we can lift chains maps out of $A$ along surjective quasi-isomorphisms,
meaning if $\varphi\colon A\to X$ is any chain map and $\tau\colon Y\to X$
is any surjective quasi-isomorphism, then there exists a chain map
$\widetilde{\varphi}\colon S\to Y$ such that $\tau\widetilde{\varphi}=\varphi$,
moreover such a lift is unique up to homotopy. The assumption that
$A$ is semiprojective is mild whereas the assumption that there exists
a chain map $S\twoheadrightarrow A$ which splits the inclusion map
$A\hookrightarrow S$ is rather subtle. We shall see that such a surjective
chain map $\pi\colon S\twoheadrightarrow A$ will exist if $A$ has
a DG $R$-algebra structure on it, and in Proposition~(\ref{propsurjkls}),
we shall see that such a surjective chain map $\pi\colon S\twoheadrightarrow A$
exists in the case where $A$ is a projective resolution of a cyclic
$R$-module.

\begin{prop}\label{prop} Let $R$ be a commutative ring and let $A$
be an $R$-complex centered at $R$. 
\begin{enumerate}
\item (Base Change) Let $R'$ be an $R$-algebra. Then
\begin{equation}
\mathrm{S}_{R}(A)\otimes_{R}R'=\mathrm{S}_{R'}(A\otimes_{R}R').\label{eq:basecangelll}
\end{equation}
\item (Exact Sequences) Let \begin{equation}\label{equation}\begin{tikzcd} B \arrow[r] &  A \arrow[r] &  A' \arrow[r] & 0 \end{tikzcd}\end{equation}be
an exact sequence of $R$-complexes where $A'$ is centered at a cyclic
$R$-algebra, say $R'=R\slash I$ for some ideal $I$ of $R$. Then
we obtain an exact sequence\begin{equation}\label{equation}\begin{tikzcd} \mathrm{S}_R (A) \otimes _R B \arrow[r] &  \mathrm{S}_R (A) \arrow[r] &  \mathrm{S}_{R'} (A') \arrow[r] & 0 \end{tikzcd}\end{equation}.
\item (Universal Mapping Property) For every chain map of the form $\varphi\colon A\to A'$,
where $A'$ is a DG algebra centered at a ring $R'$ and where $\varphi$
restricts to a ring homomorphism $\varphi_{0}\colon R\to R'$, there
exists a unique DG algebra homomorphism $\widetilde{\varphi}\colon\mathrm{S}_{R}(A)\to A'$
which extends $\varphi\colon A\to A'$, that is, such that $\widetilde{\varphi}\circ\iota=\varphi$
where $\iota\colon A\hookrightarrow\mathrm{S}_{R}(A)$ is the inclusion
map. We express this in terms of a commutative diagram as below:\begin{equation}\label{equationump2}\begin{tikzcd}[row sep =40, column sep = 40]  A \arrow[dr,"\varphi ", swap] \arrow[r, " \iota ", hook  ] & \mathrm{S}_R (A) \arrow[d, "  \widetilde { \varphi }  " ] \\ & A' \end{tikzcd}\end{equation}
\end{enumerate}
\end{prop}

\begin{rem}\label{rem} Strictly speaking, one should write $R\otimes_{R}R'$
in the subscript on the right hand side of Equation~(\ref{eq:basecangelll}).
However we may view $R'$ as being the homological degree $0$ part
by identifying $R'$ with $R\otimes_{R}R'$ via the canonical isomorphism
$R'\simeq R\otimes_{R}R'$. \end{rem}

\begin{proof}\label{proof} We only prove the third property since
the first two properties are straightforward to show. Let $\varphi\colon A\to A'$
be such a chain map and denote $S=\mathrm{S}_{R}(A)$. We define $\widetilde{\varphi}\colon S\to A'$
by setting $\widetilde{\varphi}|_{A}=\varphi$ and
\begin{equation}
\widetilde{\varphi}(a_{1}\cdots a_{m})=\varphi(a_{1})\cdots\varphi(a_{m})\label{eq:satisfythis-1}
\end{equation}
for all homogeneous elementary products $a_{1}\cdots a_{m}$ in $S^{\geq2}$
and then extending it $R$-linearly everywhere else. By construction,
$\widetilde{\varphi}$ is multiplicative and extends $\varphi\colon A\to A'$.
Furthermore, $\widetilde{\varphi}$ is a chain map since it is a graded
$R$-linear map which commutes with the differential. Indeed, we clearly
have $\widetilde{\varphi}\mathrm{d}(1)=0=\mathrm{d}\widetilde{\varphi}(1)$,
and for all homogeneous elementary products $a_{1}\cdots a_{m}$ in
$S^{\geq2}$, we have
\begin{align*}
\widetilde{\varphi}\mathrm{d}(a_{1}\cdots a_{m}) & =\sum_{j=1}^{m}(-1)^{|a_{1}|+\cdots+|a_{j-1}|}\widetilde{\varphi}(a_{1}\cdots\mathrm{d}(a_{j})\cdots a_{m})\\
 & =\sum_{j=1}^{m}(-1)^{|a_{1}|+\cdots+|a_{j-1}|}\varphi(a_{1})\cdots\varphi\mathrm{d}(a_{j})\cdots\varphi(a_{m})\\
 & =\sum_{j=1}^{m}(-1)^{|a_{1}|+\cdots+|a_{j-1}|}\varphi(a_{1})\cdots\mathrm{d}\varphi(a_{j})\cdots\varphi(a_{m})\\
 & =\mathrm{d}(\varphi(a_{1})\cdots\varphi(a_{m}))\\
 & =\mathrm{d}\widetilde{\varphi}(a_{1}\cdots a_{m}).
\end{align*}
Finally, if $\widehat{\varphi}\colon S\to A'$ were another DG algebra
homomorphism which extended $\varphi\colon A\to B$, then we would
have
\[
\widetilde{\varphi}(a_{1}\cdots a_{m})=\widehat{\varphi}(a_{1})\cdots\widehat{\varphi}(a_{m})=\varphi(a_{1})\cdots\varphi(a_{m})=\widetilde{\varphi}(a_{1}\cdots a_{m})
\]
for all homogeneous elementary products $a_{1}\cdots a_{m}$ in $S^{\geq2}$,
which implies $\widehat{\varphi}=\widetilde{\varphi}$. \end{proof}

\begin{defn}\label{defn} Let $A$ and $B$ be two $R$-complexes
centered at $R$. We define their \textbf{wedge sum }$A\lor B$ to
be the $R$-complex centered at $R$ whose underlying graded $R$-module
is given by
\[
(A\lor B)_{i}=\begin{cases}
A_{i}\oplus B_{i} & \text{if }i\geq1\\
R & \text{if }i=0
\end{cases}
\]
and whose differential is defined by
\[
\mathrm{d}(a,b)=\begin{cases}
(\mathrm{d}a,\mathrm{d}b) & \text{if }|a|=|b|\geq2\\
\mathrm{d}a-\mathrm{d}b & \text{if }|a|=|b|=1
\end{cases}
\]
Observe that
\[
\mathrm{H}_{i}(A\lor B)=\begin{cases}
R\slash(\mathrm{d}A_{1}+\mathrm{d}B_{1}) & \text{if }i=0\\
(A_{1}\times_{R}B_{1})\slash(\mathrm{d}A_{2}\oplus\mathrm{d}B_{2}) & \text{if }i=1\\
\mathrm{H}_{i}(A)\oplus\mathrm{H}_{i}(B) & \text{if }i\geq2
\end{cases}
\]
\end{defn}

\begin{prop}\label{prop} Let $A$ and $B$ be two $R$-complexes
centered at $R$. Then we have
\[
\mathrm{S}_{R}(A\lor B)=\mathrm{S}_{R}(A)\otimes_{R}\mathrm{S}_{R}(B).
\]
\end{prop}

\begin{proof}\label{proof} In terms of the underlying graded $R$-algebras,
we have
\begin{align*}
\mathrm{S}_{R}(A\lor B) & =\mathrm{Sym}_{R}(A_{+}\oplus B_{+})\\
 & =\mathrm{Sym}_{R}(A_{+})\otimes_{R}\mathrm{Sym}_{R}(B)\\
 & =\mathrm{S}_{R}(A)\otimes_{R}\mathrm{S}_{R}(B).
\end{align*}
It is easy to check that the differential of $\mathrm{S}_{R}(A\lor B)$
is carried over to the differential of $\mathrm{S}_{R}(A)\otimes_{R}\mathrm{S}_{R}(B)$
under this isomorphism (we write equality here because $\mathrm{S}_{R}(A)\otimes_{R}\mathrm{S}_{R}(B)$
satisfies the universal mapping property of the symmetric DG $R$-algebra
of $A\lor B$. \end{proof}

\subsection{Presentation of the Maximal Associative Quotient}

Let $A$ be an $R$-complex centered at $R$ and let $S=\mathrm{S}_{R}(A)$
be the symmetric DG $R$-algebra of $A$. Equip $A$ with a multiplication
$\mu=(\mu,\star)$ giving it the structure of an MDG $R$-algebra.
In particular, note that if $a_{1},a_{2}\in A_{1}$, then
\[
a_{1}a_{2}\in S_{2}^{2},\,\,\,a_{1}\star a_{2}\in S_{2}^{1},\,\,\,\text{and}\,\,\,[a_{1},a_{2}]\in S_{2},
\]
where $[a_{1},a_{2}]=a_{1}\star a_{2}-a_{1}a_{2}$ is the multiplicator
of the inclusion map $\iota\colon A\hookrightarrow S$ evaluated at
$(a_{1},a_{2})\in A^{2}$. Let $\mathfrak{s}=\mathfrak{s}(\mu)$ be
the $S$-ideal generated by all such multiplicators, so
\[
\mathfrak{s}=\mathrm{span}_{S}\{[a_{1},a_{2}]\mid a_{1},a_{2}\in A\}.
\]
Also let $\pi\colon S\to S\slash\mathfrak{s}$ and $\pi^{\mathrm{as}}\colon A\twoheadrightarrow A^{\mathrm{as}}$
denote the canonical quotient maps. The universal mapping property
of the symmetric DG algebra of $A$ implies $\pi^{\mathrm{as}}\colon A\twoheadrightarrow A^{\mathrm{as}}$
extends uniquely to a DG algebra homomorphism $S\twoheadrightarrow A^{\mathrm{as}}$
which we again denote by $\pi^{\mathrm{as}}$. We let $S^{\geq2}=S\slash A$
be the $R$-complex whose underlying graded $R$-module is $S^{\geq2}$
and whose differential $\mathrm{d}^{\geq2}$ is defined by
\[
\mathrm{d}^{\geq2}|_{S^{m}}=\begin{cases}
\eth|_{S^{2}} & \text{if }m=2\\
\mathrm{d}|_{S^{m}} & \text{if }m>2.
\end{cases}
\]
We also let $\rho\colon S\twoheadrightarrow S\slash A=S^{\geq2}$
be the canonical quotient map. We now present the fourth main theorem
from the introduction.

\begin{theorem}\label{theorempresentationsym} With the notation as
above, we have
\[
A^{\mathrm{as}}=\mathrm{coker}(\mathfrak{s}\hookrightarrow S)=S\slash\mathfrak{s}
\]
More specifically, there is a unique isomorphism $A^{\mathrm{as}}\to S\slash\mathfrak{s}$
of $\mathrm{DG}$ $S$-algebras (thus we are justified in writing
$\pi\colon S\to A^{\mathrm{as}}$ to denote both $\pi^{\mathrm{as}}\colon S\to A^{\mathrm{as}}$
and $\pi\colon S\to S\slash\mathfrak{s}$ in order to simplify notation).
In particular, this implies
\[
\langle A\rangle=A\cap\mathfrak{s}=\mathfrak{s}^{\leq1}=\ker(\mathfrak{s}\to S^{\geq2})
\]
Thus we have the following canonically defined hexagonal-shaped diagram
of $R$-complexes which is exact everywhere in every direction:\begin{equation}\label{canondiagram}\begin{tikzcd}[row sep=30 ]& S ^{\geq 2} \arrow[r] & 0 \\ \mathfrak{s} \arrow[ur, two heads, color = green ]  \arrow[r, hook, "i ", color = blue ] & S \arrow[u, two heads, "\rho ", swap , color = green ] \arrow[r, two heads, "\pi ", color = blue ] & A^{\mathrm{as}}  \arrow[u] \\ \mathfrak{s} ^{\leq 1} \arrow[r, hook, color = red ] \arrow[u, hook' , color = green ] & A \arrow[u, hook, "\iota ", swap, color = green  ] \arrow[ur, two heads, color = red ]\end{tikzcd}\end{equation}where
the blue arrows are $\mathrm{DG}$ $S$-module homomorphisms, where
the green arrows are chain maps as $R$-complexes, and where the red
arrows are $\mathrm{MDG}$ $A$-module homomorphisms.\end{theorem}

\begin{proof}\label{proof} Observe that $\pi^{\mathrm{as}}\colon S\twoheadrightarrow A^{\mathrm{as}}$
satifies
\begin{align*}
\pi^{\mathrm{as}}[a_{1},a_{2}] & =\pi^{\mathrm{as}}(a_{1}\star a_{2}-a_{1}a_{2})\\
 & =\pi^{\mathrm{as}}(a_{1}\star a_{2})-\pi^{\mathrm{as}}(a_{1}a_{2})\\
 & =\pi^{\mathrm{as}}(a_{1})\star\pi^{\mathrm{as}}(a_{2})-\pi^{\mathrm{as}}(a_{1})\star\pi^{\mathrm{as}}(a_{2})\\
 & =0.
\end{align*}
Thus the universal mapping property of the quotient $S\slash\mathfrak{s}=\mathrm{coker}(\mathfrak{s}\hookrightarrow S)$
implies there is a unique DG algebra homomorphism $\overline{\pi}^{\mathrm{as}}\colon S\slash\mathfrak{s}\to A^{\mathrm{as}}$
such that 
\[
\overline{\pi}^{\mathrm{as}}\circ\pi=\pi^{\mathrm{as}}.
\]
Similarly, note that the composite $\pi\circ\iota\colon A\to S\slash\mathfrak{s}$
is an MDG algebra homomorphism which is surjective. Indeed, if $a_{1}\cdots a_{m}$
is a homogeneous elementary tensor in $S^{m}$, then we have
\[
a_{1}a_{2}a_{3}\cdots a_{m}=((\cdots(a_{1}\star a_{2})\star a_{3})\star\cdots)\star a_{m}
\]
in $S\slash\mathfrak{s}$. Thus every element in $S\slash\mathfrak{s}$
can be represented by an element in $A=S^{1}$ which implies $\pi\iota\colon A\twoheadrightarrow S\slash\mathfrak{s}$
is surjective as claimed. In particular, since $S\slash\mathfrak{s}$
is associative, it follows from the universal mapping property of
the maximal associative quotient of $A$ that there is a unique DG
algebra homomorphism $\overline{\pi}\colon A^{\mathrm{as}}\to S\slash\mathfrak{s}$
such that
\[
\pi\circ\iota=\overline{\pi}\circ\pi^{\mathrm{as}}.
\]
Combining all of this together, we have a commutative diagram of MDG
$S$-modules: \begin{center}\begin{tikzcd}[column sep=50, row sep =50]

 S \arrow[r, " \pi  ", two heads] \arrow[dr, " \pi ^{\mathrm{as}} ", dashed ] & S \slash \mathfrak{s} \arrow[d, "\overline{\pi } ^{\mathrm{as}} ", shift left= 0.2, dashed ]

\\

 A \arrow[u, " \iota ", hook' ] \arrow[r, " \pi ^{\mathrm{as}} ", swap, two heads] & A^{\mathrm{as} } \arrow[u, dashed, " \overline{\pi } ", shift left=1] \end{tikzcd}\end{center}where the dashed arrows indicates uniqueness. \end{proof}

\begin{cor}\label{cor} Let $A$ be an $R$-complex centered at $R$
and let $S=\mathrm{S}_{R}(A)$ be the symmetric DG algebra of $A$.
Then a necessary condition for $A$ to have a DG algebra structure
is that the canonical short exact sequence of $R$-complexes \begin{equation}\label{equationfjskdl}\begin{tikzcd} 0 \arrow[r] &  A \arrow[r, "\iota "] &  S \arrow[r, "\rho " ] &  S^{ \geq 2} \arrow[r] & 0 \end{tikzcd}\end{equation}is
split. \end{cor}

\begin{proof}\label{proof} Indeed, assume that $A=A^{\mathrm{as}}$.
Then the canonical map $\mathfrak{s}\to S^{\geq2}$ defined on multiplicators
by
\[
[a_{1},a_{2}]\mapsto a_{1}a_{2}
\]
is an isomorphism of $R$-complexes. Let $\theta\colon S^{\geq2}\xrightarrow{\simeq}\mathfrak{s}\hookrightarrow S$
be the composite map where $S^{\geq2}\xrightarrow{\simeq}\mathfrak{s}$
is the inverse isomorphism of the canonical map $\mathfrak{s}\to S^{\geq2}$.
We obtain a short exact sequence of $R$-complexes \begin{equation}\label{equation}\begin{tikzcd} 0 \arrow[r] &  S^{\geq 2} \arrow[r, "\theta "] &  S \arrow[r, "\pi " ] &  A \arrow[r] & 0 \end{tikzcd}\end{equation}which
is split by the inclusion map $\iota\colon A\to S$. Similarly, the
short exact sequence of $R$-complexes \begin{equation}\label{equation}\begin{tikzcd} 0 \arrow[r] &  A \arrow[r, "\iota "] &  S \arrow[r, "\rho " ] &  S^{ \geq 2} \arrow[r] & 0 \end{tikzcd}\end{equation}is
split by $\theta\colon S^{\geq2}\to S$. \end{proof}

~~~~~~In the case where $A=P$ is any projective resolution
of a cyclic $R$-module such that $P_{0}=R$, then it turns out that
the short exact sequence (\ref{equationfjskdl}) always splits.

\begin{prop}\label{propsurjkls} Let $R$ be a commutative ring, let
$I$ be an ideal of $R$, let $P$ be a projective resolution of $R\slash I$
over $R$, such that $P_{0}=R$, and let $S=\mathrm{S}_{R}(P)$ be
the symmetric DG algebra of $P$ over $R$. There exists a surjective
chain map $\pi\colon S\twoheadrightarrow P$ which splits the inclusion
map $P\hookrightarrow S$. \end{prop}

\begin{proof}\label{proof} It suffices to show that $\mathrm{Ext}_{R}^{1}(S\slash P,P)=0$.
Note that the underlying graded $R$-module of $S\slash P$ is just
$S^{\geq2}=\bigoplus_{n\geq2}P_{+}^{\otimes n}$. In particular, $S\slash P$
is semi-projective, thus $\mathrm{Hom}_{R}^{\star}(S\slash P,-)$
preserves quasi-isomorphisms. It follows that
\[
\mathrm{Ext}_{R}^{1}(S\slash P,P)=\mathrm{Ext}_{R}^{1}(S\slash P,R\slash I)=0,
\]
where the last part follows from the fact that $R\slash I$ sits in
homological degree $0$ but $(S\slash P)_{i}=0$ for all $i\leq1$.
\end{proof}

\begin{rem}\label{rem} Note that giving a surjective chain map $\pi\colon S\twoheadrightarrow P$
which splits the inclusion map is equivalent to giving chain maps
$\pi^{n}\colon P^{\otimes n}\to P$ for each $n\geq2$ such that each
$\pi^{n}$ is strictly commutative and such that for all $1\leq i\leq n$
and for all $a_{1},\dots,a_{i-1},a_{i+1},\dots,a_{n}\in P_{+}$ we
have
\[
\pi^{n}(a_{1},\dots,a_{i-1},1,a_{i},\dots,a_{n})=\pi^{n-1}(a_{1},\dots,a_{i-1},a_{i},\dots,a_{n}).
\]

For instance, if $a_{1},a_{2},a_{3}$ are homogeneous elements in
$F$ with $|a_{1}|=1$ and $|a_{2}|,|a_{3}|\geq2$, then we have
\[
\mathrm{d}\pi^{3}(a_{1},a_{2},a_{3})=r_{1}\pi^{2}(a_{2},a_{3})-\pi^{3}(a_{1},\mathrm{d}a_{2},a_{3})+\pi^{3}(a_{1},a_{2},\mathrm{d}a_{3}),
\]
where $r_{1}=\mathrm{d}a_{1}$. \end{rem}

\subsection{Symmetric Powers of Chain Complexes}

In this subsection, we describe a construction given by Tchernev (in
\cite{Tch95}) and explain how it is related to our construction.
In particular, let $X$ be an $R$-complex. We construct the \emph{non-unital
}symmetric DG algebra of $X$ over $R$, denoted $\mathrm{C}_{R}(X)$
as follows: we begin with the non-unital tensor DG algebra of $X$
over $R$, given by
\[
\mathrm{U}_{R}(X)=\bigoplus_{n=1}^{\infty}X^{\otimes n}
\]
where the tensor product is taken as $R$-complexes. Just as before,
an elementary tensor in $U=\mathrm{U}_{R}(A)$ is denoted $\boldsymbol{x}=x_{1}\otimes\cdots\otimes x_{n}$
where $x_{1},\dots,x_{n}\in X$ and $n\geq1$, and the differential
of $U$ is denoted by $\mathrm{d}$ again to simplify notation and
is defined on $\boldsymbol{x}$ by
\[
\mathrm{d}\boldsymbol{x}=\sum_{j=1}^{n}(-1)^{|x_{1}|+\cdots+|x_{j-1}|}x_{1}\otimes\cdots\otimes\mathrm{d}x_{j}\otimes\cdots\otimes x_{n}.
\]

We say\textbf{ $\boldsymbol{x}$} is a homogeneous elementary tensor
if each $x_{i}$ is a homogeneous element in $X$. What is different
this time is that we equip $U=\mathrm{U}_{R}(X)$ with a different
bi-graded structure; namely we set
\[
|\boldsymbol{x}|=\sum_{i=1}^{n}|x_{i}|\quad\text{and}\quad\deg\boldsymbol{x}=n.
\]
Thus we make no distinction on whether or not $x_{i}\in X_{0}$ or
$x_{i}\in X_{<0}$. With $|\cdot|$ and $\deg$ defined as above,
we observe that $U$ admits a bi-graded decomposition:
\[
U=\bigoplus_{i\in\mathbb{Z}}U_{i}=\bigoplus_{n\geq1}U^{n}=\bigoplus_{i,n}U_{i}^{n},
\]
where the component $U_{i}^{n}$ consists of all finite $R$-linear
combinations of homogeneous elementary tensors $\boldsymbol{x}\in U$
such that $|\boldsymbol{x}|=i$ and $\deg\boldsymbol{x}=n$. We equip
$U$ with an associative (but not commutative nor unital) bi-graded
$R$-bilinear multiplication which is defined on homogeneous elementary
tensors by $(\boldsymbol{x},\boldsymbol{x}')\mapsto\boldsymbol{x}\otimes\boldsymbol{x}'$
and is extended $R$-bilinearly everywhere else. This multiplication
is easily seen to satisfy Leibniz rule, however note that $U$ is
not unital under this multiplication since $(1,1)\mapsto1\otimes1\neq1$
(hence why we call this the \emph{non-unital }tensor DG algebra). 

~~~~~~~Next let $\mathfrak{c}=\mathfrak{c}(X)$ be the $U$-ideal
generated by all elements of the form
\[
[x_{1},x_{2}]_{\sigma}:=(-1)^{|x_{1}||x_{2}|}x_{2}\otimes x_{1}-x_{1}\otimes x_{2}\quad\text{and}\quad[x]_{\tau}:=x\otimes x,
\]
where $x,x_{1},x_{2}\in X$ are homogeneous and $|x|$ is odd. We
then define the \textbf{non-unital symmetric DG algebra }of $X$ over
$R$ to be the quotient
\[
\mathrm{C}_{R}(X):=U\slash\mathfrak{c}.
\]
Since the generators of $\mathfrak{c}$ are homogeneous with respect
to both homological and total degree, we see that $C=\mathrm{C}_{R}(X)$
inherits a bi-graded structure from $U$. In particular, if $X$ is
a positive $R$-complex (meaning $X_{i}=0$ for all $i<0$), then
one has $C_{0}^{n}=\mathrm{Sym}_{R}^{n}(X_{0})$. In general, we call
$C^{n}$ the \textbf{$n$th symmetric power }of $X$. The second symmetric
power and its properties were studied in \cite{FST08}. The next proposition
helps clarify how our construction is related to Tchernev's construction:

\begin{prop}\label{prop} Let $A$ be an $R$-complex centered at
$R$. Denote $S=\mathrm{S}_{R}(A)$ and $C=\mathrm{C}_{R}(A)$. We
have $S^{\leq n}\cong C^{n}$ as $R$-complexes. \end{prop}

\begin{proof}\label{proof} Define $\varphi_{h}\colon S^{\leq n}\to C^{n}$,
called \textbf{homogenization}, as follows: let $f\in S^{\leq n}$
and express it as $f=\sum_{k=0}^{n}f^{k}$ where $f^{k}$ is the total
degree $k$ component of $f$. We set
\[
\varphi_{h}(f)=1^{n-1}\otimes f^{0}+\sum_{k=1}^{n}1^{\otimes(n-k)}\otimes f^{k}.
\]
Conversely, define $\varphi_{d}\colon C^{n}\to S^{\leq n}$, called
\textbf{dehomogenization}, as follows: we set
\[
\varphi_{d}(1^{\otimes k}\otimes\boldsymbol{a})=\boldsymbol{a}
\]
where $\boldsymbol{a}\in A_{+}^{\otimes(n-k)}$ is a homogeneous elementary
tensor. We extend $\varphi_{d}$ everywhere else $R$-linearly. It
is straightforward to check that both $\varphi_{h}$ and $\varphi_{d}$
are chain maps and are inverse to each other. \end{proof}

~~~~~~ Let $X$ be an $R$-complex. Denote $C=\mathrm{C}_{R}(X)$,
$\mathfrak{c}=\mathfrak{c}(X)$, and $U=\mathrm{U}_{R}(X)$. There's
an alternative description of $C^{n}$ which in the case where $R$
contains $\mathbb{Q}$ which is often useful. Let $\sigma=(ij)$ be
a transposition in the symmetric group $\Sigma_{n}$ and let $\boldsymbol{x}=x_{1}\otimes\cdots\otimes x_{n}$
be a homogeneous elementary tensor in $U$. We set
\begin{equation}
\sigma\boldsymbol{x}=\begin{cases}
0 & \text{if }x_{i}=x_{j}\text{ and }|x_{i}|\text{ is odd}\\
(-1)^{|x_{i}||x_{j}|}x_{1}\otimes\cdots x_{j}\otimes\cdots\otimes x_{i}\otimes\cdots\otimes x_{n} & \text{else}.
\end{cases}\label{eq:extendssym}
\end{equation}
Then (\ref{eq:extendssym}) extends to an action of the symmetric
group $\Sigma_{n}$ on $U^{n}$. In particular, $U^{n}$ has the structure
of an $R[\Sigma_{n}]$-module. With this understood, we have $C^{n}=(U^{n})_{\Sigma_{n}}$.
If $R$ contains $\mathbb{Q}$, then the short exact sequence of $R$-complexes\begin{equation}\label{equation}\begin{tikzcd} 0 \arrow[r] &  \mathfrak{c} \arrow[r] &  U \arrow[r] &  C \arrow[r] & 0 \end{tikzcd}\end{equation}is
split exact with splitting map $C\to U$ defined on homogeneous elementary
products by
\[
x_{1}\cdots x_{n}\mapsto\frac{1}{n!}\sum_{\sigma\in\Sigma_{n}}\sigma(x_{1}\otimes\cdots\otimes x_{n}).
\]
In particular, we may identify $C^{n}$ with the $R$-subcomplex of
$U^{n}$ which is fixed by $\Sigma_{n}$ in this case. 

\begin{theorem}\label{theorem} Assume that $\mathbb{Q}\subseteq R$.
Let $\varphi,\psi\colon X\to X'$ be chain maps of $R$-complexes.
Denote $C=\mathrm{C}_{R}(X)$, $C'=\mathrm{C}_{R}(X')$, $U=\mathrm{U}_{R}(X)$,
and $U'=\mathrm{U}_{R}(X')$, and identify $C$ and $C'$ with the
$R$-subcomplexes of $U$ and $U'$ fixed by the symmetric groups.
If $\varphi$ is homotopic to $\psi$, then $\varphi^{\otimes n}$
is homotopic to $\psi^{\otimes n}$ for each $n$. Moreover, we can
choose a homotopy $h^{n}\colon U^{n}\to U'^{n}$ from $\varphi^{\otimes n}$
to $\psi^{\otimes n}$ which restricts to a homotopy $h^{n}|_{C}\colon C^{n}\to C'^{n}$
from $\varphi^{\otimes n}|_{C}$ to $\psi^{\otimes}|_{C}$. \end{theorem}

\begin{proof}\label{proof} Let $h$ be a homotopy from $\varphi$
to $\psi$. For $n=1$, we set $h^{1}=h$. The case where $n=2$ was
shown in \cite{FST08}. More generally for $n\geq2$ we set
\[
h^{n}:=\frac{1}{n!}\sum_{\sigma\in\Sigma_{n}}\sigma\left(\sum_{k=0}^{n-1}(\varphi^{\otimes(n-k-1)}\otimes h\otimes\psi^{\otimes k})\right).
\]
One checks that $h^{n}$ is a homotopy from $\varphi^{\otimes n}$
to $\psi^{\otimes n}$ and by construction is restricts to a map from
$C^{n}$ to $C'^{n}$. \end{proof}

\begin{cor}\label{cor} Assume that $\mathbb{Q}\subseteq R$. Let
$\varphi,\psi\colon A\to A'$ be chain maps of $R$-complexes centered
at $R$. Denote $S=\mathrm{S}_{R}(A)$ and $S'=\mathrm{S}_{R}(A')$,
and let $\widetilde{\varphi},\widetilde{\psi}\colon S\to S'$ be the
lifts of $\varphi$ and $\psi$ from the universal mapping property.
If $\varphi$ is homotopic to $\psi$, then $\widetilde{\varphi}$
is homotopic to $\widetilde{\psi}$. \end{cor}

\subsection{The Symmetric DG Algebra of a Finite Free Complex over an Integral
Domain}

Throughout this subsection, we assume that $R$ is an integral domain
with quotient field $K$. Let $F$ be an $R$-complex centered at
$R$ such that the underlying graded $R$-module of $F$ is finite
and free. Let $e_{1},\dots,e_{n}$ be an ordered homogeneous basis
of $F_{+}$ as a graded $R$-module which is ordered in such a way
that if $i<j$, then $|e_{i}|\le|e_{j}|$. We denote by $R[\boldsymbol{e}]=R[e_{1},\dots,e_{n}]$
to be the free \emph{non-strict }graded-commutative $R$-algebra generated
by $e_{1},\dots,e_{n}$. In particular, if $e_{i}$ and $e_{j}$ are
distinct, then we have
\[
e_{i}e_{j}=(-1)^{|e_{i}||e_{j}|}e_{j}e_{i},
\]
in $R[\boldsymbol{e}]$, however elements of odd degree do not square
to zero in $R[\boldsymbol{e}]$. The reason we do not want elements
of odd degree to square to zero is because we will want to calculate
Gröbner bases in $K[\boldsymbol{e}]$, and the theory of Gröbner bases
for $K[\boldsymbol{e}]$ is much simpler when we do not have any zerodivisors.
In any case, one recovers the symmetric DG $R$-algebra of $F$ as
below:
\[
R[\boldsymbol{e}]\slash\langle\{e_{i}^{2}\mid|e_{i}|\text{ is odd}\}\rangle\simeq\mathrm{S}_{R}(F).
\]
Finally, equip $F$ with a multiplication $\mu$ giving it the structure
of an MDG algebra. Our goal is to compute the maximal associative
quotient of $F$ using the presentation given in Theorem~(\ref{theorempresentationsym})
as well as the theory of Gröbner bases in $K[\boldsymbol{e}]$.

\subsubsection{Monomials and Monomial Orderings}

Before we can do this, we first need to introduce some notation for
Gröbner basis applications in $K[\boldsymbol{e}]$. Our notation mostly
follows \cite{BE77} and \cite{Mot10} however we introduce some of
our own notation as well. A \textbf{monomial }in $K[\boldsymbol{e}]$
is an element of the form
\begin{equation}
e^{\boldsymbol{\alpha}}=e_{1}^{\alpha_{1}}\cdots e_{n}^{\alpha_{n}}\label{eq:monomial}
\end{equation}
where $\boldsymbol{\alpha}=(\alpha_{1},\dots,\alpha_{n})\in\mathbb{N}^{n}$
is called the \textbf{multidegree }of $e^{\boldsymbol{\alpha}}$ and
is denoted $\mathrm{multideg}(e^{\boldsymbol{\alpha}})=\boldsymbol{\alpha}.$
Similarly we define its \textbf{total degree}, denoted $\mathrm{deg}(e^{\boldsymbol{\alpha}})$,
and its \textbf{homological degree}, denoted $|e^{\boldsymbol{\alpha}}|$,
by
\[
\mathrm{deg}(e^{\boldsymbol{\alpha}})=\sum_{i=1}^{n}\alpha_{i}\quad\text{and}\quad|e^{\boldsymbol{\alpha}}|=\sum_{i=1}^{n}\alpha_{i}|e_{i}|.
\]
By convention we set $e^{\boldsymbol{0}}=1$ where $\boldsymbol{0}=(0,\dots,0)$
is the zero vector in $\mathbb{N}^{n}$. Note how the ordering in
(\ref{eq:monomial}) matters. In particular, if $i<j$ and both $|e_{i}|$
and $|e_{j}|$ are odd, then $e_{j}e_{i}$ is not a monomial in $K[\boldsymbol{e}]$
since it can be expressed as a non-trivial coefficient times a monomial:
\[
e_{j}e_{i}=-e_{i}e_{j}.
\]
On the other hand, if one of the $e_{i}$ or $e_{j}$ is even, then
$e_{j}e_{i}$ is a monomial in $K[\boldsymbol{e}]$ since $e_{j}e_{i}=e_{i}e_{j}$.
We equip $K[\boldsymbol{e}]$ with a weighted lexicographical ordering
$>$ with respect to the weighted vector\textbf{ $w=(|e_{1}|,\dots,|e_{n}|)$
}(the notation for this monomial ordering in Singular is Wp(w)). More
specifically, given two monomials $e^{\boldsymbol{\alpha}}$ and\textbf{
$e^{\boldsymbol{\beta}}$ }in $K[\boldsymbol{e}]$, we say $e^{\boldsymbol{\beta}}>e^{\boldsymbol{\alpha}}$
if either
\begin{enumerate}
\item $|e^{\boldsymbol{\beta}}|>|e^{\boldsymbol{\alpha}}|$ or;
\item $|e^{\boldsymbol{\beta}}|=|e^{\boldsymbol{\alpha}}|$ and $\beta_{1}>\alpha_{1}$
or;
\item $|e^{\boldsymbol{\beta}}|=|e^{\boldsymbol{\alpha}}|$ and there exists
$1<j\leq n$ such that $\beta_{j}>\alpha_{j}$ and $\beta_{i}=\alpha_{i}$
for all $1\leq i<j$. 
\end{enumerate}
Given a nonzero polynoimal $f\in K[\boldsymbol{e}]$, there exists
unique $c_{1},\dots,c_{m}\in K\backslash\{0\}$ and unique $\boldsymbol{\alpha}_{1},\dots,\boldsymbol{\alpha}_{m}\in\mathbb{N}^{n}$
where $\boldsymbol{\alpha}_{i}\neq\boldsymbol{\alpha}_{j}$ for all
$1\leq i<j\leq m$ such that
\begin{equation}
f=c_{1}e^{\boldsymbol{\alpha}_{1}}+\cdots+c_{m}e^{\boldsymbol{\alpha}_{m}}=\sum c_{i}e^{\boldsymbol{\alpha}_{i}}\label{eq:termldaklb-1}
\end{equation}
The $c_{i}e^{\boldsymbol{\alpha}_{i}}$ in (\ref{eq:termldaklb-1})
are called the \textbf{terms }of $f$ and the $e^{\boldsymbol{\alpha}_{i}}$
in (\ref{eq:termldaklb-1}) are called the \textbf{monomials }of $f$.
By reindexing the $\boldsymbol{\alpha}_{i}$ if necessary, we may
assume that $e^{\boldsymbol{\alpha}_{1}}>\cdots>e^{\boldsymbol{\alpha}_{m}}$.
In this case, we call $c_{1}e^{\boldsymbol{\alpha}_{1}}$ the \textbf{lead
term }of $f$, we call $e^{\boldsymbol{\alpha}_{1}}$ the \textbf{lead
monomial }of $f$, and we call $c_{1}$ the \textbf{lead coefficient
}of $f$. We denote these, respectively, by
\[
\mathrm{LT}(f)=c_{1}e^{\boldsymbol{\alpha}_{1}},\quad\mathrm{LM}(f)=e^{\boldsymbol{\alpha}_{1}},\quad\text{and}\quad\mathrm{LC}(f)=c_{1}.
\]
The \textbf{multidegree }of $f$ is defined to be the multidegree
of its lead monomial $e^{\boldsymbol{\alpha}_{1}}$ and is denoted
$\mathrm{multideg}(f)=\boldsymbol{\alpha}_{1}$. The \textbf{total
degree }of $f$ is defined to be the maximum of the total degrees
of its monomials and is denoted
\[
\mathrm{deg}(f)=\max_{1\leq i\leq m}\{\deg(e^{\boldsymbol{\alpha}_{i}})\}.
\]
We say $f$ is \textbf{homogeneous} of homological degree $i$ if
each of its monomials is homogeneous of homological degree $i$. In
this case, we say $f$ has \textbf{homological degree $i$ }and we
denote this by $|f|=i$. 

\begin{lemma}\label{lemma} For each $1\leq i\leq j\le n$, let $f_{ij}=e_{i}e_{j}-e_{i}\star e_{j}$.
We have
\[
\mathrm{LT}(f_{ij})=e_{i}e_{j}.
\]

\end{lemma}

\begin{proof}\label{proof} If $e_{i}\star e_{j}=0$, then this is
clear, otherwise let $e_{k}$ be a monomial of $e_{i}\star e_{j}$.
Since $\star$ respects homological degree, we have $|e_{k}|=|e_{i}|+|e_{j}|=|e_{i}e_{j}|$.
It follows that $|e_{k}|>\max\{|e_{i}|,|e_{j}|\}$ since $|e_{i}|,|e_{j}|\geq1.$
This implies $k>\max\{i,j\}$ by our assumption on the ordering of
$e_{1},\dots,e_{n}$. Therefore since $|e_{i}e_{j}|=|e_{k}|$ and
$k>\max\{i,j\}$, we see that $e_{i}e_{j}>e_{k}$. \end{proof}

\subsubsection{Gröbner Basis Calculations}

Our goal is to use the theory of Gröbner bases to help us calculate
\[
F^{\mathrm{as}}=\mathrm{S}_{R}(F)\slash\mathfrak{s}(\mu)\simeq R[\boldsymbol{e}]\slash\langle\{f_{ij}\}\rangle,
\]
where $f_{ij}\in R[\boldsymbol{e}]$ are defined by
\[
f_{ij}=e_{i}e_{j}-e_{i}\star e_{j}=e_{i}e_{j}-\sum_{k}c_{ij}^{k}e_{k},
\]
where the $c_{ij}^{k}\in R$ are the entries of the matrix representation
of $\mu$ with respect to the ordered homogeneous basis $e_{1},\dots,e_{n}$.
In order to do this, we work over $K$ instead of $R$ sinc that is
where the theory of Gröbner bases works best. Thus we wish to calculate:
\[
F_{K}^{\mathrm{as}}:=F^{\mathrm{as}}\otimes_{R}K\simeq K[\boldsymbol{e}]\slash\langle\{f_{ij}\}\rangle.
\]
To this end, let $\mathcal{F}=\{f_{ij}\mid1\leq i,j\leq n\}$ and
let $\mathfrak{a}$ be the $K[\boldsymbol{e}]$-ideal generated by
$\mathcal{F}$. We wish to construct a left Gröbner basis for $\mathfrak{a}$
(which will turn out to be a two-sided Gröbner basis) via Buchberger's
algorithm using the monomial ordering described above. Suppose $f,g$
are two nonzero polynomials in $K[\boldsymbol{e}]$ with $\mathrm{LT}(f)=ce^{\boldsymbol{\alpha}}$
and $\mathrm{LT}(g)=de^{\boldsymbol{\beta}}$. Set $\boldsymbol{\gamma}=\mathrm{lcm}(\boldsymbol{\alpha},\boldsymbol{\beta})$
and define the left $\mathrm{S}$\textbf{-polynomial }of $f$ and
$g$ to be
\begin{align}
\mathrm{S}(f,g) & =e^{\boldsymbol{\gamma}-\boldsymbol{\alpha}}f\pm(c/d)e^{\boldsymbol{\gamma}-\boldsymbol{\beta}}g\label{eq:spolydef}
\end{align}
where the $\pm$ in (\ref{eq:spolydef}) is chosen to be $+$ or $-$
depending on which sign will cancel out the lead terms. We begin Buchberger's
algorithm by calculating the $\mathrm{S}$-polynomials of all pairs
of polynomials in $\mathcal{F}$. In other words, we calculate all
$\mathrm{S}$-polynomials of the form $\mathrm{S}(f_{kl},f_{ij})$
where $1\leq i,j,k,l\leq n$. Note that if $k>l$, then $f_{lk}=(-1)^{|e_{k}||e_{l}|}f_{kl}$
implies
\[
\mathrm{S}(f_{lk},f_{ij})=(-1)^{|e_{k}||e_{l}|}\mathrm{S}(f_{kl},f_{ij})=\pm\mathrm{S}(f_{ij},f_{lk}),
\]
where the last equality follows from the fact that the lead coefficient
of $f_{ij}$ and $f_{lk}$ is $\pm1$. Thus we may assume that $j\geq i$
and $l\geq k\geq i$. Obviously we have $\mathrm{S}(f_{ij},f_{ij})=0$
for each $i,j$, however something interesting happens when we calculate
the $\mathrm{S}$-polynomial of $f_{jk}$ and $f_{ij}$ where $j>i$
and then divide this by $\mathcal{F}$ (where division by $\mathcal{F}$
means taking the left normal form of $\mathrm{S}(f_{jk},f_{ij})$
with respect to $\mathcal{F}$ using the left normal form described
in \cite{GP02}). In particular, we obtain the associator $[e_{i},e_{j},e_{k}]$!
Indeed, we have
\begin{align*}
\mathrm{S}(f_{jk},f_{ij}) & =e_{i}(e_{j}e_{k}-e_{j}\star e_{k})-(e_{i}e_{j}-e_{i}\star e_{j})e_{k}\\
 & =(e_{i}\star e_{j})e_{k}-e_{i}(e_{j}\star e_{k})\\
 & =\sum_{l}c_{ij}^{l}e_{l}e_{k}-\sum_{l}c_{jk}^{l}e_{i}e_{l}\\
 & \to\sum_{l}c_{ij}^{l}e_{l}\star e_{k}-\sum_{l}c_{jk}^{l}e_{i}\star e_{l}\\
 & =(e_{i}\star e_{j})\star e_{k}-e_{i}\star(e_{j}\star e_{k})\\
 & =[e_{i},e_{j},e_{k}],
\end{align*}
where in the fourth line we did division by $\mathcal{F}$ (note that
if $[e_{i},e_{j},e_{k}]\neq0$, then $\mathrm{deg}([e_{i},e_{j},e_{k}])=1$,
so we cannot divide this anymore by $\mathcal{F}$). Next suppose
that $j>i$, $l>k$, and $j\neq k$. Then we have
\begin{align*}
\mathrm{S}(f_{kl},f_{ij}) & =e_{i}e_{j}f_{kl}-f_{ij}e_{k}e_{l}\\
 & =(e_{i}\star e_{j})e_{k}e_{l}-e_{i}e_{j}(e_{k}\star e_{l})\\
 & \to(e_{i}\star e_{j})\star(e_{k}\star e_{l})-(e_{i}\star e_{l})\star(e_{k}\star e_{l})\\
 & =0
\end{align*}
where in the third line we did division by $\mathcal{F}$. Next, suppose
that
\[
f=ce_{k}+c'e_{k'}+\cdots+c''e_{k''}\in\langle F\rangle
\]
where $c,c',c''\in R$ with $c\neq0$ and where $\mathrm{LM}(f)=e_{k}$.
Then we have
\begin{align*}
\mathrm{S}(f,f_{jk}) & =e_{j}f-cf_{jk}\\
 & =c'e_{j}e_{k'}+\cdots+c''e_{j}e_{k''}+ce_{j}\star e_{k}\\
 & \to c'e_{j}\star e_{k'}+\cdots+c''e_{j}\star e_{k''}+ce_{j}\star e_{k}\\
 & =e_{j}\star(ce_{k}+c'e_{k'}+\cdots+c''e_{k''})\\
 & =e_{j}\star f\\
 & \in\langle F\rangle
\end{align*}
where in the third line we did division by $\mathcal{F}$. Similarly,
if $i\neq k\neq j$, then we have
\begin{align*}
\mathrm{S}(f,f_{ij}) & =e_{i}e_{j}f-cf_{ij}e_{k}\\
 & =c'(e_{i}e_{j})e_{k'}+\cdots+c''(e_{i}e_{j})e_{k''}+c(e_{i}\star e_{j})e_{k}\\
 & \to c'(e_{i}\star e_{j})\star e_{k'}+\cdots+c''(e_{i}\star e_{j})\star e_{k''}+c(e_{i}\star e_{j})\star e_{k}\\
 & =(e_{i}\star e_{j})\star(ce_{k}+c'e_{k'}+\cdots+c''e_{k''})\\
 & =(e_{i}\star e_{j})\star f\\
 & \in\langle F\rangle.
\end{align*}
where in the third line we did division by $\mathcal{F}$. Finally
suppose that
\[
g=de_{m}+d'e_{m'}+\cdots+d''e_{m''}\in\langle F\rangle
\]
where $d,d',d''\in R$ with $d\neq0$ and where $\mathrm{LM}(g)=e_{m}$.
If $k=m$, then we have
\[
d\mathrm{S}(f,g)=cf-dg\in\langle F\rangle.
\]
On the other hand, if $k\neq m$, then we have
\begin{align*}
d\mathrm{S}(f,g) & =de_{m}f-cge_{k}\\
 & =dc'e_{m}e_{k'}+\cdots+dc''e_{m}e_{k''}-cd'e_{m'}e_{k}-\cdots-cd''e_{m''}e_{k}\\
 & \to dc'e_{m}\star e_{k'}+\cdots+dc''e_{m}\star e_{k''}-cd'e_{m'}\star e_{k}-\cdots-cd''e_{m''}\star e_{k}\\
 & =de_{m}\star(c'e_{k'}+\cdots+c''e_{k''})-c(d'e_{m'}+\cdots+d''e_{m''})\star e_{k}\\
 & =de_{m}\star(f-ce_{k})-c(g-de_{m})\star e_{k}\\
 & =de_{m}\star f+cg\star e_{k}-dce_{m}\star e_{k}+cde_{m}\star e_{k}\\
 & =de_{m}\star f+cg\star e_{k}\\
 & \in\langle F\rangle.
\end{align*}
It follows that we can construct a Gröbner basis
\[
\mathcal{G}:=\mathcal{F}\cup\{g_{1},\dots,g_{m}\}
\]
of $\mathfrak{a}$ such that the $g_{i}$ all belong to $\langle F\rangle$.

\begin{example}\label{example455} Let $R=\Bbbk[x,y,z,u,v]$, let
$\boldsymbol{m}=zv,yv,uv,xv,xu,yzu$, and let $F$ be the minimal
free resolution of $R\slash\boldsymbol{m}$ over $R$. Then $F$ can
be realized as the $R$-complex supported on the $\boldsymbol{m}$-labeled
cellular complex pictured below: \begin{center}
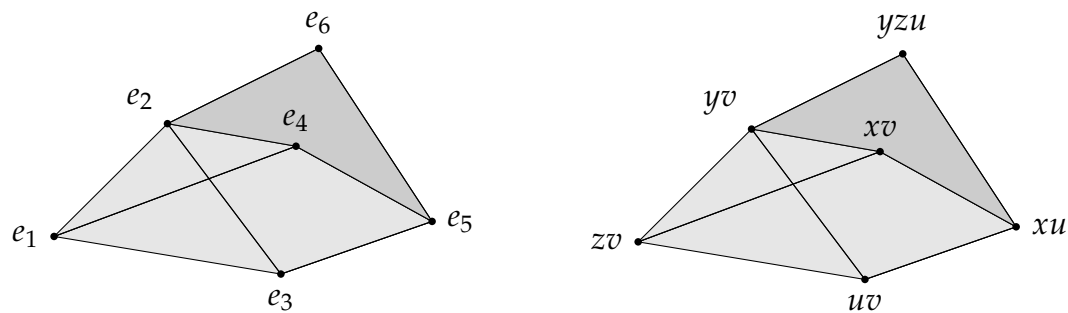
\begin{tabular}{cccc}
\begin{tikzpicture}[scale=1]

\draw[fill=gray!20] (0,0) -- (3,-0.5) -- (5,0.2) -- (3.2,1.2)-- (0,0); 
\draw[fill=gray!20] (0,0) -- (1.5,1.5) -- (3.2,1.2)-- (0,0); 
\draw[] (1.5,1.5) -- (3,-0.5); \draw[fill=gray!20] (3,-0.5) -- (5,0.2) -- (3.2,1.2);
\draw[color=black!100] (1.5,1.5) -- (3,-0.5) -- (5,0.2);
\draw[color=black!100] (0,0) -- (3.2,1.2) -- (5,0.2);
\draw[fill=gray!40] (1.5,1.5) -- (3.2,1.2) -- (5,0.2)-- (3.5,2.5) -- (1.5,1.5); 

\node[circle, fill=black, inner sep=1pt, label=left:$e _1 $] (a) at (0,0) {};
\node[circle, fill=black, inner sep=1pt, label=above left:$e _2 $] (b) at (1.5,1.5) {};
\node[circle, fill=black, inner sep=1pt, label=below:$e _3 $] (c) at (3,-0.5) {};
\node[circle, fill=black, inner sep=1pt, label=above:$e _4 $] (d) at (3.2,1.2) {}; 
\node[circle, fill=black, inner sep=1pt, label=right:$e _5 $] (e) at (5,0.2) {};
\node[circle, fill=black, inner sep=1pt, label=above:$e_6 $] (f) at (3.5,2.5) {};

\draw[] (b) -- (f);
\draw (e) -- (f);

\end{tikzpicture} &  &  & \begin{tikzpicture}[scale=1]

\draw[fill=gray!20] (0,0) -- (3,-0.5) -- (5,0.2) -- (3.2,1.2)-- (0,0); 
\draw[fill=gray!20] (0,0) -- (1.5,1.5) -- (3.2,1.2)-- (0,0); 
\draw[] (1.5,1.5) -- (3,-0.5); \draw[fill=gray!20] (3,-0.5) -- (5,0.2) -- (3.2,1.2);
\draw[color=black!100] (1.5,1.5) -- (3,-0.5) -- (5,0.2);
\draw[color=black!100] (0,0) -- (3.2,1.2) -- (5,0.2);
\draw[fill=gray!40] (1.5,1.5) -- (3.2,1.2) -- (5,0.2)-- (3.5,2.5) -- (1.5,1.5); 

\node[circle, fill=black, inner sep=1pt, label=left:$zv $] (a) at (0,0) {};
\node[circle, fill=black, inner sep=1pt, label=above left:$yv $] (b) at (1.5,1.5) {};
\node[circle, fill=black, inner sep=1pt, label=below:$uv $] (c) at (3,-0.5) {};
\node[circle, fill=black, inner sep=1pt, label=above:$xv $] (d) at (3.2,1.2) {}; 
\node[circle, fill=black, inner sep=1pt, label=right:$xu $] (e) at (5,0.2) {};
\node[circle, fill=black, inner sep=1pt, label=above:$yzu $] (f) at (3.5,2.5) {};

\draw[] (b) -- (f);
\draw (e) -- (f);

\end{tikzpicture}\tabularnewline
\end{tabular}
\par\end{center}We write down the homogeneous components of $F$ as a graded module
below:
\begin{align*}
F_{0} & =R\\
F_{1} & =Re_{1}+Re_{2}+Re_{3}+Re_{4}+Re_{5}+Re_{6}\\
F_{2} & =Re_{12}+Re_{13}+Re_{14}+Re_{23}+Re_{24}+Re_{26}+Re_{35}+Re_{45}+Re_{56}\\
F_{3} & =Re_{123}+Re_{124}+Re_{1345}+Re_{2345}+Re_{2456}\\
F_{4} & =Re_{12345}
\end{align*}
We will use Singular to help us find an associative multigraded multiplication
$\mu$ on $F$ such that $e_{\sigma}^{2}=0$ for all $\sigma$. From
multidegree and Leibniz rule considerations, we begin constructing
$\mu$ as follows:
\begin{align*}
e_{1}\star e_{2} & =ve_{12} & e_{3}\star e_{5} & =ue_{35}\\
e_{1}\star e_{3} & =ve_{13} & e_{3}\star e_{6} & =-zue_{23}+ue_{26}\\
e_{1}\star e_{4} & =ve_{14} & e_{4}\star e_{5} & =xe_{45}\\
e_{1}\star e_{5} & =ue_{14}+ze_{45} & e_{4}\star e_{6} & =-zue_{24}+xe_{26}\\
e_{1}\star e_{6} & =zue_{12}+ze_{26} & e_{5}\star e_{6} & =ue_{56}\\
e_{2}\star e_{3} & =ve_{23} & e_{1}\star e_{23} & =ve_{123}\\
e_{2}\star e_{4} & =ve_{24} & e_{1}\star e_{24} & =ve_{124}\\
e_{2}\star e_{5} & =ue_{24}+ye_{45} & e_{1}\star e_{35} & =-ve_{1345}\\
e_{2}\star e_{6} & =ye_{26} & e_{1}\star e_{56} & =-uze_{124}+ze_{2456}\\
e_{3}\star e_{4} & =ve_{35}-ve_{45} & e_{1}\star e_{2345} & =ve_{12345}.
\end{align*}
At this point, Singular can help us determine how we should define
$\mu$ everywhere else. First we input the following code into Singular:\begin{center}
\begin{tabular}{c}
\begin{lstlisting}
LIB "ncalg.lib";

intvec V = 1:6, 2:9, 3:5, 4:1;
 
ring A=(0,x,y,z,u,v),(e1,e2,e3,e4,e5,e6,
e12,e13,e14,e23,e24,e26,e35,e45,e56,
e123,e124,e1345,e2345,e2456,e12345),Wp(V);

matrix C[21][21]; matrix D[21][21]; int i; int j;  
for (i=1; i<=21; i++) {for (j=1; j<=21; j++) {C[i,j]=(-1)^(V[i]*V[j]);}}  
ncalgebra(C,D);  

poly f(1)(2) = e1*e2 - v*e12; 
poly f(1)(3) = e1*e3 - v*e13; 
poly f(1)(4) = e1*e4 - v*e14; 
poly f(1)(5) = e1*e5 - u*e14 - z*e45; 
poly f(1)(6) = e1*e6 - zu*e12 - z*e26; 
poly f(2)(3) = e2*e3 - v*e23; 
poly f(2)(4) = e2*e4 - v*e24; 
poly f(2)(5) = e2*e5 - u*e24 - y*e45; 
poly f(2)(6) = e2*e6 - y*e26; 
poly f(3)(4) = e3*e4 - v*e35 + v*e45;   
poly f(3)(5) = e3*e5 - u*e35;   
poly f(3)(6) = e3*e6 + zu*e23 - u*e26; 
poly f(4)(5) = e4*e5 - x*e45;   
poly f(4)(6) = e4*e6 + zu*e24 - x*e26; 
poly f(5)(6) = e5*e6 - u*e56;     
poly f(1)(23) = e1*e23 - v*e123;   
poly f(1)(24) = e1*e24 - v*e124;  
poly f(1)(35) = e1*e35 + v*e1345;  
poly f(1)(56) = e1*e56 + uz*e124 - z*e2456; 
poly f(1)(2345) = e1*e2345 - v*e12345; 

list L =  (e1,e2,e3,e4,e5,e6,
e12,e13,e14,e23,e24,e26,e35,e45,e56,
e123,e124,e1345,e2345,e2456,e12345);  

ideal I; int i; for (i=1; i<=21; i++) {I = I + L[i]*L[i];} 

I = I + f(1)(2),f(1)(3),f(1)(4),f(1)(5),f(1)(6),f(2)(3),f(2)(4),
f(2)(5),f(2)(6),f(3)(4),f(3)(5),f(3)(6),f(4)(5),f(4)(6),
f(5)(6),f(1)(23),f(1)(24),f(1)(35),f(1)(56),f(1)(2345); 
\end{lstlisting}
\tabularnewline
\end{tabular}
\par\end{center} To see that the multiplication is associative thus far, we calculate
the Gröbner basis of $I$ with respect to our fixed monomial ordering
using the command \lstinline!std(I)! in Singular. Singular gives
us the following output:\begin{center}
\begin{tabular}{c}
\begin{lstlisting}
_[1]=e6^2 
_[2]=e5*e6+(-u)*e56 
_[3]=e5^2 
...
_[57]=e2*e56+(-y)*e2456
_[58]=e2*e45
_[59]=(z*u)*e2*e35+(-v)*e6*e35+(u*v)*e2456
_[60]=e2*e26
...
_[209]=e124*e12345 
_[210]=e123*e12345 
_[211]=e12345^2
\end{lstlisting}
\tabularnewline
\end{tabular}
\par\end{center}where we omitted most of the Gröbner basis elements due to size constraints.
Since the lead term of each polynomial showing up in the list has
total degree $>1$, we conclude that the multiplication we have defined
so far is associative. Now observe that if we want the multiplication
to continue being associative, then we need to define $e_{2}\star e_{26}=0$
since
\begin{align*}
ye_{2}\star e_{26} & =e_{2}\star(e_{2}\star e_{6})\\
 & =(e_{2}\star e_{2})\star e_{6}-[e_{2},e_{2},e_{6}]\\
 & =-[e_{2},e_{2},e_{6}].
\end{align*}
In fact, Singular already tells us this since it is computing the
maximal associative quotient! In particular, setting \lstinline!I = std(I)!
and running the command \lstinline!reduce(e2*e26 , I)! outputs \lstinline!0!
in Singular which tells us that in the maximal associative quotient
we have $e_{1}\star e_{12}=0$. Alternatively, we could simply read
this off the list of polynomials that Singular outputted as the polynomial
$e_{2}\star e_{26}$ shows up in the Gröbner basis. Similarly, Singular
tells us that we should define $e_{2}\star e_{56}=-ye_{2456}$ since
the polynomial $e_{2}\star e_{56}-ye_{2456}$ shows up in the Gröbner
basis. On the other hand, if we run the command \lstinline!reduce(e6*e35 , I)!,
then Singular outputs \lstinline!e6*e35! which tells us that we still
need to define $e_{6}\star e_{35}$. Upon reflection of the multigrading
and Leibniz rule, we define
\[
e_{6}\star e_{35}=-zue_{2345}+ue_{2456}.
\]
Thus we add the polynomial \lstinline!poly f(6)(35) = e6*e35 + zu*e2345 - y*e2456!
to our ideal in the code. We observe that our multiplication is still
associative by running the command \lstinline!std(I)! and checking
that none of the polynomials listed has lead term of total degree
$1$ again. Furthermore, running the command\begin{center}
\begin{tabular}{c}
\begin{lstlisting}
for(i=1;i<=21;i++){for(j=i+1;j<=21;j++){reduce(L[i]*L[j],I);};};
\end{lstlisting}
\tabularnewline
\end{tabular}
\par\end{center}shows that the multiplication is now defined everywhere. For instance,
the command \lstinline!reduce(e12*e35 , I)! outputs \lstinline!(-v)*e12345!.
This tells us that $e_{12}\star e_{35}=-ve_{12345}$. \end{example}

\begin{example}\label{example} In Example~(\ref{example1}) we calculate
the associator $[e_{1},e_{5},e_{2}]$ using the following Singular
code:\begin{center}
\begin{tabular}{c}
\begin{lstlisting}
LIB "ncalg.lib"; 

intvec V = 1:3, 2:5, 3:5; 

ring A=(0,x,y,z,w),(e1,e2,e5,
e12,e14,e23,e35,e45,
e123,e124,e134,e234,e345),Wp(V);

matrix C[13][13]; matrix D[13][13]; int i; int j; 
for (i=1; i<=13; i++) {for (j=1; j<=13; j++) {C[i,j]=(-1)^(V[i]*V[j]);}} 
ncalgebra(C,D);

poly f(1)(2) = e1*e2-e12; 
poly f(1)(5) = e1*e5-yz2*e14-x*e45;  
poly f(2)(5) = e2*e5-y2z*e23-w*e35;  
poly f(2)(45) = e2*e45+yz*e234-w*e345;  
poly f(1)(35) = e1*e35-yz*e134+x*e345;  
poly f(1)(23) = e1*e23-e123;  
poly f(2)(14) = e2*e14+e124;
poly S(1)(5)(2) = f(1)(5)*e2+e1*f(2)(5);  

ideal I = f(1)(2),f(1)(5),f(2)(5),f(1)(23),f(1)(35),f(2)(14),f(2)(45);
reduce(S(1)(5)(2),I); 

// [e1,e5,e2] = (y^2*z)*e123-(y*z^2)*e124+(y*z*w)*e134-(x*y*z)*e234
\end{lstlisting}
\tabularnewline
\end{tabular}
\par\end{center} \end{example}

\section*{Acknowledgements}

First and foremost, I would like to express my deepest gratitude to
Keri Sather-Wagstaff, my PhD advisor, for her invaluable guidance,
continuous support, and the confidence she placed in me. Her regular
Zoom meetings, insightful discussions, and unwavering dedication have
been instrumental in shaping my research. Her role in my academic
journey cannot be overstated, and I am immensely thankful for her
mentorship. I am also deeply thankful to Dr. Saeed Nasseh, my advisor
during my Master's program. His initial recommendation was a pivotal
point in my academic career, leading me to pursue my PhD under Dr.
Sather-Wagstaff's guidance. His ongoing support and insightful feedback
on my paper have been incredibly helpful, especially his suggestions
for edits and improvements for peer review. Lastly, I extend my gratitude
to Keller VandeBogert. His thorough review and constructive feedback
on my paper were exceptionally beneficial. This journey would not
have been possible without the support and encouragement from each
of these individuals. I am deeply appreciative of their contributions
to my academic and personal growth.

\newpage{}


\begin{thebibliography}{AHH97}
\bibitem[AHH97]{AHH97} A. Aramova, J. Herzog, T. Hibi. \emph{Gotzmann
theorems for exterior algebras and combinatorics}. Journal of Algebra,
Vol. 191, No. 1 (1997), https://doi.org/10.1006/jabr.1996.6903.No.
3 (1977), pp. 447-485.

\bibitem[Avr81]{Avr81}L. L. Avramov. \emph{Obstructions to the existence
of multiplicative structures on minimal free resolutions}. American
Journal of Mathematics, Vol. 103, No. 1 (1981), pp. 1--31.

\bibitem[BE77]{BE77}D. A. Buchsbaum and D. Eisenbud. \emph{Algebra
structures for finite free resolutions, and some structure theorems
for ideals of codimension 3}. American Journal of Mathematics, Vo.
99,

\bibitem[BPS98]{BPS98} D. Bayer, I. Peeva, and B. Sturmfels. \emph{Monomial
resolutions}. Math Research Letters, 5.1-2 (1998), pp. 31--46.

\bibitem[BS98]{BS98} D. Bayer and B. Sturmfels. \emph{Cellular resolutions
of monomial modules}. Journal fur die Reine und Angewandte Mathematik,
502 (2000), https://doi.org/10.1515/crll.1998.083.

\bibitem[CE91]{CE91} H. Charalambous and E. Graham Evans. \emph{A
deformation theory approach to betti numbers of finite length modules}.
Journal of Algebra, Vol. 143, No. 1 (1991), pp. 246-251

\bibitem[Eis95]{Eis95} D. Eisenbud. (1995) \emph{Commutative Algebra:
With a View Toward Algebraic Geometry. }New York: Springer-Verlag\emph{.}

\bibitem[Erm10]{Erm10}D. Erman. \emph{A special case of the Buchsbaum-Eisenbud-Horrocks
rank conjecture}. Math Research Letters, Vol 17, No. 06 (2010), pp.
1079-1089.

\bibitem[FST08]{FST08} A. J. Frankild, S. Sather-Wagstaff, and A.
Taylor. \emph{Second symmetric powers of chain complexes}. Bulletin
of the Iranian Mathematical Society, Vol. 37, No. 3 (2011), pp. 57.

\bibitem[GP02]{GP02} Greuel, G.-M., and Pfister, G. (2008). \emph{A
Singular Introduction to Commutative Algebra} (2nd ed.). Berlin: Springer-Verlag.

\bibitem[IRU07]{IRU07} C. Ionescu, G. Restuccia, R. Utano.\emph{
Fitting conditions of symmetric algebras of modules of finite projective
dimension}. Bollettino dell\textquoteright Unione Matematica Italiana,
Vol. 10-B, No. 3 (2007), pp. 681-696.

\bibitem[Kat19]{Kat19} L. Katthän. \emph{The structure of DGA resolutions
of monomial ideals}. Journal of Pure and Applied Algebra, Vol. 223,
No. 3 (2019), pp. 1227-1245.

\bibitem[Mot10]{Mot10} O. Motsak. \emph{Graded commutative algebra
and related structures in SINGULAR with applicaitons}. 2020. PhD thesis,
Technische Universität Kaiserslautern.

\bibitem[MS05]{MS05} E. Miller and B. Sturmfels\emph{. Combinatorial
commutative algebra, }Graduate Texts in Mathematics, Vol. 227, Springer-Verlag,
New York, 2005.

\bibitem[Sch08]{Sch08} Schafer, R. D. (1966). \emph{An Introduction
to Nonassociative Algebras}. New York and London: Academic Press.

\bibitem[Sri92]{Sri92}H. Srinivasan. \emph{The non-existence of a
minimal algebra resolution despite the vanishing of Avramov obstructions}.
Journal of Algebra, Vol 146 (1992), pp. 251-266.

\bibitem[Tch95]{Tch95} A. B. Tchernev\emph{. Acyclicity of symmetric
and exterior powers of complexes}. Journal of Algebra, Vol. 184, No.
3 (1996), pp. 1115--1135.

\bibitem[Van22]{Van22} K. Vandebogert. \emph{Vanishing of Avramov
obstructions for products of sequentially transverse ideals}. Journal
of Pure and Applied Algebra, Vol. 226, No. 11 (2022), https://doi.org/107111.

\bibitem[VW23]{VW23} K. Vandebogert, M. Walker. \emph{The total rank
conjecture in characteristic two}. arxiv-2305.09771.

\bibitem[Wal17]{Wal17} M. Walker. \emph{Total Betti numbers of modules
of finite projective dimension}. Annals of Mathematics, Vol. 186,
No. 2 (2017), pp. 641-646.
\end{thebibliography}
\end{document}